\documentclass[12pt]{amsart}

\usepackage{amsfonts,amsthm,amsmath,amssymb,amscd,mathrsfs,tikz}
\usepackage{graphics}
\usepackage{indentfirst}
\usepackage{cite}
\usepackage{latexsym}
\usepackage[dvips]{epsfig}

\usepackage{color}

\newtheorem{theorem}{Theorem}[section]
\newtheorem{remark}{Remark}[section]
\newtheorem{definition}{Definition}[section]
\newtheorem{lemma}[theorem]{Lemma}
\newtheorem{pro}[theorem]{Proposition}
\newtheorem{cor}[theorem]{Corollary}

\renewcommand{\div}{{\rm div }}

\newcommand{\bt}{\begin{theorem}}
\newcommand{\bl}{\begin{lemma}}
\newcommand{\el}{\end{lemma}}
\newcommand{\et}{\end{theorem}}

\newcommand{\bn}{\begin{eqnarray}}
\newcommand{\en}{\end{eqnarray}}
\newcommand{\bnn}{\begin{eqnarray*}}
\newcommand{\enn}{\end{eqnarray*}}

\newcommand{\ba}{\begin{aligned}}
\newcommand{\ea}{\end{aligned}}
\newcommand{\be}{\begin{equation}}
\newcommand{\ee}{\end{equation}}

\newcommand{\R}{\mathbb{R}}

\newcommand{\Bx}{{\boldsymbol{x}}}
\newcommand{\Bv}{{\boldsymbol{v}}}
\newcommand{\oBv}{\overline{{\boldsymbol{v}}}}

\newcommand{\Bn}{{\boldsymbol{n}}}
\newcommand{\Bt}{{\boldsymbol{\tau}}}
\newcommand{\Bu}{{\boldsymbol{u}}}
\newcommand{\oBu}{\overline{{\boldsymbol{u}}}}
\newcommand{\Be}{{\boldsymbol{e}}}

\newcommand{\Bh}{{\boldsymbol{h}}}

\newcommand{\BU}{\boldsymbol{U}}

\newcommand{\Bg}{{\boldsymbol{g}}}
\newcommand{\Bw}{{\boldsymbol{w}}}
\newcommand{\BD}{{\boldsymbol{D}}}
\newcommand{\Ba}{{\boldsymbol{a}}}

\newcommand{\Bp}{{\boldsymbol{\phi}}}
\newcommand{\mcH}{\mathcal{H}}

\newcommand{\mfc}{\mathfrak{c}}

\begin{document}

\title[Two-dimensional flows with Navier slip boundary condition]
{On the Steady Navier-Stokes system with Navier slip boundary conditions in two-dimensional channels}

\author{Kaijian Sha}
\address{School of mathematical Sciences, Shanghai Jiao Tong University, 800 Dongchuan Road, Shanghai, China}
\email{kjsha11@sjtu.edu.cn}

\author{Yun Wang}
\address{School of Mathematical Sciences, Center for dynamical systems and differential equations, Soochow University, Suzhou, China}
\email{ywang3@suda.edu.cn}

\author{Chunjing Xie}
\address{School of mathematical Sciences, Institute of Natural Sciences,
Ministry of Education Key Laboratory of Scientific and Engineering Computing,
and CMA-Shanghai, Shanghai Jiao Tong University, 800 Dongchuan Road, Shanghai, China}
\email{cjxie@sjtu.edu.cn}

\begin{abstract}
In this paper, we investigate the incompressible steady Navier-Stokes system with Navier slip boundary condition in a two-dimensional channel. As long as the width of cross-section of the channel grows more slowly than the linear growth, the existence of solutions with arbitrary flux is established. Furthermore, if the flux is suitably small, the solution is unique even when the width of the channel is unbounded, and  approaches to the shear flows at far field where the channels tend to be straight at far fields. One of the major difficulties for the analysis on flows with Navier boundary conditions is that the tangential velocity may not be zero on the boundary so that we have to study the behavior of solutions near the boundary carefully.
The crucial ingredients of analysis include the construction of an appropriate flux carrier, and the detailed analysis for the flow behavior near boundary via combining a Hardy type inequality for normal component of velocity and the divergence free property of the velocity.
\end{abstract}


\thanks{Updated on \today}

\maketitle

\section{Introduction}
It is not only mathematically interesting but also physically important  to study the well-posedness of the steady Navier-Stokes system
\begin{equation}\label{NS}
\left\{
\begin{aligned}
&-\Delta \Bu+\Bu\cdot \nabla \Bu +\nabla p=0 ~~~~&\text{ in }\Omega,\\
&{\rm div}~\Bu=0&\text{ in }\Omega,
\end{aligned}\right.
\end{equation}
in a channel type domain $\Omega$, where the unknown function $\Bu=(u_1,u_2)$ is the velocity and $p$ is the pressure. Here the domain $\Omega$ is assumed to be a two-dimensional channel,
\begin{equation}\label{defOmega}
\Omega=\{(x_1,x_2):x_1\in \mathbb{R},~f_1(x_1)<x_2<f_2(x_1)\},
\end{equation}
where $f_1$ and $f_2$ are smooth functions.
\begin{figure}[h]
	\centering

	\tikzset{every picture/.style={line width=0.75pt}} 

	\begin{tikzpicture}[x=0.75pt,y=0.75pt,yscale=-1,xscale=1]
	
	\draw    (118.6,84.5) .. controls (130.6,88.5) and (160.6,112.5) .. (197.8,121) .. controls (235,129.5) and (234.12,105.1) .. (268.8,105) .. controls (303.48,104.9) and (307.37,144.08) .. (337.6,145.5) .. controls (367.83,146.92) and (447.45,98.71) .. (472.8,92) ;
	\draw    (126.6,196.5) .. controls (152.46,185.3) and (171.99,167.36) .. (214.6,174.5) .. controls (257.21,181.64) and (256.17,206.04) .. (286.6,213.5) .. controls (317.03,220.96) and (359.6,186.5) .. (388.95,186.06) .. controls (418.31,185.62) and (439.6,189.5) .. (484.6,210.5) ;
	\draw  [dash pattern={on 4.5pt off 4.5pt}]  (310,131.2) -- (310,212.2) ;
	
	\draw (300,233.4) node [anchor=north west][inner sep=0.75pt]    {$\Omega $};
	\draw (129,64.4) node [anchor=north west][inner sep=0.75pt]    {$f_{2}( x_{1})$};
	\draw (129,207.4) node [anchor=north west][inner sep=0.75pt]    {$f_{1}( x_{1})$};
	\draw (314,161.4) node [anchor=north west][inner sep=0.75pt]    {$\Sigma ( x_{1})$};

	\end{tikzpicture}

	\caption{The channel $\Omega$}
\end{figure}
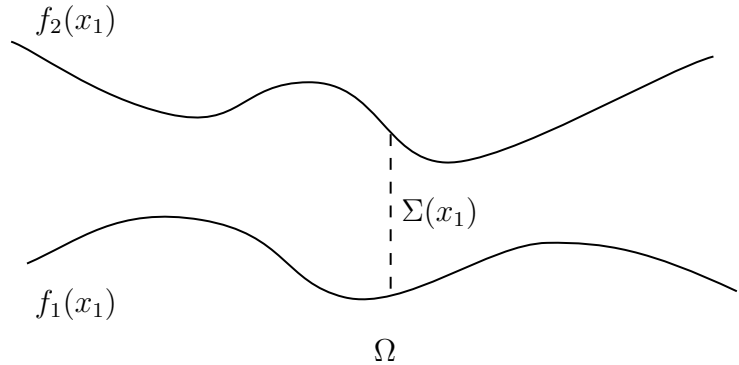

 If, in addition, the outlets of the channel tend to be straight, in 1950s, Leray proposed a  problem to study the solutions of \eqref{NS} with no-slip boundary condition, which tend to Poiseuille flows at far field.  More precisely, suppose that there exist constants $c_i^\pm (i=1,2)$ and $K>0$ such that $f_i(x_1)=c_i^-$ and $f_i(x_1)=c_i^+(i=1,2)$ for any $x_1<-K$ and $x_1>K$, respectively. Without loss of generality, we assume that $c_i^+=c_i^-=c_i$. Hence one looks for the solutions of the Navier-Stokes system \eqref{NS} supplemented with no-slip boundary conditions so that 
\begin{equation}\label{far field}
	\Bu \to \BU \text{ as }|x_1|\to \infty,
\end{equation}
where  $\BU=U(x_2)\Be_1$ is the shear flow solution of \eqref{NS} with flux $\Phi$ in the corresponding straight channel
\begin{equation*}
{\Omega}^\sharp=\{(x_1,x_2):~x_1\in \R,~x_2\in (c_1,c_2)\},
\end{equation*}
satisfying no-slip boundary condition. The far field behavior \eqref{far field} implies automatically the flux constraint
\begin{equation}\label{flux constraint}
	\int_{\Sigma(x_1)}\Bu\cdot \Bn\,ds=\Phi \text{ for any }x_1\in \R,
\end{equation}
where $\Phi$ is a constant and $\Sigma(x_1)=\{(x_1,x_2): (x_1,x_2)\in \Omega\}$.
This problem is  called Leray problem nowadays. Without loss of generality, we always assume 
that $\Phi$ is nonnegative  in this paper.

The major breakthrough for the Leray problem in infinitely long channels was made by Amick \cite{A1,A2,AF}, Ladyzhenskaya and Solonnikov \cite{LS}. It was proved in \cite{A1, LS} that Leray problem is solvable as long as the flux is small. Actually, the existence of solutions
of \eqref{NS} in a nozzle
with arbitrary flux  was also proved in \cite{LS}. However, the far field behavior and the uniqueness of such solutions are not clear when the flux is large. In order to completely resolve Leray problem, one needs only to prove the far field convergence of solutions obtained in \cite{LS}. The far field behavior of solutions was studied in \cite{AF,Ga} and references therein. To the best of our knowledge, there is no result on the far field behavior of solutions of steady Navier-Stokes system with large flux except for the axisymmetric solutions in a pipe studied in \cite{WX2}.

For viscous flows near solid boundary, besides the no-slip boundary condition,
 the Navier boundary conditions
\begin{equation}\label{BC}
	\Bu\cdot \Bn=0,~(\Bn\cdot \BD(\Bu)+\alpha \Bu)\cdot \Bt=0~\text{ on }\partial\Omega,
\end{equation}
are also usually used, which were suggested by Navier \cite{Na} for the first time. Here $\BD(\Bu)$ is the strain tensor defined by
\[
 (\BD(\Bu))_{ij}=(\partial_{x_j}u_i + \partial_{x_i} u_j )/2,
\] and $\alpha\ge 0$ is the friction coefficient which measures the tendency of a fluid to slip over the boundary. $\Bt$ and $\Bn$ are the unit tangent and outer normal vector on the boundary $\partial\Omega$, respectively. If $\alpha=0$, the boundary conditions \eqref{BC} are also called the full slip boundary conditions. If $\alpha\to \infty$, the boundary conditions \eqref{BC} formally reduce to the classical no-slip boundary conditions.

The Navier-Stokes system with Navier slip boundary conditions has been widely studied in various aspects. One may refer to \cite{Bei1,CMR,DL,DLX,IS,Ke,LZ,MR,APS,Ga,GL,JM,WWX,XX} and the references therein for some important results on the nonstationary problem.  For the stationary problem, the existence and regularity of the solutions were first studied in \cite{SS}, where the usual Dirichlet conditions and the full slip conditions are imposed on different parts of the boundary of a three-dimensional interior or exterior domain. It is noteworthy that  the existence and the regularity for solutions of a generalized Stokes system with Navier boundary conditions were investigated in \cite{Bei} in some regular domain. In a flat domain supplemented with full slip boundary conditions, the existence and uniqueness of very weak, weak, and strong solutions  of the Navier-Stokes system  were proved in appropriate Banach spaces in \cite{Ber, AR}.  Recently,  the stationary Stokes and Navier-Stokes system with nonhomogeneous Navier boundary conditions in a three-dimensional bounded domain were studied in \cite{AACG}, where the existence and uniqueness for weak and strong solutions in $W^{1,p}$ and $W^{2,p}$ spaces, respectively were established even when friction coefficient $\alpha$ is generalized to a function with minimal regularity, and the behavior of the solution is also investigated when $\alpha$ tends to infinity. For more analysis on flows with the Navier slip boundary condition, one may refer to \cite{Co,Bei2,Me}.

For flows in a channel with Navier-slip boundary condition, we can also consider Leray problem where the far field shear flows can be prescribed. Without loss of generality, we assume $c_1^\pm=-1$ and $c_2^\pm= 1$ for the channel $\Omega ^\sharp$, then the associated shear flow in $\Omega^\sharp$ with Navier boundary conditions can be written as 
\begin{equation}\label{shearflow}
	\BU=\Phi\left(\frac{3(2+\alpha)}{4(3+\alpha)}-\frac{3\alpha}{4(3+\alpha)}x_2^2\right)\Be_1.	
\end{equation}
 The corresponding Leray problem has been studied by \cite{Mu1,Mu2,Mu3,Ko,LPY} and references therein.  In \cite{Mu1}, the problem \eqref{NS}-\eqref{BC} with friction coefficient $\alpha=0$ was solved for any flux provided the two-dimensional channel with straight upper boundary is contained in a straight channel and coincides with the straight channel at far field. Then the exponential convergence of the velocity is studied in \cite{Mu3}. It's worth noting that the Dirichlet norm of the solution is finite since the corresponding shear flow $\BU$ is a constant flow in the case $\alpha=0$. For general two-dimensional channel with straight outlets, it was also proved  in \cite{Mu2} that the Navier-Stokes system has a smooth solution with arbitrary flux if
 \[
 \|\alpha-2\chi\|_{L^\infty(\partial\Omega)}\leq C(\Omega),
 \]
 where $\chi$ is the curvature of the boundary and $C(\Omega)$ is a constant depending only on $\Omega$. However, the far behavior is not known even when the flux is small. The Leray problem for flows in a general two-dimensional channel with total slip boundary condition has been solved in \cite{SWX22} recently.

In the case that the domain is a three-dimensional pipe with straight outlets, a weak solution of the Navier-Stokes system with arbitrary flux is  obtained in \cite{Ko}, which satisfies mixed boundary condition and the far field behavior \eqref{far field}. Very recently, Leray problem with Navier boundary condition was solved  in \cite{LPY},  as long as the flux $\Phi$ is small and the pipe becomes straight at large distance.  The exponential convergence at far fields and regularity of the solutions with small fluxes were also achieved in \cite{LPY}.

In this paper, we study the Navier-Stokes system \eqref{NS} supplemented with \eqref{flux constraint}-\eqref{BC} in channels whose cross-section can be even unbounded. Before stating the main results of this paper, the definitions of some function spaces and the weak solution are introduced.
\begin{definition}\label{def1}
	Given a domain $D\subseteq \mathbb{R}^2$, denote
	\[
	L_0^2(D)=\left\{w(x): w\in L^2(D), \, \int_D w(x)dx =0.\right\}.
	\]
Given $\Omega$ defined in \eqref{defOmega},		for any  constants $a<b$ and $T>0$, denote
	\begin{equation*}
	\Omega_{a,b}=\{(x_1,x_2)\in \Omega:a<x_1<b\}\  \text{ and }\ \Omega_T=\Omega_{-T,T}.
	\end{equation*}
	Define
	\begin{equation*}
		\mathcal{C}(\Omega_{a,b})=\left\{\Bu|_{\Omega_{a,b}}:\begin{array}{l} \Bu\in C^\infty(\overline{\Omega}),~\Bu=0\text{ in  }\Omega\setminus \Omega_{a,b},~
			\Bu\cdot \Bn=0 \text{ on }\partial\Omega_{a,b}\cap \partial\Omega\end{array} \right\}
	\end{equation*}
and
	\begin{equation*}
		\mathcal{C}_\sigma(\Omega_{a,b})=\left\{\Bu|_{\Omega_{a,b}}: ~ \Bu\in\mathcal{C}(\Omega_{a,b}),
			~{\rm div}\,\Bu=0 \text{ in }\Omega_{a,b} \right\}.
	\end{equation*}
Let $\mcH(\Omega_{a,b})$ and $\mcH_\sigma(\Omega_{a,b})$  be the completions of $\mathcal{C}(\Omega_{a,b})$ and  $\mathcal{C}_\sigma(\Omega_{a,b})$  under $H^1$ norm, respectively. Denote
\begin{equation*}
	H_\sigma(\Omega)=\left\{\Bu\in H_{loc}^1(\Omega):~{\rm div}\,\Bu=0 \text{ in }\Omega,~\Bu\cdot \Bn=0 \text{ on } \partial\Omega\right\}
\end{equation*}
and $H_*^1(\Omega_{a,b})$ to be the set of vector-valued functions in $H^1(\Omega_{a,b})$
with zero flux, i.e., for any $\Bv\in H_*^1(\Omega_{a,b})$, one has
\begin{equation}\label{zeroflux}
\int_{f_1(x_1)}^{f_2(x_1)} v_1(x_1, x_2) dx_2=0\quad \text{for any}\,\, x_1\in (a, b).
\end{equation}
\end{definition}

\begin{definition}
	A vector field $\Bu\in H_\sigma(\Omega)$ is said to be a weak solution of the Navier-Stokes system \eqref{NS} with Navier slip boundary conditions \eqref{BC} if for any $T>0$, $\Bu$ satisfies
\begin{equation}\label{weak solution}
\int_{\Omega}2\BD(\Bu):\BD(\Bp)+\Bu\cdot\nabla\Bu\cdot \Bp\,dx+2\alpha \int_{\partial\Omega}\Bu\cdot \Bp\,ds
=0 \quad \text{for any}\,\,  \Bp \in \mcH_\sigma(\Omega_T). 	
\end{equation}
\end{definition}

Denote
\begin{equation}\label{deffbar}
	f(x_1):= f_2(x_1)-f_1(x_1)\quad \text{and}\quad \bar{f}(x_1):= \frac{f_2(x_1)+f_1(x_1)}{2}.
\end{equation}
In this paper, we always assume that
\begin{equation}\label{assumpf}
	\inf_{x_1\in \mathbb{R}} f(x_1)=d >0\quad \text{and}\quad \max_{i=1,2} \|f_i'\|_{C(\mathbb{R})}=\beta <+\infty.
\end{equation}

 The first main result of this paper can be stated as follows.
\begin{theorem}\label{bounded channel}
	Let $\Omega$ be the domain given in \eqref{defOmega} and the friction coefficient $\alpha>0$. If the width of the channel $\Omega$  is uniformly bounded, i.e.,
	\be \label{bounded}
	f(x_1) := f_2(x_1)-f_1(x_1) \leq \overline{d} < +\infty, \ \ \ \text{for any}\, \, x_1 \in \mathbb{R},
	\ee
	then the problem \eqref{NS} and \eqref{flux constraint}-\eqref{BC} has a solution $\Bu\in H_\sigma(\Omega)$ satisfying the estimate
	\begin{equation}\label{1-4}
		\|\nabla \Bu\|_{L^2(\Omega_t)}^2+\|\Bu\|_{L^2(\partial\Omega_t\cap\partial\Omega)}^2\leq \overline{C}(1+t),\ \ \text{ for any }t\ge 0,
	\end{equation}
	where $\overline{C}$ is a positive constant independent of $t$.
	Furthermore, the solution satisfies the following properties.
	\begin{enumerate}
		\item[(a)] There exists a constant $\Phi_0>0$ such that for any flux $\Phi\in [0,\Phi_0)$, the solution $\Bu$ obtained in Theorem \ref{bounded channel} is unique
	in the class of functions satisfying \eqref{1-4}.

		\item[(b)]	If, in addition, the outlet of the channel is straight, i.e., there exists a constant $k>0$ such that $\Omega\cap \{x_1>k\}=\{(x_1,x_2): x_1>k, x_2\in (c_1, c_2)\}$ for some constants $c_1<c_2$, then there exists a constant $\Phi_1>0$ such that for any flux $\Phi\in [0,\Phi_1)$, the solution $\Bu$ tends to the corresponding shear flow $\BU=U(x_2)\Be_1$ in the sense
		\begin{equation*}
			\|\Bu-\BU\|_{H^1(\Omega\cap \{x_1>k\})}<\infty.
		\end{equation*}
	\end{enumerate}
\end{theorem}

There are a few remarks in order.
\begin{remark}
	The constant $\overline{C}$ depends only on the slip coefficient $\alpha$, the flux $\Phi$, and the domain $\Omega$. More precisely, it depends on $\alpha$, $\Phi$, $\|f_i\|_{C^2(\R)}$, and $\inf_{x_1\in \R} f(x_1)$.
\end{remark}

\begin{remark}
	The case $\alpha=0$ is studied in detail in \cite{SWX22}, where Leray problem with arbitrary flux in two-dimensional channels with straight ends is completely solved. A similar result has been  obtained independently in \cite{LPY2}. 
	\end{remark}
\begin{remark}
After we finished this paper, we got to know that  the existence and asymptotic behavior of solutions for steady Navier-Stokes system in channels with Navier boundary conditions have also been investigated in \cite{LPY2}. However, the major results and methods in \cite{LPY2} are in the similar spirit as that in \cite{A1}, which are quite different from that in this paper. 
\end{remark}

\begin{remark}
	When the flux is small, Theorem \ref{bounded channel} provides a positive answer to Leray problem with Navier slip boundary condition. Employing the same technique as in \cite{LS,LPY2,Kn,To}, one can also  prove the exponential convergence of $\Bu$ to the corresponding shear flow.
\end{remark}
\begin{remark}
	The asymptotic convergence of the solutions to the shear flows at far field  highly relates to the Liouville type theorem of shear flows in a flat strip. One may refer to \cite{WX3} for recent progress along this direction for the flows in a pipe with Navier boundary conditions.
\end{remark}

For the flows in channels with unbounded outlets, we have also the following theorem.
\begin{theorem}\label{unbounded channel}
	Let $\Omega$ be the domain given in \eqref{defOmega} and the friction coefficient $\alpha>0$. Suppose that	
	\begin{equation}\label{assumpf''}
		\max_{i=1,2}\sup_{x_1\in \R}|(f_i^{\prime \prime} f)(x_1)|<\infty.
	\end{equation}
\begin{enumerate}
	\item[(\romannumeral1)] (Existence) The problem \eqref{NS} and \eqref{flux constraint}-\eqref{BC} has a solution $\Bu\in H_\sigma(\Omega)$  satisfying the estimate
	\begin{equation}\label{1-9}
		\|\nabla \Bu\|_{L^2(\Omega_t)}^2+\|\Bu\|_{L^2(\partial\Omega_t\cap\partial\Omega)}^2\leq \widetilde{C}\left(1+\int_{-t}^t f^{-3}(x_1)\,dx_1 \right)\text{ for any }t\ge 0,
	\end{equation}
	where $\widetilde{C}$ is a positive constant depending only on the friction coefficient $\alpha$, the flux $\Phi$, and $\Omega$.
	
	\item[(\romannumeral2)] (Uniqueness)
	 If, in addition, it holds  that either
		\begin{equation}\label{1-16}
			\left|\int_0^{\pm\infty}f^{-3}(\tau)\,d\tau\right|=\infty,\ \ \ \lim_{|t|\to +\infty} f'(t)= 0,
		\end{equation}
		or
		\begin{equation}\label{1-17}
			\left|\int_0^{\pm\infty}f^{-3}(\tau)\,d\tau\right|<\infty,\ \ \ \lim_{t\to \pm\infty}\frac{\sup_{\pm \tau \ge |t|}f'(\tau)}{\left|\int_t^{\pm \infty} f^{-3}(\tau)\,d\tau\right|^\frac12}= 0,
		\end{equation}
	 then there exists a constant $\Phi_2>0$ such that for any flux $\Phi\in [0,\Phi_2)$, the solution $\Bu$ obtained in Theorem \ref{unbounded channel} is unique in the class of functions satisfying \eqref{1-9}.
\end{enumerate}
\end{theorem}

There are some remarks in order.
\begin{remark}
If $f(t)$ is a power function at far field, the conditions \eqref{1-16}-\eqref{1-17} are equivalent to 
\begin{equation*}
	f(t)=o(t^\frac35)\ \ \ \ \text{ as }|t|\to\infty.
\end{equation*} 
\end{remark}

\begin{remark}
It should be emphasized that there is no restriction on the flux $\Phi$ for the existence of solutions in both Theorems \ref{bounded channel} and \ref{unbounded channel}.
\end{remark}

\begin{remark} In certain sense, the estimates \eqref{1-4} and \eqref{1-9} are optimal, as there exists a constant $C>0$ such that
\begin{equation}\label{optimalest}
\Phi^2  \int_{-t}^t f^{-3} (x_1) \, dx_1 \leq C(\|\nabla \Bu\|_{L^2(\Omega_t)}^2+\|\Bu\|_{L^2(\partial\Omega_t\cap\partial\Omega)}^2).
\end{equation}
 Indeed, for any flow $\Bu$ with flux $\Phi$, one has
\begin{equation*}\begin{aligned}
		\Phi^2=&\left|\int_{\Sigma(x_1)}\Bu\cdot \Bn\,ds\right|^2\leq |\Sigma(x_1)|\int_{\Sigma(x_1)}|\Bu|^2\,dx_2 \\
		\leq&|f(x_1)|\int_{\Sigma(x_1)}\left|\Bu(x_1, f_2(x_1))-\int^{f_2(x_1)}_{x_2}\partial_{x_2}\Bu(x_1,\xi) \,d\xi\right|^2\,dx_2\\
		\leq& 2|f (x_1)|\int_{\Sigma(x_1)}\left(|\Bu(x_1,f_2(x_1))|^2+|\Sigma(x_1)|\int_{\Sigma(x_1)}|\partial_{x_2}\Bu(x_1,\xi)|^2\,d\xi\right)\,dx_2\\
		\leq&C|f (x_1)|^3\left(|\Bu(x_1,f_2(x_1))|^2+\int_{\Sigma(x_1)}|\nabla \Bu|^2\,dx_2\right).
\end{aligned}\end{equation*}
Integrating this inequality with respect to $x_1$ over $(-t,t)$, one has \eqref{optimalest}.
\end{remark}

\begin{remark}
For the channels satisfying \eqref{1-16} and \eqref{1-17}, we also obtain the decay rate for the Dirichlet norm of the solutions in Proposition \ref{decay rate-right}. The pointwise convergence rates of the solutions of steady Navier-Stokes system in a class of two-dimenisonal channels are studied in detail in a forthcoming paper \cite{SWXrate}. 
\end{remark}



Here we give a sketch of the proof and point out the major difficulties and key ingredients of the analysis.  Inspired by the analysis of \cite{LS}, we construct a flux carrier $\Bg$, which is a solenoidal vector field with flux $\Phi$ satisfying the Navier slip boundary condition. We seek for a solution of the form $\Bu = \Bv + \Bg$. To prove the existence of solutions, the main difficulty is to estimate the term $\int \Bv \cdot \nabla\Bv \cdot \Bg\,dx$. When the flow satisfies the no-slip boundary condition, i.e. $\Bv=0$ on $\partial\Omega$, the main idea in \cite{A1, LS} is to
apply the Hardy inequality to show
\begin{equation}\label{1-5}
	\left|\int \Bv \cdot \nabla\Bv \cdot \Bg\,dx \right|\leq\delta ^2\|\nabla\Bv\|_{L^2}^2
\end{equation}
for arbitrary $\delta>0$ provided
\[
|\Bg(\Bx)|\leq \frac{c}{\mathrm{ dist}(\Bx,\partial\Omega)}\quad \text{for}\,\, \Bx \,\,\text{near the boundary}.
\]
  If $\Bv$ satisfies the Navier boundary conditions so that in general $\Bv$ does not vanish on the boundary, one cannot apply Hardy inequality directly to prove \eqref{1-5}. A key observation is that the inequality \eqref{1-5} still holds even   when one has only $\Bv\cdot \Bn=0$ on the boundary, if the flux carrier $\Bg$ is chosen appropriately.

The second difficulty is that when $\Bv$ is not zero on the boundary so that we cannot apply Poincar\'e's inequality directly. Another key point in this paper is that the Poincar\'e's inequality for the normal component of velocity field still holds since $\Bv \cdot \Bn =0$ on the boundary. Meanwhile,  we can estimate the horizontal component since the flux of $\Bv$ equals to zero.


The third issue is about the Bogovskii map which is used to estimate the pressure term. Since the norm of the Bogovskii map depends on the domain so that we have to choose the truncated domains in a delicate way to make sure that the norm is uniform, in particular, when the width of channels grows to be unbounded at far fields.

 Finally,  Korn's inequality plays an important role in the energy estimate for the Navier-Stokes system. In this paper, we prove that the constant in Korn's inequality is uniform with respect to different truncated domains.

The rest of the paper is organized as follows. In Section \ref{secpreliminary}, we give some  lemmas which are used here and there in the paper. The flux carrier $\Bg$ is constructed in Section \ref{secflux}, which is used to prove the existence of the approximate solution on the bounded domain $\Omega_T$ via Leray-Schauder fixed point theorem. Section 4 devotes to the study on the Navier-Stokes system in channels with bounded cross-sections. The crucial growth estimate on the Dirichlet norm of the approximate solution is established with the aid of the analysis on divergence equation. By virtue of these estimates, the well-posedness of the Navier-Stokes system and the far field behavior of the solutions are proved in Section \ref{secexist1}. In Section \ref{secexist2}, the flows in channels with unbounded outlets are investigated, where the existence of solutions is also proved with the help of the growth estimate on Dirichlet norm.

\section{Preliminaries}\label{secpreliminary}
In this section, we collect some elementary but important lemmas. We first give the Poincar\'e type inequality in channels.
\begin{lemma}\label{lemmaA1}
For any $\Bv \in H^1_*(\Omega_{a,b})$ satisfying $\Bv \cdot \Bn =0$ on $\partial\Omega_{a,b}\cap \partial\Omega$, one has
\begin{equation}\label{A1-0}
	\left\|\frac{v_1}{f}\right\|_{L^2(\Omega_{a,b})}\leq M_0 \left\|\partial_{x_2} v_1\right\|_{L^2(\Omega_{a,b})}
\end{equation}
and
\begin{equation}
\left\|\Bv\right\|_{L^2(\Omega_{a,b})}\leq M_1(\Omega_{a,b}) \left\|\nabla\Bv\right\|_{L^2(\Omega_{a,b})},
\end{equation}
 where $M_0$ is a uniform constant independent of the domain $\Omega_{a,b}$ and
 \begin{equation}\label{defM1}
 M_1(\Omega_{a,b})=C\|f\|_{L^\infty(a, b)}
  \cdot \left(1+\|f_2'\|_{L^\infty(a,b)}\right).
 \end{equation}
\end{lemma}
\begin{proof}
For any $\Bv\in H^1_*(\Omega_{a,b})$, one has \eqref{zeroflux}.
Hence it follows from Poincar\'{e}'s inequality that
\begin{equation}\label{A1-1}
\int_{\Sigma(x_1)}|v_1|^2\,dx_2\leq M_0^2 | \Sigma(x_1)|^2\int_{\Sigma(x_1)}|\partial_{x_2}v_1|^2\,dx_2= M_0^2  f^2(x_1)\int_{\Sigma(x_1)}|\partial_{x_2}v_1|^2\,dx_2.
\end{equation}
Integrating \eqref{A1-1} with respect to $x_1$ from $a$ to $b$, one has \eqref{A1-0} and
\begin{equation}\label{A1-2}
\|v_1\|_{L^2(\Omega_{a,b})}\leq M_0 \|f\|_{L^\infty(a,b)}\|\partial_{x_2}v_1\|_{L^2(\Omega_{a,b})}.
\end{equation}
On the other hand, the boundary condition $\Bv\cdot \Bn=0$ on the upper boundary $x_2=f_2(x_1)$ can be written as
\begin{equation*}
v_2(x_1, f_2(x_1))-f_2'(x_1)v_1(x_1, f_2(x_1))=0.
\end{equation*}
Using Poinc\'{a}re's inequality again yields
\begin{equation}\label{A1-3}
\int_{\Sigma(x_1)}|v_2-f_2'v_1|^2\,dx_2\leq C|\Sigma(x_1)|^2\int_{\Sigma(x_1)}|\partial_{x_2}(v_2-f_2'v_1)|^2\,dx_2.
\end{equation}
Integrating \eqref{A1-3} with respect to $x_1$ from $a$ to $b$, one also has
\begin{equation}\label{A1-4}
\|v_2-f_2'v_1\|_{L^2(\Omega_{a,b})}\leq C\|f\|_{L^\infty(a,b)}\|\partial_{x_2}(v_2-f_2'v_1)\|_{L^2(\Omega_{a,b})}.
\end{equation}
With the aid of Minkowski's inequality, it follows from  \eqref{A1-2} and \eqref{A1-4} that one has
\begin{equation}\label{A1-5}
\begin{aligned}
\|v_2\|_{L^2(\Omega_{a,b})}\leq &\|v_2-f_2'v_1\|_{L^2(\Omega_{a,b})}+\|f_2'\|_{L^\infty(a,b)}\|v_1\|_{L^2(\Omega_{a,b})}\\
\leq &C\|f\|_{L^\infty(a,b)}\left(\|\partial_{x_2}(v_2-f_2'v_1)\|_{L^2(\Omega_{a,b})}+\|f_2'\|_{L^\infty(a,b)}\|\partial_{x_2}v_1\|_{L^2(\Omega_{a,b})}\right)\\
\leq&C\|f\|_{L^\infty(a,b)}\left(\|\partial_{x_2}v_2\|_{L^2(\Omega_{a,b})}+\|f_2'\|_{L^\infty(a,b)}\|\partial_{x_2}v_1\|_{L^2(\Omega_{a,b})}\right)\\
\leq &C\|f\|_{L^\infty(a,b)}(1+\|f_2'\|_{L^\infty(a,b)})\|\partial_{x_2}\Bv\|_{L^2(\Omega_{a,b})}.
\end{aligned}
\end{equation}
This finishes the proof of the lemma.
\end{proof}

\begin{lemma}\label{lemmaA2} Assume that $f(x_1)\ge d_{a,b}>0$ for any $x_1\in (a,b)$. Then for any $\Bv \in H_*^1(\Omega_{a,b})$ satisfying  $\Bv \cdot \Bn = 0$ on $\partial\Omega_{a,b}\cap \partial\Omega$, one has
\begin{equation*}
\|\Bv\|_{L^4(\Omega_{a,b})}\leq M_4(\Omega_{a,b}) \| \nabla \Bv\|_{L^2(\Omega_{a,b})},
\end{equation*}
where 
\begin{equation}\label{defM4}
M_4(\Omega_{a,b})=C(1+\|(f_1',f_2')\|_{L^\infty(a,b)}^2) \left(\frac{M_1}{b-a}+1\right)^{\frac12}(|\Omega_{a,b}|+(b-a)d_{a,b})^{\frac14} \left(1+\frac{M_1}{d_{a,b}}\right)
\end{equation}
with a universal constant $C$ and $M_1=M_1(\Omega_{a,b})$ defined in \eqref{defM1}.
\end{lemma}
\begin{proof} We divide the proof into two steps.
	
	{\it Step 1. Extension of functions.}
We claim that $\Bv$ can be extended to a function $\widetilde{\Bv}\in H^1((a,b)\times \mathbb{R})$ satisfying
\begin{equation*}
\operatorname{supp} \widetilde{\Bv} \subset \{(x_1,x_2):x_1\in[a,b],~~x_2\in[f_1(x_1)-d_{a,b},f_2(x_1)+d_{a,b}]\}
\end{equation*} and
\begin{equation}\label{eqA2_1}
\|\nabla \widetilde{\Bv}\|_{H^1((a,b)\times \mathbb{R})}\leq M_2(\Omega_{a,b}) \left( \|\nabla \Bv\|_{L^2(\Omega_{a,b})} + d_{a,b}^{-1} \|\Bv\|_{L^2(\Omega_{a, b}) }\right),
\end{equation}
where $M_2(\Omega_{a,b})=C  \left(1+\|(f_1',f_2')\|_{L^\infty(a,b)}^2 \right)$.

Let $\chi(t)$ be a smooth decreasing function on $[0,+\infty)$ satisfying
\begin{equation}\label{defchi}
 \chi(t)=1\ \text{ if }0\leq t\leq \frac{d_{a,b}}{4},  \ \ \chi(t)=0\  \text{ if }t\ge \frac{d_{a,b}}{2}, \quad \text{and}\,\, |\chi'(t)|\leq \frac{8}{d_{a,b}}\,\,\text{for}\,\, t\in [0, +\infty).
\end{equation}
Given $\Bv\in C^1(\overline{\Omega_{a,b}})$,
one can  decompose it as follows,
\begin{equation*}
	\begin{aligned}
\Bv=&\Bv(1-\chi(f_2(x_1)-x_2)-\chi(x_2-f_1(x_1)))+\Bv\chi(f_2(x_1)-x_2)+\Bv\chi(x_2-f_1(x_1))\\
=:&\Bv_0+\Bv_1+\Bv_2.
\end{aligned}
\end{equation*}
Noting that $\Bv_1$ is defined on $\{(x_1,x_2):x_1\in [a,b],~~x_2\in(-\infty,f_2(x_1)]\}$, we define
\begin{equation*}
\widetilde{\Bv}_1 =:\left\{
\begin{aligned}
&\Bv_1(x_1,x_2),~~~~~~~~~~~~~~~~~~~~~~~~~~~~~~~~~~~~~~~~&x_2\leq f_2(x_1),\\
&-3\Bv_1(x_1,-x_2+2f_2(x_1))+4\Bv_1 \left(x_1,-\frac12x_2+\frac32f_2(x_1) \right), &x_2>f_2(x_1).
\end{aligned}\right.
\end{equation*}
It's easy to verify that $\widetilde{\Bv}_1\in C^1([a,b]\times \mathbb{R})$ and
\begin{equation*}
\operatorname{supp} \widetilde{\Bv}_1 \subset \left\{(x_1,x_2):\, x_1\in [a,b],~x_2\in \left[f_2(x_1)-\frac{d_{a,b}}{2},\, f_2(x_1)+d_{a,b} \right] \right\}.
\end{equation*}
Moreover, direct computation shows that
\begin{equation*}
\| \nabla \widetilde{\Bv}_1 \|_{L^2((a,b)\times \mathbb{R})}\leq C(1+\|f_2'\|_{L^\infty(a,b)})\| \nabla \Bv_1\|_{L^2(\Omega_{a,b})}.
\end{equation*}
Similarly, one can also extend $\Bv_2$ to the domain $[a,b]\times \mathbb{R}$ by defining
\begin{equation*}
\widetilde{\Bv}_2 =:\left\{
\begin{aligned}
&\Bv_2(x_1,x_2), ~~~~~~~~~~~~~~~~~~~~~~~~~~~~~~~~~~~~~&x_2\ge  f_1(x_1),\\
&-3\Bv_2(x_1,-x_2+2f_1(x_1))+4\Bv_2 \left(x_1,-\frac12x_2+\frac32f_1(x_1)\right), &x_2<f_1(x_1).
\end{aligned}\right.
\end{equation*}
The extension $\widetilde{\Bv}_2 \in C^1([a,b]\times \mathbb{R})$ satisfies
\begin{equation*}
\operatorname{supp} \widetilde{\Bv}_2 \subset \left\{(x_1,x_2):x_1\in [a,b],~x_2\in \left[f_1(x_1)-d_{a,b},\, f_1(x_1)+\frac{d_{a,b}}{2} \right]\right\}
\end{equation*}
and
\begin{equation*}
\|\nabla \widetilde{\Bv}_2 \|_{L^2((a,b)\times \mathbb{R})}\leq C(1+\|f_1'\|_{L^\infty(a,b)})\| \nabla \Bv_2\|_{L^2(\Omega_{a,b})}.
\end{equation*}

Now we define the extension of $\Bv$ by
\begin{equation*}
\widetilde{\Bv}:=\Bv_0+\widetilde{\Bv}_1 +\widetilde{\Bv}_2 .
\end{equation*}
Let $\widetilde{\Omega}_{a, b} = (a, b) \times \mathbb{R}$. According to the above computations,
\begin{equation*}
\begin{aligned}
\| \nabla \widetilde{\Bv}\|_{L^2( \tilde{\Omega}_{a, b} ) }\leq & \| \nabla \widetilde{\Bv}_1 \|_{L^2(\widetilde{\Omega}_{a, b} ) }+ \| \nabla \widetilde{\Bv}_2 \|_{L^2 (\widetilde{\Omega}_{a, b} )}+ \|\nabla \Bv_0\|_{L^2 ( \widetilde{\Omega}_{a, b} ) }\\
\leq&C(1+\|(f_1',f_2')\|_{L^\infty(a,b)}) \left( \|\nabla \Bv_1 \|_{L^2(\Omega_{a, b} )} + \| \nabla \Bv_2 \|_{L^2(\Omega_{a, b})}+ \|\nabla \Bv_0\|_{L^2(\Omega_{a, b})}        \right) \\
\leq&C  \left(1+\|(f_1',f_2')\|_{L^\infty(a,b)}^2 \right) \left( \|\nabla \Bv\|_{L^2(\Omega_{a, b}) }+  d_{a,b}^{-1} \|\Bv\|_{L^2(\Omega_{a, b})} \right),
\end{aligned}
\end{equation*}
where the bound of $\chi'$ in \eqref{defchi} has been used. By standard density argument, the conclusion \eqref{eqA2_1} holds for any $\Bv \in H^1(\Omega_{a,b})$.

{\it Step 2. Embedding for the extended functions.}
Denote $\Sigma^d(x_1):=(f_1(x_1)-d_{a,b},\, f_2(x_1)+d_{a,b})$ and
\begin{equation*}
\Omega_{a,b}^d:=\{(x_1,x_2):\, x_1\in(a,b),~~x_2\in \Sigma^d(x_1)\}
\end{equation*}
It is noted that $\operatorname{supp}\widetilde{\Bv}\subset \overline{\Omega_{a, b}^d} $.
Using Gagliardo-Nirenberg interpolation inequality (\hspace{1sp}\cite{Fiorenza}) gives
\begin{equation}\label{A2-1}
\begin{aligned}
\int_{\Omega_{a,b}^d}|\widetilde{\Bv}|^4\,dx=&\int_a^b\,dx_1\int_{\Sigma^d(x_1)}|\widetilde{\Bv}|^4\,dx_2\\
\leq &C\int_{a}^b\,dx_1\int_{\Sigma^d(x_1)}|\partial_{x_2}\widetilde{\Bv}|^2\,dx_2\left(\int_{\Sigma^d(x_1)}|\widetilde{\Bv}|\,dx_2\right)^2\\
\leq &C\sup_{x_1\in (a,b)}\left(\int_{\Sigma^d(x_1)}|\widetilde{\Bv}|\,dx_2\right)^2\int_{\Omega_{a,b}^d}|\partial_{x_2}\widetilde{\Bv}|^2\,dx.
\end{aligned}
\end{equation}
On the other hand, note that for any $x_1\in (a,b)$, one has
\begin{equation*}
|\widetilde{\Bv}(x_1,x')|\leq |\widetilde{\Bv}(x_1^0,x')|+\int_a^b|\partial_{x_1}\widetilde{\Bv}|\,dx_1.
\end{equation*}
Then we integrate the inequality with respect to $x'\in \Sigma^d (x_1)$ and $x_1^0$ from $a$ to $b$, and use H\"{o}lder inequality to obtain
\begin{equation*}\begin{aligned}
&\int_{\Sigma^d(x_1)}|\widetilde{\Bv}|\,dx_2\leq (b-a)^{-1}\int_{\Omega^d_{a,b}}|\widetilde{\Bv}|\,dx+\int_{\Omega^d_{a,b}}|\partial_{x_1}\widetilde{\Bv}|\,dx\\
\leq & (b-a)^{-1}|\Omega^d_{a,b}|^\frac12\left(\int_{\Omega^d_{a,b}}|\widetilde{\Bv}|^2\,dx\right)^\frac12
+|\Omega^d_{a,b}|^\frac12\left(\int_{\Omega^d_{a,b}}|\partial_{x_1}\widetilde{\Bv}|^2\,dx\right)^\frac12\\
\leq& \left[ (b-a)^{-1} M _1(\Omega_{a,b}^d)+1 \right] |\Omega^d_{a,b}|^\frac12\left(\int_{\Omega^d_{a,b}}|\nabla\widetilde{\Bv}|^2\,dx\right)^\frac12,
\end{aligned}\end{equation*}
where $ M _1(\Omega_{a,b}^d)$ is the constant obtained in Lemma \ref{lemmaA1} corresponding to the domain $\Omega^d_{a,b}$. This, together with \eqref{A2-1} gives
\begin{equation}\label{eqA2_2}
	\begin{aligned}
\|\widetilde{\Bv}\|_{L^4(\widetilde{\Omega}_{a,b})}\leq M_3(\Omega_{a,b})\|\nabla\widetilde{\Bv} \|_{L^2(\widetilde{\Omega}_{a,b})},
\end{aligned}\end{equation}
where
\begin{equation*}
\begin{aligned}
	M_3(\Omega_{a,b})=&\left[ (b-a)^{-1} M _1(\Omega_{a,b}^d)+1 \right]^\frac12|\Omega^d_{a,b}|^\frac14\\
	=&\left[ C(b-a)^{-1} \|f+2d_{a,b}\|_{L^\infty(a,b)}(1+\|f_2'\|_{L^\infty(a,b)})+1 \right]^\frac12(|\Omega_{a,b}|+2(b-a)d_{a,b})^\frac14\\
	\leq& \left[ 3(b-a)^{-1} M _1(\Omega_{a,b})+1 \right]^\frac12(|\Omega_{a,b}|+2(b-a)d_{a,b})^\frac14.
\end{aligned}
\end{equation*}
Finally, by virtue of Lemma \ref{lemmaA1}, one has
\begin{equation*}
\begin{aligned}
\|\Bv\|_{L^4(\Omega_{a,b})} &\leq \|\widetilde{\Bv}\|_{L^4( \widetilde{\Omega}_{a,b})}\leq M_3(\Omega_{a,b}) \|\nabla \widetilde{\Bv}\|_{L^2( \widetilde{\Omega}_{a,b})}\\
&\leq M_2(\Omega_{a,b})M_3(\Omega_{a,b})\left( \|\nabla \Bv\|_{L^2(\Omega_{a,b}) } + d_{a,b}^{-1} \|\Bv\|_{L^2(\Omega_{a,b})} \right)\\
& \leq M_4(\Omega_{a,b}) \|\nabla \Bv\|_{L^2(\Omega_{a,b})},
\end{aligned}
\end{equation*}
where $M_4(\Omega_{a,b})$ is given in \eqref{defM4}, and \eqref{eqA2_1} has been used to get the third inequality. This completes the proof of Lemma \ref{lemmaA2}.
\end{proof}

\begin{lemma}\label{lemmaA3}
	Given  $\Bv \in \mcH_\sigma(\Omega_{a, b}),$ it holds that
\begin{equation}\label{eqA3_1}
\mfc\|\nabla \Bv\|_{L^2(\Omega_{a, b} )}^2 \leq 2\|\BD(\Bv)\|_{L^2(\Omega_{a, b} )}^2+\alpha\|\Bv\|_{L^2(\partial\Omega_{a, b}\cap \partial\Omega)}^2,
\end{equation}
where
\be \label{defmfc}
\mfc = \frac{\alpha}{\alpha + \| \partial_\tau \Bn\|_{L^\infty(\partial \Omega)} }.
\ee
\end{lemma}

\begin{proof}
Without loss of generality, we assume that $\Bv \in \mathcal{C}_\sigma(\Omega_{a,b})$. According to the formula
\begin{equation*}
\Delta \Bv=2\div \BD(\Bv),
\end{equation*}
integration by parts yields
\begin{equation}\label{A3-2}
\begin{aligned}
&\int_{\Omega_{a, b} }2|\BD(\Bv)|^2\,dx -\int_{\partial\Omega_{a, b}\cap \partial \Omega } 2\Bn\cdot\BD(\Bv)\cdot\Bv  \,ds\\
=& \int_{\Omega_{a, b} }-2{\rm div}\BD(\Bv)\cdot\Bv\,dx=\int_{\Omega_{a, b} }-\Delta \Bv\cdot\Bv\,dx\\
=&\int_{\Omega_{a, b} }|\nabla\Bv|^2\,dx-\int_{\partial\Omega_{a, b} \cap \partial\Omega}\Bn\cdot \nabla \Bv\cdot \Bv\,ds.
\end{aligned}
\end{equation}
Therefore, one has
\begin{equation}\label{A3-3}
\int_{\Omega_{a, b} }|\nabla\Bv|^2\,dx=\int_{\Omega_{a, b} }2|\BD(\Bv)|^2\,dx -  \int_{\partial\Omega_{a, b}\cap \partial \Omega }2 \Bn \cdot \BD(\Bv)  \cdot \Bv  - \Bn\cdot \nabla \Bv\cdot \Bv\,ds.
\end{equation}
Note that
\begin{equation}\label{A3-3.5}
\Bn\cdot \nabla\Bv\cdot \Bv=2\Bn\cdot \BD(\Bv)\cdot \Bv-\sum_{i,j=1}^2n_j\partial_{x_i} v_jv_i.
\end{equation}
The boundary condition $\Bv\cdot \Bn= 0$  implies that  $\partial_{\tau }(\Bv\cdot \Bn)=0$ on the boundary.  Hence it holds that
\begin{equation}\label{A3-4}
\sum_{i,j=1}^2n_j\partial_{x_i}v_jv_i=(\Bv\cdot  \Bt) [\partial_\tau (\Bv \cdot \Bn)-\Bv \cdot \partial_\tau\Bn]=-(\Bv\cdot  \Bt)(\Bv \cdot \partial_\tau\Bn),\ \ \ \ \mbox{on}\ \partial \Omega_{a, b} \cap \partial \Omega.
\end{equation}
Substituting \eqref{A3-3.5}-\eqref{A3-4} into \eqref{A3-3} yields
\begin{equation*}
\int_{\Omega_{a, b}}|\nabla\Bv|^2\,dx\leq 2\int_{\Omega_{a,b}}|\BD(\Bv)|^2\,dx+\|\partial_\tau \Bn\|_{L^\infty(\partial\Omega)}\int_{\partial\Omega_{a, b}\cap \partial\Omega}|\Bv|^2\,ds.
\end{equation*}
This gives \eqref{eqA3_1}.
Hence the proof of Lemma \ref{lemmaA3} is completed.
\end{proof}

The following lemma on the solvability of the divergence equation is used to obtain the estimates involving pressure. For the proof, one may refer to \cite[Theorem  \uppercase\expandafter{\romannumeral3}.3.1 ]{Ga} and \cite{Bo}. 
\begin{lemma}\label{lemmaA5}
Let $D \subset \R^n$ be a locally Lipschitz domain. Then there exists a constant $M_5$ such that for any $w\in L_0^2(D)$, the  problem
\begin{equation}\label{A5-1}
\left\{\begin{aligned}
{\rm div}~\Ba=w ~~~~~~~~~~&\text{ in }D,\\
\Ba=0 ~~~~~~~~~~~~~~~~~&\text{ on }\partial D,
\end{aligned}\right.
\end{equation}
has a solution $\Ba \in H^1_0(D)$ satisfying
\[
\|\nabla\Ba\|_{L^2(D)}\leq M_5(D)\|w\|_{L^2(D)}.
\]  In particular, if the domain is of the form
\begin{equation*}
 D=\bigcup_{k=1}^N D_k,
\end{equation*}
where each $D_k$ is star-like with respect to some open ball $B_k$ with  $\overline{B_k}\subset D_k$, then the constant $M_5(D)$ admits the following estimate
\begin{equation}\label{A5-2}M_5(D)\leq C_D \left(\frac{R_0}{R}\right)^n\left(1+\frac{R_0}{R}\right).
\end{equation}
Here, $R_0$ is the diameter of the domain $D$, $R$ is the smallest radius of the balls $B_k$, and
\begin{equation}\label{A5-3}
C_D=\max_{1\leq k\leq N}\left(1+\frac{|D|^\frac12}{|\tilde{D}_k|^\frac12}\right)\prod_{i=1}^{k-1}\left(1+\frac{|\hat{D}_i\setminus D_i|^\frac12}{|\tilde{D}_i|^\frac12}\right),
\end{equation}
with $\tilde{D}_i=D_i\cap \hat{D}_i$ and $\hat{D}_i=\bigcup_{j=i+1}^N D_j$.
\end{lemma}
Later on, for any $D=\{\Bx\in \Omega:~ t-1<x_1<t\}$, we shall study the dependence of the constant $M_5(D)$ on $t$ carefully, which is crucial to obtain the uniform estimate of the pressure.

We next recall the estimates for some differential inequalities, whose proof can be found in \cite{LS}. These differential inequalities play crucial role in the estimates for local Dirichlet norm. 
\begin{lemma}\label{lemmaA4}~
(1) Let $z(t)$ and $\varphi(t)$ be the nontrivial, nondecreasing, and nonnegative smooth functions. Suppose that $\Psi(t, s)$ is a monotonically increasing function with respect to $s$, equals to zero for $s=0$ and tends to $\infty$ as $s\to \infty$. Suppose that $\delta_1\in (0,1)$ is a fixed constant and for any  $t\in [t_0,T]$, $z(t)$ and $\varphi(t)$ satisfy
\begin{equation}\label{A4-1}
z(t)\leq \Psi(t, z'(t))+(1-\delta_1)\varphi(t)
\end{equation}
and
\begin{equation}\label{A4-2}
\varphi(t)\geq \delta_1^{-1}\Psi(t, \varphi'(t)).
\end{equation}
If $z(T) \leq \varphi(T)$, then
\begin{equation}\label{A4-3}
z(t)\leq \varphi(t)\text{ for any }t\in [t_0,T].
\end{equation}

(2) Assume that $\Psi(t, s) = \Psi(s)$ and the inequalities \eqref{A4-1} and \eqref{A4-2} are fulfilled for any $t\ge t_0$.  If
\begin{equation*}
  \liminf_{t\to \infty} \frac{z(t)}{\varphi(t)}<1\ \ \ \ \text{ or } \ \ \ \ \lim_{t\to \infty}\frac{z(t)},{\widetilde{z}(t)}=0
\end{equation*} 
where $\widetilde{z}(t)$ is the positive solutions to the equation
\begin{equation*}
  \tilde{z}(t)=\delta_1^{-1}\Psi(\tilde{z}'(t)),
\end{equation*}
then \eqref{A4-3}
holds.

(3) Assume that $\Psi(t, s)= \Psi(s)$ and the function $z(t)$ is nontrivial and nonnegative. If there exist $m>1,t_0, s_1\ge 0,c_0>0$ such that
\begin{equation*}
  z(t)\leq \Psi(z'(t))~~~~\text{ for any } t\ge t_0
\end{equation*}
  and
\begin{equation*}
\Psi(s)\leq c_0 s^m \text{ for any }s \ge s_1,
\end{equation*}
then
\begin{equation*}
  \liminf_{t\to \infty} t^\frac{-m}{m-1}z(t)>0.
\end{equation*}
\end{lemma}
As an application of Lemma \ref{lemmaA4},
when there is the estimate of the Dirichlet norm in some fixed sub-domain $\Omega_T$, one could further prove the uniform estimate in arbitrary $\Omega_t$ with $t\leq T$ with the aid of the differential inequalities.

\section{Flux carrier and the approximate problem}\label{secflux}
In this section, we construct the so-called flux carrier, which is a solenoidal vector field with flux $\Phi$ and satisfies the Navier-slip boundary condition \eqref{BC}.

Let $\mu(t)$ be a smooth function on $\R$ which satisfies
\begin{equation*}
	\mu(t)=\left\{
	\begin{aligned}
		&0 \,\,\,\,\,\,\text{ if } t\ge 1,\\
		&1 \,\,\,\,\,\,\text{ if } t\le 0.
	\end{aligned}
	\right.
\end{equation*}
For any $\varepsilon\in(0,1)$ to be determined, define
\begin{equation}\label{defg}
\Bg=(g_1,g_2)=(\partial_{x_2}G,-\partial_{x_1}G),
\end{equation}
where
\begin{equation*}
G(x_1,x_2;\varepsilon)=\left\{
\begin{aligned}&\Phi\mu\left(1+\varepsilon \ln \frac{f_2(x_1)-x_2}{x_2-\bar{f}(x_1)}\right),&\text{ if }x_2>\bar{f}(x_1),\\
&	0,&\text{ if }x_2\leq \bar{f}(x_1),
\end{aligned}\right.
\end{equation*}
with $\bar{f}$ defined in \eqref{deffbar}. Clearly, $\Bg$ is a smooth solenoidal vector field.

Noting
\begin{equation*}
G(x_1,x_2;\varepsilon)=\left\{
\begin{aligned}
&\Phi, &&\text{ if }x_2\text{ near }f_2(x_1),\\
&0, &&\text{ if }x_2\leq \bar{f}(x_1),
\end{aligned}
\right.
\end{equation*}
one can see that the vector field $\Bg$ vanishes near the boundary $\partial \Omega$ and satisfies the flux constraint \eqref{flux constraint}. Since $\operatorname{supp} \mu' \subset [0,1]$, one has
\begin{equation*}
\operatorname{supp} \Bg\subset \left\{ (x_1, x_2) \in \Omega:~e^{-\frac1\varepsilon}\leq \frac{f_2(x_1)-x_2}{x_2-\bar{f}(x_1)}\leq 1 \right\}.
\end{equation*}
This implies that for any $\Bx\in \operatorname{supp}\Bg$, one has
\begin{equation}\label{1-6}
 f_2(x_1)-x_2 \leq  x_2-\bar{f}(x_1)\leq e^\frac1\varepsilon \left(f_2(x_1)-x_2\right).
\end{equation}
It also follows from \eqref{1-6} that for any $\Bx\in \operatorname{supp}\Bg$, one has
\begin{equation*}
	2(x_2-\bar{f}(x_1))\ge f_2(x_1)-x_2+x_2-\bar{f}(x_1)=f_2(x_1)-\bar{f}(x_1)=\frac{f(x_1)}{2}
\end{equation*}
and
\begin{equation*}
	(1+e^{-\frac1\varepsilon})(x_2-\bar{f}(x_1))\leq  x_2-\bar{f}(x_1)+f_2(x_1)-x_2=f_2(x_1)-\bar{f}(x_1)=\frac{f(x_1)}{2},
\end{equation*}
where $f$ is defined in \eqref{deffbar}.
Hence, one has
\begin{equation}\label{1-7}
	\frac{f(x_1)}{4}\leq x_2-\bar{f}(x_1)\leq \frac{1}{1+e^{-\frac1\varepsilon}}\frac{f(x_1) }{2}\leq \frac{f(x_1)}{2}
\end{equation}
and
\begin{equation}\label{1-10}
	f_2(x_1)-x_2\ge e^{-\frac1\varepsilon }(x_2-\bar{f}(x_1)) \ge e^{-\frac1\varepsilon }\frac{f(x_1)}{4}.
   \end{equation}

Moreover, direct computations give
\begin{equation}\label{1-11-1}
\begin{aligned}
	g_1=&\Phi\partial_{x_2} \mu\left(1+\varepsilon \ln (f_2 (x_1)-x_2)- \varepsilon \ln(x_2-\bar{f}(x_1))\right)\\
	=&\varepsilon \Phi \mu'(\cdot) \left(\frac{-1}{f_2 (x_1)-x_2}-\frac{1}{x_2-\bar{f}(x_1)}\right)
\end{aligned}
\end{equation}
and
\begin{equation}\label{1-11-2}
\begin{aligned}
	g_2=&-\Phi\partial_{x_1} \mu\left(1+\varepsilon \ln (f_2 (x_1)-x_2)- \varepsilon \ln\left(x_2-\bar{f}(x_1)\right)\right)\\
=&-\varepsilon \Phi \mu'(\cdot) \left(\frac{f_2'(x_1)}{f_2 (x_1)-x_2}+\frac{\bar{f}'(x_1)}{x_2-\bar{f}(x_1)}\right),
\end{aligned}
\end{equation}
where $\mu'(\cdot)=\mu'\left(1+\varepsilon \ln (f_2(x_1)-x_2)- \varepsilon \ln \left(x_2-\bar{f}(x_1) \right)\right)$.
Clearly, one has
\begin{equation}\label{1-11}
\begin{aligned}
	g_2=&g_1 f_2'(x_1)+\varepsilon \Phi \mu'(\cdot) \frac{f_2'(x_1)}{x_2-\bar{f}(x_1)}-\varepsilon \Phi \mu'(\cdot)\frac{\bar{f}'(x_1)}{x_2-\bar{f}(x_1)}\\
	=&g_1 f_2'(x_1)+\varepsilon \Phi \mu'(\cdot) \frac{ f'(x_1)}{2(x_2-\bar{f}(x_1))}.
\end{aligned}
\end{equation}

The following lemma gives some properties of the flux carrier $\Bg$, which play an important role in the construction of approximate solutions, especially when $\Phi$ is not small.
\begin{lemma}\label{lemma1}
For any function $w\in H^1(\Omega_{a,b})$ satisfying $w=0$ on the upper boundary $S_{2;a,b}:=\{\Bx\in \partial \Omega:\ x_2=f_2(x_1),~a<x_1<b\}$, it holds that
\begin{equation*}
\int_{\Omega_{a,b}} g_1^2w^2\,dx\leq C \Phi^2\varepsilon^2\int_{\Omega_{a,b}}|\partial_{x_2} w |^2\,dx.
\end{equation*}
 Moreover, if $f_i(i=1,2)$ satisfies \eqref{assumpf''}, then one has
\begin{equation*}
	|\Bg|\leq \frac{C(\varepsilon) \Phi}{f(x_1)},\ \ \ \ \ |\nabla\Bg|\leq \frac{C(\varepsilon, \gamma) \Phi}{f^2(x_1)},
\end{equation*}
and
\begin{equation*}
  \int_{\Omega_{a,b}}|\nabla \Bg|^2+|\Bg|^4\,dx\leq C(\epsilon, \gamma) ( \Phi^2 +  \Phi^4) \int_{a}^b f^{-3}(x_1)\,dx_1,
\end{equation*}
where $C(\varepsilon)$ is a constant depending on $\varepsilon$ and $C(\varepsilon,\gamma)$ depends on $\varepsilon$ and $\gamma$ with
\begin{equation}
\gamma=	\max_{i=1,2}\sup_{x_1\in \mathbb{R}} |f_i''(x_1)f(x_1)|.
\end{equation}

\end{lemma}
\begin{proof}
We use the definition, the property \eqref{1-6}, and Hardy inequality (\hspace{1sp}\cite{HLP}) to obtain
\begin{equation*}
	\begin{aligned}
\int_{\Omega_{a,b}} g_1^2w^2\,dx=&\int_{\Omega_{a,b}} \varepsilon^2\Phi^2( \mu'(\cdot))^2 \left(\frac{-1}{f_2(x_1)-x_2}-\frac{1}{x_2-\bar{f}(x_1)}\right)^2w^2\,dx\\
\leq &C\Phi^2\varepsilon^2\int_a^b\,dx_1\int_{\bar{f}(x_1)}^{f_2(x_1)}\frac{w^2}{(f_2(x_1)-x_2)^2}\,dx_2\\
\leq&C\Phi^2\varepsilon^2\int_{\Omega_{a,b}}|\partial_{x_2}w|^2\,dx.
\end{aligned}\end{equation*}
On the other hand,  it follows from \eqref{1-6}-\eqref{1-11-2} that
\begin{equation}\label{1-12}
	|\Bg|\leq \frac{C (\varepsilon) \Phi }{f(x_1)}.
\end{equation}
Moreover, direct computations give that
\begin{equation}\label{1-13}
	\begin{aligned}
	&|\partial_{x_1}g_1|=|\partial_{x_2}g_2|\\
\leq &\left|\varepsilon^2 \Phi \mu^{\prime\prime} (\cdot) \left(\frac{-1}{f_2(x_1)-x_2}-\frac{1}{x_2-\bar{f}(x_1)}\right)\left(\frac{f_2'(x_1)}{f_2(x_1)-x_2}+\frac{\bar{f}'(x_1)}{x_2-\bar{f}(x_1)}\right)\right|\\
	&+\left|\varepsilon \Phi \mu' (\cdot) \left(\frac{f_2'(x_1)}{(f_2(x_1)-x_2)^2}-\frac{\bar{f}'(x_1)}{\left(x_2-\bar{f}(x_1)\right)^2}\right)\right|\\
	\leq& \frac{C (\varepsilon) \Phi}{f^2(x_1)}
	\end{aligned}
\end{equation}
and
\begin{equation}\label{1-14}
	\begin{aligned}
|\partial_{x_2}g_1|
\leq &\left|\varepsilon^2 \Phi \mu^{\prime\prime}(\cdot) \left(\frac{-1}{f_2(x_1)-x_2}-\frac{1}{x_2-\bar{f}(x_1)}\right)^2\right|\\
	&+\left|\varepsilon \Phi \mu'(\cdot) \left(\frac{1}{(f_2(x_1)-x_2)^2}+\frac{1}{\left(x_2-\bar{f}(x_1)\right)^2}\right)\right|\\
	\leq& \frac{C(\varepsilon) \Phi }{f^2(x_1)}.
	\end{aligned}
\end{equation}
Furthermore, one has
\begin{equation}\label{1-15}
\begin{aligned}
&|\partial_{x_1}g_2|
\leq \left|\varepsilon^2 \Phi \mu^{\prime\prime}(\cdot)\left(\frac{f_2'(x_1)}{f_2(x_1)-x_2}+\frac{\bar{f}'(x_1)}{x_2-\bar{f}(x_1)}\right)^2\right|\\
	&+\left|\varepsilon \Phi \mu'(\cdot) \left(\frac{f_2^{\prime\prime} (x_1)}{f_2(x_1)-x_2}-\frac{[f_2'(x_1)]^2}{(f_2(x_1)-x_2)^2}+\frac{\bar{f}''(x_1)}{x_2-\bar{f}(x_1)}+\frac{\left[\bar{f}'(x_1)\right]^2}{(x_2-\bar{f}(x_1))^2}\right)\right|\\
	\leq& C\varepsilon \Phi\left(\frac{1}{f^2(x_1)}+ \left|\frac{f_2''(x_1)}{f_2(x_1)-x_2}\right|+ \left|\frac{\bar{f}''(x_1)}{x_2-\bar{f}(x_1)}\right|\right)\\
	\leq& C(\varepsilon ) \Phi \left(\frac{1}{f^2(x_1)}+\frac{|f_1''(x_1)|}{f(x_1)}+\frac{|f_2''(x_1)|}{f(x_1)}\right)\\
	\leq&\frac{C(\varepsilon, \gamma) \Phi}{f^2(x_1)},
	\end{aligned}
\end{equation}
where the last inequality follows from the definition of $\gamma$. Then combining \eqref{1-12}-\eqref{1-15} gives
\begin{equation*}
\begin{aligned}
	\int_{\Omega_{a,b}} |\nabla \Bg|^2+|\Bg|^4\,dx\leq &C(\varepsilon,\gamma)(\Phi^2 +\Phi^4) \int_{a}^b\int_{\Sigma(x_1)} f^{-4}(x_1)\,dx_2dx_1 \\
	\leq &C(\varepsilon,\gamma)(\Phi^2 + \Phi^4)\int_{a}^b f^{-3}(x_1)\,dx_1.
\end{aligned}
\end{equation*}
This finishes the proof of the lemma.
\end{proof}


Given the flux carrier $\Bg$ constructed in \eqref{defg}, if $\Bu$ satisfies \eqref{NS}, \eqref{flux constraint} and \eqref{BC}, then $\Bv=\Bu-\Bg$ satisfies
\begin{equation}\label{NS1}
\left\{
\begin{aligned}
&-\Delta \Bv+\Bv\cdot \nabla \Bg +\Bg\cdot \nabla \Bv+\Bv\cdot \nabla \Bv  +\nabla p=\Delta \Bg-\Bg\cdot \nabla\Bg   ~~~~&\text{ in }\Omega,\\
&{\rm div}~\Bv=0&\text{ in }\Omega,\\
&\Bv\cdot \Bn=0,~~(\Bn\cdot \BD(\Bv)+\alpha \Bv)\cdot \Bt=0&\text{ on }\partial\Omega,\\
&\int_{\Sigma(x_1)} \Bv\cdot \Bn \,ds=0&\text{ for any }x_1\in \mathbb{R}.
\end{aligned}\right.
\end{equation}

The weak solutions of \eqref{NS1} is defined as follows.
\begin{definition}
A vector field $\Bv\in H_\sigma(\Omega)$ is said to be a weak solution of the problem \eqref{NS1} if for any $\Bp\in \mcH_\sigma(\Omega_T)$ with $T>0$, one has
\begin{equation*}
\begin{aligned}
&\int_{\Omega}2\BD(\Bv):\BD(\Bp)+(\Bv\cdot \nabla \Bg +(\Bg+\Bv)\cdot\nabla \Bv)\cdot \Bp\,dx+2\alpha \int_{\partial\Omega}\Bv\cdot \Bp\,ds \\
=&\int_{\Omega}(\Delta \Bg-\Bg\cdot \nabla \Bg)\cdot \Bp\,dx.
\end{aligned}
\end{equation*}

\end{definition}

In the rest of this section, we study the following approximate problems of \eqref{NS1} on the bounded domain $\Omega_{a,b}$,
\begin{equation}\label{aNS}
\left\{
\begin{aligned}
&-\Delta \Bv+\Bv\cdot \nabla \Bg +\Bg\cdot \nabla \Bv+\Bv\cdot \nabla \Bv  +\nabla p=\Delta \Bg-\Bg\cdot \nabla \Bg   ~~~~&\text{ in }\Omega_{a,b},\\
&{\rm div}~\Bv=0&\text{ in }\Omega_{a,b},\\
&\Bv\cdot \Bn=0,~(\Bn\cdot \BD(\Bv)+\alpha \Bv)\cdot \Bt=0&\text{ on }\partial\Omega_{a,b}\cap \partial\Omega,\\
&\Bv=0&\text{ on }\Sigma(a)\cup\Sigma(b)
\end{aligned}\right.
\end{equation}
and its linearized problem
\begin{equation}\label{laNS}
\left\{\begin{aligned}
&-\Delta \Bv+\Bv\cdot \nabla \Bg +\Bg\cdot \nabla \Bv +\nabla p=\Delta \Bg-\Bg\cdot \nabla \Bg   ~~~~&\text{ in }\Omega_{a,b},\\
&{\rm div}~\Bv=0&\text{ in }\Omega_{a,b},\\
&\Bv\cdot \Bn=0,~(\Bn\cdot \BD(\Bv)+\alpha \Bv)\cdot \Bt=0&\text{ on }\partial\Omega_{a,b}\cap \partial\Omega,\\
&\Bv=0&\text{ on }\Sigma(a)\cup\Sigma(b).\end{aligned}\right.\end{equation}

The weak solutions of problems \eqref{aNS} and \eqref{laNS} can be defined as follows.
\begin{definition}
A vector field $\Bv\in \mcH_\sigma(\Omega_{a,b})$ is a weak solution of the problem \eqref{aNS} and \eqref{laNS}, respectively
if for any  $\Bp\in \mcH_\sigma(\Omega_{a,b})$, $\Bv$ satisfies
\begin{equation}\label{2-1}
\begin{aligned}
&\int_{\Omega_{a,b}}2\BD(\Bv):\BD(\Bp)+(\Bv\cdot \nabla \Bg +(\Bg+\Bv)\cdot \nabla \Bv)\cdot \Bp\,dx+2\alpha \int_{\partial\Omega_{a,b}\cap \partial\Omega}\Bv\cdot \Bp\,ds \\
=&\int_{\Omega_{a,b}}(\Delta \Bg-\Bg\cdot \nabla \Bg)\cdot \Bp\,dx
\end{aligned}
\end{equation}
and
\begin{equation}\label{2-2}\begin{aligned}
&\int_{\Omega_{a,b}}2\BD(\Bv):\BD(\Bp)+(\Bv\cdot \nabla \Bg +\Bg\cdot \nabla \Bv)\cdot \Bp\,dx+2\alpha \int_{\partial\Omega_{a,b}\cap \partial\Omega}\Bv\cdot \Bp\,ds\\
=&\int_{\Omega_{a,b}}(\Delta \Bg-\Bg\cdot \nabla \Bg)\cdot \Bp\,dx,
\end{aligned}\end{equation}respectively.
\end{definition}

Next, we use Leray-Schauder fixed point theorem (cf. \cite[Theorem 11.3]{GT}) to prove the existence of solutions to the approximate  problem \eqref{aNS}. To this end,  the existence of solutions to  the linearized problem \eqref{laNS} is first established by the following lemma.
\begin{lemma}\label{lemma5}
For any $\Bh\in L^\frac43(\Omega_{a,b})$, there exists a unique $\Bv\in \mcH_\sigma(\Omega_{a,b})$ such that for any $\Bp\in \mcH_\sigma(\Omega_{a,b})$, it holds that
\begin{equation}\label{2-12}
\int_{\Omega_{a,b}}2\BD(\Bv):\BD(\Bp)+(\Bv\cdot \nabla \Bg +\Bg\cdot \nabla \Bv)\cdot \Bp\,dx+2\alpha \int_{\partial\Omega_{a,b}\cap \partial\Omega}\Bv\cdot \Bp\,ds
=\int_{\Omega_{a,b}}\Bh\cdot \Bp\,dx.
\end{equation}
\end{lemma}
\begin{proof} The proof is based on Lax-Milgram theorem and is divided into two steps.

{\em Step 1. Bilinear functional.} For any $\Bv,\Bw\in \mcH_\sigma(\Omega_{a,b})$, define the bilinear functional on $\mcH_\sigma(\Omega_{a,b})$ as follows
\begin{equation*}
B[\Bv,\Bw]=\int_{\Omega_{a,b}}2\BD(\Bv):\BD(\Bw)+(\Bv\cdot \nabla \Bg +\Bg\cdot \nabla \Bv)\cdot \Bw\,dx+2\alpha \int_{\partial\Omega_{a,b}\cap \partial\Omega}\Bv\cdot \Bw\,ds.
\end{equation*}

 Using the trace theorem and H\"{o}lder inequality yields
\begin{equation}\label{2-1-1}
|B[\Bv,\Bw]|\leq C\|\Bv\|_{H^1(\Omega_{a,b})}\|\Bw\|_{H^1(\Omega_{a,b})}.
\end{equation}
It follows from Lemma \ref{lemmaA3} that one has
\begin{equation}\label{2-4}
2\int_{\Omega_{a,b}}|\BD(\Bv)|^2\,dx+\alpha \int_{\partial\Omega_{a,b}\cap \partial\Omega}|\Bv|^2\,ds\ge \mfc\|\nabla \Bv\|_{L^2(\Omega_{a,b})}^2
\end{equation}
where $\mfc$ is given in \eqref{defmfc}. Moreover, using integration by parts gives
\begin{equation}\label{2-5}
\int_{\Omega_{a,b}}\Bg\cdot \nabla \Bv\cdot \Bv\,dx=0
\end{equation}
and
\begin{equation}\label{2-6} \begin{aligned}
\int_{\Omega_{a,b}}\Bv\cdot \nabla \Bg\cdot \Bv\,dx & =-\int_{\Omega_{a,b}}\Bv\cdot\nabla \Bv\cdot \Bg\,dx \\
& = - \int_{\Omega_{a, b}} (v_1 \partial_{x_1} + v_2 \partial_{x_2}) v_1 g_1 \, dx - \int_{\Omega_{a, b}} (v_1 \partial_{x_1} + v_2 \partial_{x_2}) v_2 g_2 \, dx .
\end{aligned} \end{equation}
Recalling the equality \eqref{1-11} and using integration by parts yield
\begin{equation}\label{2-6-1}
\begin{aligned}
&\int_{\Omega_{a,b}}(v_1\partial_{x_1}+v_2\partial_{x_2})v_1 g_1\,dx\\
=&\int_{\Omega_{a,b}} \left( v_1\partial_{x_1}v_1g_1+v_1\partial_{x_2}v_1 g_2 \right) \,dx+\int_{\Omega_{a,b}} \left( v_2\partial_{x_2}v_1g_1-v_1\partial_{x_2}v_1 g_2 \right)\,dx\\
=&-\int_{\Omega_{a,b}}\frac12 (v_1)^2(\partial_{x_1}g_1+\partial_{x_2}g_2)\,dx+\int_{\Omega_{a,b}}(v_2g_1-v_1 g_2)\partial_{x_2}v_1\,dx\\
=&\int_{\Omega_{a,b}}(v_2-v_1 f_2')g_1\partial_{x_2}v_1\,dx-\int_{\Omega_{a,b}}\varepsilon \Phi \mu' (\cdot) \frac{\frac12 f'(x_1)}{x_2-\bar{f}(x_1)} v_1\partial_{x_2}v_1\,dx
\end{aligned}\end{equation}
and
\begin{equation}\label{2-6-2}
\begin{aligned}
&\int_{\Omega_{a,b}}(v_1\partial_{x_1}+v_2\partial_{x_2})v_2g_2\,dx\\
=&\int_{\Omega_{a,b}}(v_1\partial_{x_1}v_2g_2-v_2\partial_{x_1}v_2g_1)\,dx+\int_{\Omega_{a,b}}(v_2\partial_{x_1}v_2g_1+v_2\partial_{x_2}v_2g_2)\,dx\\
=&\int_{\Omega_{a,b}}(v_1g_2-v_2g_1)\partial_{x_1}v_2\,dx-\int_{\Omega_{a,b}}\frac12 (v_2)^2(\partial_{x_1}g_1+\partial_{x_2}g_2)\,dx\\
=&\int_{\Omega_{a,b}}(v_1 f_2'-v_2)g_1\partial_{x_1}v_2\,dx+\int_{\Omega_{a,b}}\varepsilon \Phi \mu'(\cdot) \frac{\frac12 f'(x_1)}{x_2-\bar{f}(x_1)} v_1\partial_{x_1}v_2\,dx.
\end{aligned}
\end{equation}
It follows from \eqref{1-6}-\eqref{1-7} that one has
\begin{equation}\label{2-7-1}
\begin{aligned}
\left|\int_{\Omega_{a,b}}\Bv\cdot\nabla \Bv\cdot \Bg\,dx\right|\leq &\left|\int_{\Omega_{a,b}}(v_2-v_1f_2')(\partial_{x_1}v_2-\partial_{x_2}v_1)g_1\,dx\right|\\
&+\left|\int_{\Omega_{a,b}}\varepsilon \Phi \mu' (\cdot) \frac{ \frac12 f'(x_1)}{x_2-\bar{f}(x_1)} v_1(\partial_{x_1}v_2-\partial_{x_2}v_1)\,dx\right|\\
\leq &\int_{\Omega_{a,b}}|(v_2-v_1f_2')(\partial_{x_1}v_2-\partial_{x_2}v_1)g_1|\,dx\\
&+C\varepsilon \Phi\int_{\Omega_{a,b}}\left| \frac{v_1}{f}(\partial_{x_1}v_2-\partial_{x_2}v_1)\right|\,dx.
\end{aligned}
\end{equation}
Note that on the upper boundary $S_{2;a,b}=\{\Bx:~a<x_1<b,~x_2=f_2(x_1)\}$, the impermeability condition $\Bv\cdot\Bn=0$ can be written as
\begin{equation*}
v_2(x_1,f_2(x_1))-f_2'(x_1)v_1(x_1,f_2(x_1))=0.
\end{equation*}
Then one uses Lemmas \ref{lemmaA1} and \ref{lemma1} to obtain
\begin{equation}\label{2-7}
\begin{aligned}
\left|\int_{\Omega_{a,b}}\Bv\cdot\nabla \Bv\cdot \Bg\,dx\right|
\leq&C\varepsilon \Phi \|\nabla \Bv\|_{L^2(\Omega_{a,b})} \|\partial_{x_2}(v_2-f_2'v_1)\|_{L^2(\Omega_{a,b})}\\
&+C\varepsilon \Phi \|\nabla \Bv\|_{L^2(\Omega_{a,b})} \left\|\frac{v_1}{f}\right\|_{L^2(\Omega_{a,b})}^2\\
\leq&C\varepsilon \Phi\|\nabla\Bv\|_{L^2(\Omega_{a,b})}^2.
\end{aligned}
\end{equation}
Choosing sufficiently small $\varepsilon$ such that $C\varepsilon \Phi <\frac 12 \mfc$ gives
\begin{equation}\label{2-8}
	\begin{aligned}
B[\Bv,\Bv] =&\int_{\Omega_{a,b}}2|\BD(\Bv)|^2+(\Bv\cdot \nabla \Bg +\Bg\cdot \nabla \Bv)\cdot \Bv\,dx+2\alpha \int_{\partial\Omega_{a,b}\cap \partial\Omega}|\Bv|^2\,ds \\
\geq& \frac \mfc2 \|\nabla\Bv\|_{L^2(\Omega_{a,b})}^2+\alpha\|\Bv\|_{L^2(\partial\Omega\cap \partial\Omega_{a,b})}^2.
\end{aligned}
\end{equation}
Therefore, by Lemma \ref{lemmaA1}, one has
\begin{equation}\label{2-9}
B[\Bv,\Bv]\ge \frac{c}{2(1+M_1(\Omega_{a,b})^2)}\|\Bv\|_{H^1(\Omega_{a,b})}^2.
\end{equation}

{\em Step 2. Existence of weak solution. } 
For any $\Bp\in \mcH_\sigma(\Omega_{a,b})$, one uses H\"{o}lder inequality and Lemma \ref{lemmaA2} to obtain
\begin{equation}\label{2-10}
	\left|\int_{\Omega_{a,b}}\Bh\cdot \Bp\,dx\right|\leq \|\Bh\|_{L^\frac43(\Omega_{a,b})}\|\Bp\|_{L^4(\Omega_{a,b})}\leq C\|\Bh\|_{L^\frac43(\Omega_{a,b})}\|\nabla \Bp\|_{L^2(\Omega_{a,b})}.
\end{equation}
Then the Lax-Milgram theorem, together with \eqref{2-1-1} and \eqref{2-9}-\eqref{2-10}, shows that there exists a unique $\Bv\in \mcH_\sigma(\Omega_{a,b})$ such that for any $\Bp\in \mcH_\sigma(\Omega_{a,b})$, it holds that 
\begin{equation*}
	B[\Bv,\Bp]=\int_{\Omega_{a,b}}\Bh\cdot \Bp\,dx.
\end{equation*} 
This finishes the proof of the lemma.
\end{proof}

Note that $\Bg\in C^2(\bar{\Omega})$. Hence one has   $\Delta \Bg-\Bg\cdot \nabla \Bg \in L^\frac43(\Omega_{a,b})$. Therefore, the existence of solutions to the linearized problem \eqref{laNS} is a consequence of Lemma \ref{lemma5}.
\begin{cor} For any $a<b$, the linearized problem \eqref{laNS} admits a unique solution $\Bv\in \mcH_\sigma(\Omega_{a,b})$. 
\end{cor}

Now we are ready to prove  the existence of solutions for the approximate problem \eqref{aNS}.
\begin{pro}\label{appro-existence}
For any $a<b$, the problem \eqref{aNS} has a weak solution $\Bv\in \mcH_\sigma(\Omega_{a,b})$ satisfying
\begin{equation}\label{2-10-1}
\|\nabla \Bv\|_{L^2(\Omega_{a,b})}^2+\alpha\|\Bv\|_{L^2(\partial\Omega_{a,b}\cap\partial\Omega)}^2\leq C_0\int_{\Omega_{a,b}} |\nabla \Bg|^2+|\Bg|^4\,dx,
\end{equation}
where the constant $C_0$ is independent of $a$ and $b$.
\end{pro}
\begin{proof}
Lemma \ref{lemma5} defines a map $\mathcal{T}$ which maps $\Bh\in L^\frac43(\Omega)$ to $\Bv\in \mcH_\sigma(\Omega_{a,b})$. For any $\Bw\in \mcH_\sigma(\Omega_{a,b})$, using H\"{o}lder inequality and Lemma \ref{lemmaA2} gives   
\begin{equation*}
\|\Bw\cdot\nabla \Bw\|_{L^\frac43}\leq \|\Bw\|_{L^4(\Omega)}\|\nabla\Bw\|_{L^2(\Omega)}\leq C\|\nabla\Bw\|_{L^2(\Omega_{a,b})}^2.
\end{equation*}
Hence, $\Bh=\Delta \Bg-\Bg\cdot\nabla \Bg-\Bw\cdot\nabla\Bw \in L^\frac43(\Omega_{a,b})$ and one could  define the map 
	\begin{equation*}
		K(\Bw):=\mathcal{T}(\Delta \Bg-\Bg\cdot\nabla \Bg-\Bw\cdot\nabla\Bw)
	\end{equation*}
	from $\mcH_\sigma(\Omega_{a,b})$ to $\mcH_\sigma(\Omega_{a,b})$. Solving the problem \eqref{aNS} is transformed to finding a fixed point for
\begin{equation*}
K(\Bv)=\Bv.
\end{equation*}

In order to apply Leray-Schauder fixed point theorem, we show that $K:\, \mcH_\sigma(\Omega_{a,b})\to \mcH_\sigma(\Omega_{a,b})$ is continuous and compact. First, for any $\Bv^1,\Bv^2\in \mcH_\sigma(\Omega_{a,b})$, integration by parts yields
\begin{equation*}
\begin{aligned}
&\left|\int_{\Omega_{a,b}}(\Bv^1\cdot\nabla \Bv^1-\Bv^2\cdot\nabla \Bv^2)\cdot \Bp\,dx\right|\\
=&\left|\int_{\Omega_{a,b}}\Bv^1\cdot\nabla \Bp\cdot \Bv^1-\Bv^2\cdot\nabla\Bp\cdot \Bv^2\,dx\right|\\
=&\left|\int_{\Omega_{a,b}}\Bv^1\cdot\nabla \Bp\cdot (\Bv^1-\Bv^2)+(\Bv^2-\Bv^1)\cdot\nabla \Bp\cdot \Bv^2\,dx\right|\\
\leq& C(\|\Bv^1\|_{L^4(\Omega_{a,b})}+\|\Bv^2\|_{L^4(\Omega_{a,b})})\|\Bv^1-\Bv^2\|_{L^4(\Omega_{a,b})}\|\Bp\|_{H^1(\Omega_{a,b})}.
\end{aligned}
\end{equation*}
Hence it holds that
\begin{equation*}
\begin{aligned}
\|K(\Bv^1)-K(\Bv^2)\|_{H^1(\Omega_{a,b})}\leq&C\|\mathcal{T}(\Bv^1\cdot\nabla \Bv^1-\Bv^2\cdot\nabla \Bv^2)\|_{H^1(\Omega_{a,b})}\\
\leq&C(\|\Bv^1\|_{L^4(\Omega_{a,b})}+\|\Bv^2\|_{L^4(\Omega_{a,b})})\|\Bv^1-\Bv^2\|_{L^4(\Omega_{a,b})}.
\end{aligned}
\end{equation*}
This implies that $K$ is a continuous map from $\mcH_{\sigma}(\Omega_{a,b})$ into itself. Moreover, the compactness of $K$ follows from the compactness of the Sobolev embedding $H^1(\Omega_{a,b}) \hookrightarrow L^4(\Omega_{a,b})$.

Finally, if $\Bv\in \mcH_\sigma(\Omega_{a,b})$ satisfies $\Bv=\sigma K(\Bv)$ with $\sigma\in[0,1]$, then for any $\Bp\in \mcH_\sigma(\Omega_{a,b})$,
\begin{equation*}
\begin{aligned}
&\int_{\Omega_{a,b}}2\BD(\Bv):\BD(\Bp)+(\Bv\cdot \nabla \Bg +\Bg\cdot \nabla \Bv)\cdot \Bp\,dx+2\alpha \int_{\partial\Omega_{a,b}\cap \partial\Omega}\Bv\cdot \Bp\,ds \\
=&\sigma \int_{\Omega_{a,b}}(\Delta \Bg-\Bg\cdot \nabla \Bg-\Bv\cdot \nabla \Bv)\cdot \Bp\,dx.
\end{aligned}
\end{equation*}
Taking $\Bp=\Bv$, it holds that 
\begin{equation}\label{2-14}
\begin{aligned}
&\int_{\Omega_{a,b}}2|\BD(\Bv)|^2+(\Bv\cdot \nabla \Bg +\Bg\cdot \nabla \Bv)\cdot \Bv\,dx+2\alpha \int_{\partial\Omega_{a,b}\cap \partial\Omega}|\Bv|^2\,ds \\
=&\sigma \int_{\Omega_{a,b}}(\Delta \Bg-\Bg\cdot \nabla \Bg-\Bv\cdot \nabla \Bv)\cdot \Bv\,dx.
\end{aligned}
\end{equation}
Noting that $\Bg=0$ near the boundary $\partial \Omega$, $\Bv\cdot \Bn=0$ on $\partial\Omega\cap \partial\Omega_{a,b}$, and $\Bv=0$ on $\Sigma(a)\cup\Sigma(b)$, one uses integration by parts to obtain 
\begin{equation}\label{2-13}
\begin{aligned}
&\left|\int_{\Omega_{a,b}}(\Delta \Bg-\Bg\cdot \nabla \Bg-\Bv\cdot \nabla \Bv)\cdot \Bv\,dx\right|\\
=&\left|\int_{\Omega_{a,b}}-\nabla\Bg:\nabla \Bv+\Bg\cdot\nabla \Bv\cdot \Bg\,dx\right|\\
\leq& \left(\int_{\Omega_{a, b}} |\nabla \Bg|^2+|\Bg|^4\,dx\right)^\frac12\|\nabla \Bv\|_{L^2(\Omega_{a,b})}.
\end{aligned}
\end{equation}
This, together with \eqref{2-9} and \eqref{2-14}, gives
\begin{equation*}
	\|\Bv\|_{H^1(\Omega_{a,b})}^2\leq C\int_{\Omega_{a,b}}|\nabla \Bg|^2+|\Bg|^4\,dx.
\end{equation*}
Then Leray-Schauder fixed point theorem shows that there exists a solution $\Bv \in \mcH_\sigma(\Omega_{a,b})$ of the problem $\Bv=K(\Bv)$.

Moreover, combining \eqref{2-8} and \eqref{2-14}-\eqref{2-13} yields 
\begin{equation*}
\begin{aligned}
\frac{\mfc}{2}\|\nabla\Bv\|_{L^2(\Omega_{a,b})}^2+\alpha\|\Bv\|_{L^2(\partial\Omega_{a,b}\cap\partial\Omega)}^2
\leq& 2 \left(\int_{\Omega_{a, b}} |\nabla \Bg|^2+|\Bg|^4\,dx\right)^\frac12\|\nabla\Bv\|_{L^2(\Omega_{a,b})}\\
\leq & \frac{4}{\mfc}\int_{\Omega_{a, b}} |\nabla \Bg|^2+|\Bg|^4\,dx+ \frac{\mfc}{4}\|\nabla\Bv\|_{L^2(\Omega_{a,b})}^2.
\end{aligned}
\end{equation*}
Hence one has \eqref{2-10-1} and the proof of the proposition is completed.
\end{proof}

As long as the existence of weak solution $\Bv$ for the problem \eqref{2-1} is established, one can further obtain the associated pressure for \eqref{aNS} with the aid of the following lemma,  whose proof can be found in \cite[Theorem \uppercase\expandafter{\romannumeral3}.5.3]{Ga}.
\begin{pro}\label{pressure}
The vector field $\Bv\in \mcH_\sigma(\Omega_{a,b})$ is a weak solution of the approximate problem \eqref{aNS} if and only if there exists a function $p\in L^2(\Omega_{a,b})$ such that the identity
\begin{equation}\label{2-11}
\begin{aligned}
\int_{\Omega_{a,b}}2\BD(\Bv):\BD(\Bp)+(\Bv\cdot \nabla \Bg +(\Bg+\Bv)\cdot \nabla \Bv)\cdot \Bp\,dx&\\
-\int_{\Omega_{a,b}}p{\rm div}\Bp\,dx+2\alpha \int_{\partial\Omega_{a,b}\cap \partial\Omega}\Bv\cdot \Bp\,ds &=\int_{\Omega_{a,b}}(\Delta \Bg-\Bg\cdot \nabla \Bg)\cdot \Bp\,dx
\end{aligned}
\end{equation}
holds for any $\Bp\in \mcH(\Omega_{a,b})$ which is defined in Definition \ref{def1}.
\end{pro}

\section{Flows in channels with bounded outlets}\label{secexist1}
In this section, we investigate the flows in channels with bounded outlets. First, we establish some uniform estimates for the approximate solutions obtained in Proposition \ref{appro-existence}. Some bounds for the Dirichlet norm of the approximate solutions in $\Omega_T$ have been already obtained in Section \ref{secflux}. A key estimate obtained in this section is the uniform bound for approximate solutions in any fixed subdomain. Our strategy is to establish a differential inequality for the Dirichlet norm of these approximate solutions, which characterizes its growth rate and then gives a uniform estimate.

\begin{lemma}\label{lemma3}
Assume that $\Omega$ is a channel with bounded outlets, i.e., $f$ satisfies \eqref{assumpf}. Let $\Bv^T$ be the solution of the approximate problem  \eqref{aNS} in $\Omega_T$, which is obtained in Proposition \ref{appro-existence}. Then one has
\begin{equation}\label{3-0}
\|\nabla\Bv^T\|_{L^2(\Omega_t)}^2+\alpha\|\Bv^T\|_{L^2(\partial\Omega_t\cap \partial\Omega)}^2\leq C_3+C_4t \ \ \ \text{ for any }1< t\leq T-1,
\end{equation}
where the constants $C_3$ and $C_4$ are independent of $t$ and $T$.
\end{lemma}
\begin{proof} For convenience, we omit the superscript $T$ throughout the proof and divide the proof into three steps.

{\em Step 1. Truncating function.}   First, let $\zeta(\Bx,t)$ be a smooth function for $\Bx\in \Omega$ satisfying
\begin{equation*}
\zeta(\Bx,t)=\left\{
\begin{aligned}
&1,~~~~~~~~~~~~~~ &&\text{ if }x_1\in (-t+1,t-1),\\
&0,~~~~~~~~~~~~~~ &&\text{ if }x_1\in (-\infty,-t)\cup(t,\infty),\\
&t-x_1,~~~~~~&&\text{ if }x_1\in [t-1,t],\\
&t+x_1,~~~~~~&&\text{ if }x_1\in [-t,-t+1].
\end{aligned}\right.
\end{equation*}
Clearly, $\zeta$ depends only on $t,x_1$ and satisfies
\[
|\partial_{x_1}\zeta|=|\partial_t\zeta| =1\quad \text{in}\,\, E=E^+\cup E^-,
\] where
\begin{equation}\label{defE}
 E^-=\{\Bx\in\Omega:x_1\in (-t,-t+1)\}\text{ and } E^+=\{\Bx\in\Omega:x_1\in (t-1,t)\}.
\end{equation}
Now for any  $1\leq t\leq T$, taking the test function $\Bp=\zeta\Bv$ in \eqref{2-11} yields
\begin{equation}\label{3-1}
\begin{aligned}
\int_{\Omega_T}2\BD(\Bv):\BD(\zeta\Bv)+(\Bv\cdot \nabla \Bg +(\Bg+\Bv)\cdot\nabla \Bv)\cdot (\zeta\Bv)\,dx&\\
-\int_{\Omega_T}p{\rm div}\,(\zeta \Bv)\,dx+2\alpha \int_{\partial\Omega_T\cap \partial\Omega}\zeta
\Bv^2\,ds =&\int_{\Omega_T}(\Delta \Bg-\Bg\cdot \nabla \Bg)\cdot (\zeta\Bv)\,dx.
\end{aligned}
\end{equation}
Similar to the proof of Lemma \ref{lemmaA3}, for any $\Bv\in \mathcal{C}_\sigma(\Omega_{T})$, one has
\begin{equation}\label{3-3-1}
\begin{aligned}
&\int_{\Omega_T }2 \BD(\Bv):\BD(\zeta \Bv)\,dx-\int_{\partial\Omega_T\cap \partial \Omega } 2\zeta \Bn\cdot\BD(\Bv)\cdot\Bv  \,ds = \int_{\Omega_T} -2{\rm div}\BD(\Bv)\cdot(\zeta \Bv)\,dx\\
= &\int_{\Omega_T }-\Delta \Bv\cdot(\zeta \Bv)\,dx=\int_{\Omega_T }\zeta |\nabla\Bv|^2+\partial_{x_1}\zeta \partial_{x_1}\Bv\cdot \Bv\,dx-\int_{\partial\Omega_T \cap \partial\Omega}\zeta \Bn\cdot \nabla \Bv\cdot \Bv \, ds.
\end{aligned}
\end{equation}
Therefore, it holds that
	\begin{equation*}
	\begin{aligned}
		\int_{\Omega_T }\zeta |\nabla\Bv|^2\,dx=&\int_{\Omega_T }2\BD(\Bv):\BD(\zeta \Bv)-\partial_{x_1}\zeta \partial_{x_1}\Bv\cdot \Bv\,dx\\
		&-  \int_{\partial\Omega_T\cap \partial \Omega }\zeta  [2\Bn \cdot \BD(\Bv)  \cdot \Bv  -  \Bn\cdot \nabla \Bv\cdot \Bv]\,ds.
	\end{aligned}
	\end{equation*}
This, together with \eqref{A3-3.5} and \eqref{A3-4},  yields
	\begin{equation*}
	\int_{\Omega_T}\zeta |\nabla\Bv|^2\,dx\leq \int_{\Omega_T }2\BD(\Bv):\BD(\zeta \Bv)\,dx+\|\partial_\tau \Bn\|_{L^\infty(\partial\Omega)}\int_{\partial\Omega_T\cap \partial\Omega}\zeta |\Bv|^2\,ds-\int_{E}\partial_{x_1}\zeta \partial_{x_1}\Bv\cdot \Bv\,dx.
	\end{equation*}
	Hence one has
\begin{equation}\label{3-3}
\begin{aligned}
 &\mfc\int_{\Omega_T}\zeta|\nabla\Bv|^2\,dx
\leq  \int_{\Omega_T }2\BD(\Bv):\BD(\zeta \Bv)\,dx + \alpha\int_{\partial\Omega_T \cap \partial \Omega}\zeta|\Bv|^2\,ds+ \left|\int_{ E}\partial_{x_1}\zeta \partial_{x_1}\Bv\cdot \Bv\,dx\right|
\end{aligned}
\end{equation}
with $\mfc$ defined in \eqref{defmfc}.
Since $\mathcal{C}_\sigma(\Omega_T)$ is dense in $H_\sigma(\Omega_T)$, the estimate \eqref{3-3} also holds for any $\Bv\in \mcH_\sigma(\Omega_T)$.

Next, using integration by parts gives
\begin{equation}\label{3-4-1}
-\int_{\Omega_T}\Bv\cdot\nabla \Bg\cdot (\zeta \Bv)\,dx
=\int_{\Omega_T}\zeta\Bv\cdot\nabla \Bv\cdot \Bg\,dx+\int_{ E}(\Bv\cdot \Bg)v_1\partial_{x_1}\zeta\,dx.
\end{equation}
Similar to \eqref{2-6-1}-\eqref{2-7}, one uses integration by parts and \eqref{1-11} to obtain
\begin{equation*}
\begin{aligned}
&\int_{\Omega_T}\zeta(v_1\partial_{x_1}+v_2\partial_{x_2})v_1g_1\,dx\\
=&\int_{\Omega_T}\zeta(v_1\partial_{x_1}v_1 g_1+v_1\partial_{x_2}v_1 g_2)\,dx+\int_{\Omega_T}\zeta(v_2\partial_{x_2}v_1g_1-v_1\partial_{x_2}v_1 g_2)\,dx\\
=&\int_{\Omega_T}\frac12 \zeta(\partial_{x_1}(v_1^2) g_1+\partial_{x_2}(v_1^2) g_2)\,dx+\int_{\Omega_T}\zeta(v_2g_1-v_1g_2)\partial_{x_2}v_1\,dx\\
=&-\int_{\Omega_T}\frac12 \zeta v_1^2 (\partial_{x_1}g_1+\partial_{x_2}g_2)\,dx-\int_{\Omega_T}\frac12 \partial_{x_1}\zeta v_1^2 g_1\,dx\\
&+\int_{\Omega_T}\zeta\left(v_2g_1-v_1g_1 f_2'-v_1 \frac{\varepsilon \Phi \mu'(\cdot) f'(x_1) }{2(x_2-\bar{f}(x_1))} \right)\partial_{x_2}v_1\,dx\\
=&-\int_{ E}\frac12v_1^2 g_1\partial_{x_1}\zeta\,dx+\int_{\Omega_T}\zeta(v_2-f_2'v_1)\partial_{x_2}v_1g_1\,dx- \int_{\Omega_{T}}\frac{\varepsilon \Phi \mu'(\cdot) f'(x_1) }{2(x_2-\bar{f}(x_1))} \zeta v_1 \partial_{x_2}v_1\,dx
\end{aligned}\end{equation*}
and
\begin{equation*}
\begin{aligned}
&\int_{\Omega_T}\zeta(v_1\partial_{x_1}+v_2\partial_{x_2})v_2g_2\,dx\\
=&\int_{\Omega_T}\zeta(v_1\partial_{x_1}v_2 g_2-v_2\partial_{x_1}v_2 g_1)\,dx+\int_{\Omega_T}\zeta(v_2\partial_{x_2}v_2g_2+v_2\partial_{x_1}v_2 g_1)\,dx\\
=&\int_{\Omega_T}\zeta(v_1g_2-v_2g_1)\partial_{x_1}v_2\,dx+\int_{\Omega_T}\frac12 \zeta(\partial_{x_2}(v_2^2) g_2+\partial_{x_1}(v_2^2) g_1)\,dx\\
=&\int_{\Omega_T}\zeta\left(v_1g_1 f_2'+v_1 \frac{\varepsilon \Phi \mu'(\cdot) f'(x_1) }{2(x_2-\bar{f}(x_1))} -v_2g_1\right)\partial_{x_1}v_2\,dx\\
&-\int_{\Omega_T}\frac12 \zeta v_2^2 (\partial_{x_1}g_1+\partial_{x_2}g_2)\,dx-\int_{\Omega_T}\frac12 \partial_{x_1}\zeta v_2^2 g_1\,dx\\
=&\int_{\Omega_T}\zeta(f_2'v_1-v_2)\partial_{x_1}v_2g_1\,dx+\int_{\Omega_{T}}\frac{\varepsilon \Phi \mu'(\cdot) f'(x_1) }{2(x_2-\bar{f}(x_1))} \zeta v_1 \partial_{x_1}v_2\,dx-\int_{ E}\frac12v_2^2 g_1\partial_{x_1}\zeta\,dx.
\end{aligned}
\end{equation*}
Note that on the boundary $S_2=\{(x_1,x_2):x_2=f_2(x_1)\}$, the impermeability condition $\Bv\cdot\Bn=0$ can be written as
\begin{equation*}
v_2(x_1,f_2(x_1))-f_2'(x_1)v_1(x_1,f_2(x_1))=0.
\end{equation*}
It follows from  \eqref{1-7} and Lemmas \ref{lemmaA1} and \ref{lemma1} that one has
\begin{equation}\label{3-5}
\begin{aligned}
&-\int_{\Omega_T}\Bv\cdot\nabla \Bg\cdot (\zeta \Bv)\,dx\\
=&\int_{\Omega_T}\zeta(v_1\partial_{x_1}+v_2\partial_{x_2})v_1g_1+\zeta(v_1\partial_{x_1}+v_2\partial_{x_2})v_2g_2\,dx+ \int_{ E}(\Bv\cdot \Bg)v_1\partial_{x_1}\zeta\,dx\\
=&\int_{\Omega_T}\zeta(v_2-f_2'v_1)(\partial_{x_2}v_1-\partial_{x_1}v_2)g_1\,dx- \int_{\Omega_{T}}\frac{\varepsilon \Phi \mu'(\cdot) f'(x_1) }{2(x_2-\bar{f}(x_1))} \zeta v_1 (\partial_{x_2}v_1-\partial_{x_1}v_2)\,dx\\
&+\int_{ E}(\Bv\cdot \Bg)v_1\partial_{x_1}\zeta-\frac12|\Bv|^2g_1\partial_{x_1}\zeta\,dx\\
\leq &\left(\int_{\Omega_T}\zeta (v_2-f_2'v_1)^2g_1^2\,dx\right)^\frac12\|\zeta^\frac12 \nabla \Bv\|_{L^2(\Omega_T)}+C\varepsilon \Phi \int_{\Omega_{T}} \left| \frac{1}{f} \zeta v_1 (\partial_{x_2}v_1-\partial_{x_1}v_2) \right| \,dx \\
&+\int_{ E}(\Bv\cdot \Bg)v_1\partial_{x_1}\zeta-\frac12|\Bv|^2g_1\partial_{x_1}\zeta\,dx\\
\leq&C\varepsilon \Phi\|\zeta^\frac12\partial_{x_2}(v_2-f_2'v_1)\|_{L^2(\Omega_T)}\|\zeta ^\frac12 \nabla \Bv\|_{L^2(\Omega_T)}+C\varepsilon \Phi  \left\|\zeta^\frac12 \frac{v_1}{f}\right\|_{L^2(\Omega_T)}\|\zeta^\frac12 \nabla \Bv\|_{L^2(\Omega_T)}\\
&+\int_{ E}(\Bv\cdot \Bg)v_1\partial_{x_1}\zeta-\frac12|\Bv|^2g_1\partial_{x_1}\zeta\,dx\\
\leq&\frac{\mfc}{2}\|\zeta ^\frac12 \nabla\Bv\|_{L^2(\Omega_T)}^2+\int_{ E}(\Bv\cdot \Bg)v_1\partial_{x_1}\zeta-\frac12|\Bv|^2g_1\partial_{x_1}\zeta\,dx,
\end{aligned}
\end{equation}
provided that $\varepsilon$ is sufficiently small. By virtue of integration by parts, one has
\begin{equation}\label{3-6}
\begin{aligned}
\int_{\Omega_T}-(\Bg\cdot \nabla \Bv+\Bv\cdot \nabla \Bv )\cdot(\zeta\Bv)\,dx =&\int_{\Omega_T}\frac12|\Bv|^2(\Bg+\Bv)\cdot \nabla \zeta\,dx\\
=&\int_{ E}\frac12|\Bv|^2(g_1+v_1)\partial_{x_1}\zeta\,dx
\end{aligned}
\end{equation}
and
\begin{equation}\label{3-7}
\begin{aligned}
\int_{\Omega_T}(\Delta \Bg-\Bg\cdot \nabla\Bg )\cdot(\zeta\Bv)\,dx = &\int_{\Omega_T}\zeta (-\nabla \Bg:\nabla \Bv+\Bg\cdot\nabla\Bv\cdot \Bg)\,dx\\
&+\int_{ E}[-\partial_{x_1}\Bg\cdot \Bv+ (\Bg\cdot\Bv) g_1]\partial_{x_1}\zeta\,dx.
\end{aligned}
\end{equation}
Combining \eqref{3-1}-\eqref{3-7} gives
\begin{equation}\label{3-8}
\begin{aligned}
&\frac {\mfc}{2} \int_{\Omega_T}\zeta|\nabla\Bv|^2\,dx+ \alpha \int_{\partial \Omega_T \cap \partial \Omega} \zeta |\Bv|^2 \, ds  \\
\leq & \int_{\Omega_T}\zeta (-\nabla \Bg:\nabla \Bv+\Bg\cdot\nabla\Bv\cdot \Bg)\,dx   + \left|\int_{ E}\partial_{x_1}\zeta \partial_{x_1}\Bv\cdot \Bv \, dx \right|\\
&+\int_{ E}\left[ (\Bv\cdot \Bg)v_1+\frac12|\Bv|^2v_1 \right] \partial_{x_1}\zeta + \partial_{x_1}\zeta pv_1 \,dx\\
&+\int_{ E}\left[-\partial_{x_1}\Bg\cdot \Bv+(\Bg\cdot\Bv)g_1 \right] \partial_{x_1}\zeta \,dx .
\end{aligned}
\end{equation}

{\em Step 2. Estimates for the Dirichlet norm.} The major task in this step is to estimate the terms on the right hand side of \eqref{3-8}. First, using Cauchy-Schwarz inequality and Lemma \ref{lemmaA1}, one has
\begin{equation}\label{3-9}
\left|\int_{ E^\pm}\partial_{x_1}\zeta \partial_{x_1}\Bv\cdot \Bv \, dx \right|\leq M_1(E^\pm)\|\nabla\Bv\|_{L^2( E^\pm)}^2.
\end{equation}
Next, it follows from Lemmas \ref{lemmaA1} and \ref{lemmaA2} that
\begin{equation}\label{3-10}
\begin{aligned}
\left|\int_{ E^\pm}\frac12 v_1|\Bv|^2\partial_{x_1}\zeta\,dx\right| \leq & \frac12 \|\Bv\|_{L^2( E^\pm )} \|\Bv\|_{L^4( E^\pm)}^2
\leq\frac12 M_1(E^\pm) M_4^2(E^\pm) \|\nabla \Bv\|_{L^2( E^\pm)}^3.
\end{aligned}
\end{equation}

Noting that $\Bg$ is bounded when $\varepsilon$ is determined, one uses Lemmas \ref{lemmaA1}, \ref{lemma1} and Young's inequality to obtain
\begin{equation}\label{3-11}
\left|\int_{ E^\pm}(\Bv\cdot \Bg)v_1\partial_{x_1} \zeta\,dx\right|\leq C\|\Bg\|_{L^\infty( E^\pm)}\|\Bv\|_{L^2( E^\pm)}^2\leq CM_1^2(E^\pm)\|\nabla\Bv\|_{L^2( E^\pm)}^2,
\end{equation}
\begin{equation}\label{3-12}
\begin{aligned}
 \left|\int_{ E}\left[ -\partial_{x_1}\Bg\cdot \Bv+(\Bg\cdot\Bv) g_1 \right] \partial_{x_1}\zeta\,dx\right|
\leq & C M_1(E^\pm) \left( \int_{ E^\pm} |\nabla \Bg|^2 + |\Bg|^4 \, dx \right)^{\frac12}   \|\nabla\Bv\|_{L^2( E^\pm)}\\
\leq&  C \int_{ E^\pm} |\nabla \Bg|^2 + |\Bg|^4 \, dx  +M_1^2(E^\pm)\|\nabla\Bv\|_{L^2( E^\pm)}^2,
\end{aligned}
\end{equation}
and
\begin{equation}\label{3-13}
\left|\int_{\Omega_T}\zeta (-\nabla \Bg:\nabla \Bv+\Bg\cdot\nabla\Bv\cdot \Bg)\,dx\right|\leq \frac{\mfc}{4}\int_{\Omega_T}\zeta|\nabla \Bv|^2\,dx+C\int_{\Omega_t}|\nabla \Bg|^2+|\Bg|^4\,dx.
\end{equation}

The most troublesome term involves the pressure $p$. Here we adapt a method introduced in \cite{LS}, by making use of the Bogovskii map. For $v_1\in L^2_0(E^\pm)$, it follows from Lemma \ref{lemmaA5} that there exists a vector $\Ba\in H_0^1(E^\pm)$ satisfying
\begin{equation*}
	\operatorname{div} \Ba=v_1\quad \text{in}\,\,  E^\pm
\end{equation*}
 and
 \be \nonumber
 \|\nabla \Ba \|_{L^2( E^\pm)} \leq M_5(E^\pm) \|v_1\|_{L^2( E^\pm)}.
 \ee
 We claim that the corresponding constant $M_5(E^\pm)$ is in dependent of $t$. Indeed, for $ E^+=\{\Bx\in \Omega:t-1<x_1<t\}$, one may write
\begin{equation*}
   E^+=\bigcup_{k=1}^{2N-1} E_k,
\end{equation*}
where
\begin{equation*}
 E_k=\left\{\Bx\in  E:t-1+\frac{k-1}{2N}\leq x_1\leq t-1+\frac{k+1}{2N}\right\}
\end{equation*}
and $N$ is a positive integer to be determined. For $k=1,2,\cdots,2N-1$, let $B_k$ denote the ball with radius $R_0$ centered at 
\begin{equation*}
	(t_k,q_k):=\left(t-1+\frac{k}{2N},\bar{f}(t_k)\right),
\end{equation*}
where $\bar{f}$ is defined in \eqref{deffbar}. By choosing 
\begin{equation*}
	N>\frac{\beta}{d}\ \ \ \text{ and	}\ \ \ R<\min\left\{\frac{1}{2N} ,~\frac{d}{2}  -\frac{\beta }{2N } \right\},
\end{equation*}
one can verify that $\overline{B_k}\subset  E_k$. 
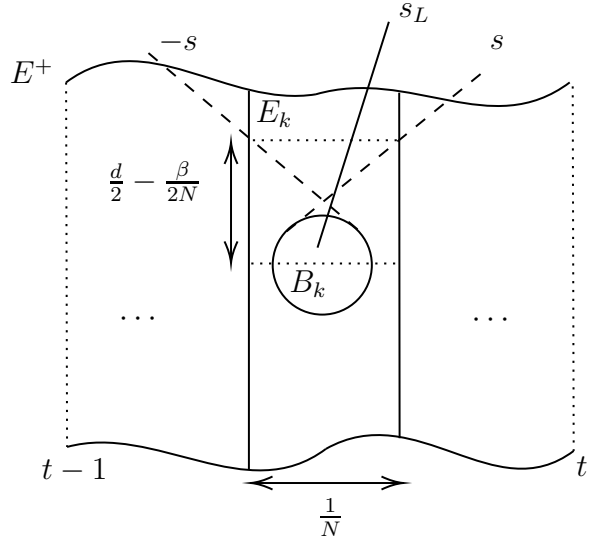
\begin{figure}[h]
\centering

\tikzset{every picture/.style={line width=0.75pt}} 

\begin{tikzpicture}[x=0.75pt,y=0.75pt,yscale=-1,xscale=1]

\draw    (215.8,65.8) .. controls (255.8,35.8) and (296.8,79) .. (341.8,71.8) ;
\draw    (216.29,249.8) .. controls (261.09,237.6) and (306.8,280.8) .. (346.8,250.8) ;
\draw    (341.8,71.8) .. controls (389.8,58) and (429.8,95.8) .. (469.8,65.8) ;
\draw    (346.8,250.8) .. controls (392.8,224.8) and (431.8,281.8) .. (471.8,251.8) ;
\draw  [line width=0.75]  (320,158) .. controls (320,144.19) and (331.19,133) .. (345,133) .. controls (358.81,133) and (370,144.19) .. (370,158) .. controls (370,171.81) and (358.81,183) .. (345,183) .. controls (331.19,183) and (320,171.81) .. (320,158) -- cycle ;
\draw  [dash pattern={on 4.5pt off 4.5pt}]  (327,141) -- (426.8,60) ;
\draw    (342.6,149.2) -- (378.6,35.2) ;
\draw    (308,70) -- (308,261.4) ;
\draw    (383.8,71) -- (384,245.4) ;
\draw  [dash pattern={on 0.84pt off 2.51pt}]  (310,95) -- (384.8,95) ;
\draw  [dash pattern={on 0.84pt off 2.51pt}]  (309.2,157.2) -- (384,157.2) ;
\draw  [dash pattern={on 0.84pt off 2.51pt}]  (215.8,65.8) -- (216.29,249.8) ;
\draw  [dash pattern={on 0.84pt off 2.51pt}]  (471.31,67.8) -- (471.8,251.8) ;
\draw  [dash pattern={on 4.5pt off 4.5pt}]  (362.8,140) -- (253.8,48) ;
\draw    (310.98,268) -- (383.8,268) ;
\draw [shift={(385.8,268)}, rotate = 180] [color={rgb, 255:red, 0; green, 0; blue, 0 }  ][line width=0.75]    (10.93,-3.29) .. controls (6.95,-1.4) and (3.31,-0.3) .. (0,0) .. controls (3.31,0.3) and (6.95,1.4) .. (10.93,3.29)   ;
\draw [shift={(308.98,268)}, rotate = 0] [color={rgb, 255:red, 0; green, 0; blue, 0 }  ][line width=0.75]    (10.93,-3.29) .. controls (6.95,-1.4) and (3.31,-0.3) .. (0,0) .. controls (3.31,0.3) and (6.95,1.4) .. (10.93,3.29)   ;
\draw    (298.98,98) -- (298.98,155) ;
\draw [shift={(298.98,157)}, rotate = 270] [color={rgb, 255:red, 0; green, 0; blue, 0 }  ][line width=0.75]    (10.93,-3.29) .. controls (6.95,-1.4) and (3.31,-0.3) .. (0,0) .. controls (3.31,0.3) and (6.95,1.4) .. (10.93,3.29)   ;
\draw [shift={(298.98,96)}, rotate = 90] [color={rgb, 255:red, 0; green, 0; blue, 0 }  ][line width=0.75]    (10.93,-3.29) .. controls (6.95,-1.4) and (3.31,-0.3) .. (0,0) .. controls (3.31,0.3) and (6.95,1.4) .. (10.93,3.29)   ;

\draw (310,73.4) node [anchor=north west][inner sep=0.75pt]    {$E_{k}$};
\draw (202,254.4) node [anchor=north west][inner sep=0.75pt]    {$t-1$};
\draw (471.8,253.2) node [anchor=north west][inner sep=0.75pt]    {$t$};
\draw (241,180) node [anchor=north west][inner sep=0.75pt]    {$\cdots $};
\draw (186,52.4) node [anchor=north west][inner sep=0.75pt]    {$E^{+}$};
\draw (382,24.4) node [anchor=north west][inner sep=0.75pt]    {$s_{L}$};
\draw (419,180.4) node [anchor=north west][inner sep=0.75pt]    {$\cdots $};
\draw (327,159.4) node [anchor=north west][inner sep=0.75pt]    {$B_{k}$};
\draw (428,40.4) node [anchor=north west][inner sep=0.75pt]    {$s$};
\draw (235,104.4) node [anchor=north west][inner sep=0.75pt]    {$\frac{d}{2} -\frac{\beta }{2N}$};
\draw (261,38.4) node [anchor=north west][inner sep=0.75pt]    {$-s$};
\draw (341,274.4) node [anchor=north west][inner sep=0.75pt]    {$\frac{1}{N}$};

\end{tikzpicture}

\caption{$E_k$ is star-like}
\end{figure}
Then we choose suitable $N$ such that any $ E_k$ is star-like with respect to $B_k$. For any  $\Bx\in B_k$, let $L$ be any ray starting from $\Bx$ and intersecting $\partial E_k$ at some $\tilde{\Bx}\in \partial E_k$. If $\tilde{\Bx}\in \partial E_k\setminus \partial\Omega$, then it is easy to see that $L$ intersects  $\partial E_k$ at and only at $\tilde{\Bx}$. If $\tilde{\Bx}\in \partial E_k\cap \partial\Omega$, then the slope $s_L$ of $L$ satisfies $|s_L|\ge s$ where $s$ is the slope of the segment above the ball $B_k$, which is tangent to the ball $B_k$ and  crosses the point $(t_k+\frac{1}{2N}, q_k+\frac{d}{2}-\frac{\beta}{2N})$. Direct computations show that $R$ and $s$ satisfy the following equality,
\begin{equation*}
  R=\frac{\left| \frac{s}{2N}-\frac{d}{2}+\frac{\beta}{2N }\right|}{\sqrt{1+s^2}}.
\end{equation*}
Noting that as $N\to +\infty$, one has $R\to 0$ and hence $s\to +\infty$. Therefore, we can choose sufficiently large $N$ such that
\begin{equation*}
  s > \beta\ge \|f_i'\|_{L^\infty}.
\end{equation*}
Thus, the ray $L$ has no intersection with $\partial E_k$ except at $\tilde{\Bx}$ and $ E_k$ is star-like with respect to $B_k$. Finally, following the notations in Lemma \ref{lemmaA5}, we denote the diameter of the domain $E^+$ by $R_0$ and set
\begin{equation*}
	\hat{E}_k= \bigcup_{i=k+1}^{2N-1} E_i, ~\tilde{E}_k=E_k\cap \hat{E}_k \ \ \ \text{ for } k=1,\cdots,2N-1.
\end{equation*}
Then direct computations give
\begin{equation*}
	R_0\leq \sqrt{1+(\overline{d}+\beta)^2},\ \ \ |E^+|\leq \overline{d},\ \ \ |\hat{E_k}\setminus E_k|\leq  |E^+|\leq \overline{d}
\end{equation*}
and
\begin{equation*}
|\tilde{E}_k|=|E_k\cap E_{k+1}|\ge \frac{d}{2N}.
\end{equation*}
These, together with \eqref{A5-2} and \eqref{A5-3}, show that the  constant $M_5(E^+)$ is independent of  $t$ since $N$ and $R$ are uniformly determined  by $d,\overline{d}$, and $\beta$.

It follows from  Lemmas \ref{lemmaA1} and \ref{lemmaA5} that one has
\begin{equation}\label{3-13-1}
\begin{aligned}
&\left|\int_{E^\pm}\partial_{x_1}\zeta pv_1\,dx\right|=\left|\int_{E^\pm}pv_1\,dx\right|=\left|\int_{ E^\pm}p{\rm div}\,\Ba\,dx\right| \\
=&\left|\int_{ E^\pm}2\BD(\Bv):\BD(\Ba)+(\Bv\cdot \nabla \Bg +(\Bg+\Bv)\cdot\nabla \Bv-\Delta \Bg+\Bg\cdot \nabla \Bg)\cdot \Ba\,dx\right|\\
=&\left|\int_{ E^\pm}2\BD(\Bv):\BD(\Ba)-\Bv\cdot \nabla \Ba\cdot \Bg - (\Bg+\Bv)\cdot\nabla \Ba\cdot \Bv+\nabla\Bg:\nabla \Ba-\Bg\cdot \nabla \Ba\cdot \Bg\,dx\right|\\
\le &C\|\nabla\Ba\|_{L^2(E^\pm)}\left(\|\nabla \Bv\|_{L^2( E^\pm)}+\|\Bv\|_{L^4( E^\pm)}^2+\|\nabla \Bg\|_{L^2( E^\pm)}+\|\Bg\|_{L^4( E^\pm )}^2\right)\\
\le &CM_5( E^\pm)\|v_1\|_{L^2( E^\pm)}\left(\|\nabla \Bv\|_{L^2( E^\pm)}+\|\Bv\|_{L^4( E^\pm)}^2+\|\nabla \Bg\|_{L^2( E^\pm)}+\|\Bg\|_{L^4( E^\pm)}^2\right)\\
\leq&CM_5( E^\pm)M_1( E^\pm)\|\nabla \Bv\|_{L^2( E^\pm)}\left(\|\nabla \Bv\|_{L^2( E^\pm)}+\|\Bv\|_{L^4( E^\pm)}^2+\|\nabla\Bg\|_{L^2( E^\pm)}\right.\\
&\left.+\|\Bg\|_{L^4( E^\pm)}^2\right),
\end{aligned}
\end{equation}
where the equality \eqref{2-11} has been used by taking $\Bp=\Ba$. Using Cauchy-Schwarz inequality and  Lemma \ref{lemmaA2} gives
\begin{equation}\label{3-14}
\begin{aligned}
	\left|\int_{E^\pm}\partial_{x_1}\zeta pv_1\,dx\right| \leq  &CM_5( E^\pm)M_1( E^\pm)\left(\|\nabla \Bv\|_{L^2( E^\pm)}^2+M_4^2( E^\pm)\|\nabla\Bv\|_{L^2( E^\pm)}^3\right)\\
	&+CM_5( E^\pm)M_1( E^\pm)\int_{ E^\pm }|\nabla \Bg|^2+|\Bg|^4\,dx.
\end{aligned}
\end{equation}

Here the constants $M_1(E^\pm)$ and $M_4(E^\pm)$ appeared in \eqref{3-9}-\eqref{3-14} depend on the domains $E^\pm$. Since the width of the cross-section $\Sigma(x_1)$ is uniformly bounded, it follows from \ref{lemmaA1} and \ref{lemmaA2} that  $M_1( E^\pm)$ and $M_4( E^\pm)$ are also uniform bounded. In summary, we have
\begin{equation}\label{summary1}
	\begin{aligned}
	 &\int_{\Omega_T}\zeta|\nabla\Bv|^2\,dx+ \alpha \int_{\partial \Omega_T \cap \partial \Omega} \zeta |\Bv|^2 \, ds\\
	\leq &C_1\left[\int_{ E}|\nabla\Bv|^2\,dx+\left(\int_{ E}|\nabla\Bv|^2\,dx\right)^\frac32\right]+ C \int_{\Omega_t} |\nabla \Bg|^2+|\Bg|^4\,dx .
\end{aligned}
\end{equation}

{\em Step 3. Growth estimate for the Dirichlet norm.} Let
\begin{equation*}
	y(t)=\int_{\Omega_t}\zeta|\nabla\Bv|^2\,dx+ \alpha \int_{\partial \Omega_t \cap \partial \Omega} \zeta |\Bv|^2 \, ds.
\end{equation*}
The straightforward computations give
\begin{equation*}
	y'(t)=\int_{ E}|\nabla\Bv|^2\,dx+ \alpha \int_{\partial  E \cap \partial \Omega}  |\Bv|^2 \, ds.
\end{equation*}
Then the estimate \eqref{summary1} can be written as
\begin{equation*}
y(t)\leq C_1\left\{y'(t)+[y'(t)]^\frac32\right\}+C_2\int_{\Omega_t} |\nabla \Bg|^2+|\Bg|^4\,dx.
\end{equation*}
Let $\delta_1=\frac12$,
\begin{equation*}
\Psi(\tau)=C_1(\tau+\tau^\frac{3}{2}),\quad \text{and}\quad
\varphi(t)=C_3+C_4t.
\end{equation*}
By Lemma \ref{lemma1}, if $f$ is bounded, then one has
\begin{equation*}
	\int_{\Omega_t} |\nabla \Bg|^2+|\Bg|^4\,dx\leq C(1+t).
\end{equation*}
 Therefore, one could choose $C_3$ and $C_4$ large enough such that
\begin{equation*}
 C_2 \int_{\Omega_t} |\nabla \Bg|^2+|\Bg|^4\,dx\leq \frac12(C_3+C_4t)=(1-\delta_1)\varphi(t)
\end{equation*}
and
\begin{equation*}
\varphi(t)=C_3+C_4t \ge 2C_1(C_4+C_4^\frac32)=\delta_1^{-1}\Psi(\varphi'(t)) \quad \text{for any}\,\, t\geq 1,
\end{equation*}
provided that $C_3$ is large enough.

Finally, according to Proposition \ref{appro-existence} and Lemma \ref{lemma1}, if $C_3$ and $C_4$ are chosen to be large enough, then one has
\begin{equation*}
y(T)\leq \|\nabla\Bv\|_{L^2(\Omega_{T})}^2+\alpha\|\Bv\|_{L^2(\partial\Omega_T\cap \partial\Omega)}^2\leq C_0\int_{\Omega_T} |\nabla \Bg|^2+|\Bg|^4\,dx\leq \varphi(T).
\end{equation*}
Hence we can apply Lemma \ref{lemmaA4} to finish the proof of this lemma.
\end{proof}

  As proved in Lemma \ref{lemma3}, for any $1\leq t\le T-1$,
\begin{equation*}
\|\nabla \Bv^T\|_{L^2(\Omega_{t})}^2+\alpha\|\Bv^T\|_{L^2(\partial\Omega_t\cap \partial\Omega)}^2\leq C_3 +C_4 t.
\end{equation*}
 Since the constants $C_3$ and $C_4$ are independent of $t$ and $T$, one can take the limit $T\to \infty$ and select a subsequence which converges weakly in $H^1_{loc}(\Omega)$ to a solution $\Bv$ of \eqref{NS1}. Moreover, $\Bv$ satisfies the estimate
\begin{equation*}
\|\nabla \Bv\|_{L^2(\Omega_{t})}^2+\alpha\|\Bv\|_{L^2(\partial\Omega_t\cap \partial\Omega)}^2\leq C_3 +C_4 t.
\end{equation*}
With the estimate for $\Bg$ in Lemma \ref{lemma1}, one has the following proposition on the existence of solutions.
\begin{pro}\label{straight-existence}
The problem \eqref{NS} and \eqref{flux constraint}-\eqref{BC} has a solution  $\Bu=\Bv+\Bg\in H_\sigma(\Omega)$ satisfying
\begin{equation}\label{3-27}
\|\nabla \Bu\|_{L^2(\Omega_{t})}^2+\|\Bu\|_{L^2(\partial\Omega_t\cap \partial\Omega)}^2\leq \tilde{C}(1+t),
\end{equation}
where the constant $\tilde{C}$ depends only on $\alpha$, $\Phi$, and $\Omega$.
\end{pro}
\begin{remark}
	There exists a constant $C>0$ such that
	for any fixed subdomain $\Omega_{a,b}$, if $\Phi>0$ is sufficiently small, one has
	\[
	\int_{\Omega_{a,b}}|\nabla \Bg|^2+|\Bg|^4\,dx\leq C\Phi^2.
	\]
	Therefore, there exists a $\Phi_0>0$ such that if $\Phi\in [0, \Phi_0)$, one has
\[
C_3+C_4+\tilde{C}\leq C\Phi^2,
\]
where $C_3,C_4$, and $\tilde{C}$ are the constant appeared in \eqref{3-0} and \eqref{3-27}.
\end{remark}

Similar to Lemma \ref{pressure}, one can also define the pressure of the problem \eqref{NS} and \eqref{BC}.
\begin{pro}\label{pressure1}
The vector field $\Bu\in H_\sigma(\Omega)$ is a weak solution of the problem \eqref{NS} and \eqref{BC} if and only if there exists a function $p\in L^2_{loc}(\Omega)$ such that  for any $\Bp\in \mcH(\Omega_T)$ with $T>0$, it holds that
\begin{equation}\label{3-14-1}
\begin{aligned}
\int_{\Omega_T}2\BD(\Bu):\BD(\Bp)+\Bu\cdot \nabla \Bu \cdot \Bp - p{\rm div}\Bp\,dx+2\alpha \int_{\partial\Omega_T\cap \partial\Omega}\Bu\cdot \Bp\,ds=0.
\end{aligned}
\end{equation}

\end{pro}

Actually, one can show that the Dirichlet norm of the solution $\Bu$ is uniformly bounded in any subdomain $\Omega_{t-1,t}$.
\begin{pro}\label{uniform estimate}
Let $\Bu$ be the solution obtained in Proposition \ref{straight-existence}. Then there exists a constant $C_7$ such that
\begin{equation}\label{estv4.4}
\|\nabla \Bu\|_{L^2(\Omega_{t-1,t})}^2+\|\Bu\|_{L^2(\partial\Omega_{t-1,t}\cap \partial\Omega)}^2\leq C_7, \quad \text{for any}\,\, t\in \mathbb{R}.
\end{equation}
\end{pro}
\begin{proof} It's sufficient to prove the estimate \eqref{estv4.4} for $\Bv=\Bu-\Bg$. For any fixed $T\ge 2$, let $\zeta_T(\Bx,t)$ be a truncating function with
\begin{equation*}
\zeta_T(\Bx,t)=\left\{\begin{aligned}
&0,~~~~~~~~~~~~~~~~~~~~~~~~~~~~~~~~~~~~&& \text{ if }x_1\in (-\infty,T-t)\cup (T+t,\infty),\\
&1,~~~~~~~~~~~~~~~~~~~~~~~~~~~~~~~~~~~~&&\text{ if }x_1\in (T-t+1,T+t-1),\\
&-T+t+x_1,~~~~~~~~~~~&&\text{ if }x_1\in [T-t,T-t+1],\\
&T+t-x_1,~~~~~~~~~~~&&\text{ if }x_1\in [T+t-1,T+t].
\end{aligned}\right.
\end{equation*}
Denote
\begin{equation*}
y_T(t)=\int_{\Omega}\zeta_T|\nabla\Bv|^2\,dx+\alpha\int_{\partial\Omega}\zeta_T|\Bv|^2\,ds.
\end{equation*}
Similar to the proof of Lemma \ref{lemma3}, one  has the estimate
\begin{equation*}
y_T(t)\leq C_1 \left\{  y_T'(t)+[y_T'(t)]^\frac32\right\}+C_2\int_{\Omega_{T-t,T+t}}|\nabla\Bg|^2+|\Bg|^4\,dx.
\end{equation*}

According to Proposition \ref{straight-existence}, one has
\begin{equation*}
  y_T(T)\leq \int_{\Omega_{0,2T}}|\nabla\Bv|^2\,dx+\alpha\int_{ \partial\Omega_{0,2T}\cap \partial\Omega}|\Bv|^2\,ds\leq C_3+2C_4T.
\end{equation*}
Then we take $\tilde{\varphi}(t)=C_5+C_6t$ where $C_5$ and $C_6$ are  large enough such that
for any $t\geq 1$,
one has
\begin{equation*}
C_3+ 2C_4T \leq \tilde{\varphi}(T),~\ \ \ C_2 \int_{\Omega_{T-t,T+t}}|\nabla\Bg|^2+|\Bg|^4\,dx \leq \frac12\tilde{\varphi}(t),
\end{equation*}
and
\begin{equation*}
\tilde{\varphi}(t)\ge 2C_1 \left(C_6 + C_6^\frac32 \right).
\end{equation*}
 Hence, by taking $\Psi(\tau)=C_2(\tau+\tau^\frac32)$ and $\delta_1=\frac12$, we make use of Lemma \ref{lemmaA4} to obtain
\begin{equation*}
y_T(t)\leq C_5+C_6t\quad \text{for any}\,\, t\in [1,T].
\end{equation*}
 In particular, choosing $t=\frac32$ gives
\begin{equation*}
\int_{\Omega_{T-\frac12,T+\frac12}}|\nabla \Bv|^2\,dx+\alpha \int_{\partial\Omega_{T-\frac12,T+\frac12}\cap \partial\Omega}|\Bv|^2\,ds \leq y_T\left(\frac32\right)=C_5+\frac32 C_6.
\end{equation*}
The case $T\leq -2$ can be proved similarly. Hence the proof of the proposition is completed.
\end{proof}

With the help of the uniform estimate given in Proposition \ref{uniform estimate}, we can prove the uniqueness of the solution when the flux is sufficiently small.
\begin{pro}\label{uniqueness}
There exists a constant $\Phi_0>0$ such that for any flux $\Phi\in [0,\Phi_0)$, the solution obtained in Proposition \ref{straight-existence} is unique.
\end{pro}
\begin{proof} We divide the proof into three steps.

{\em Step 1. Set up.} Let $\Bu$ be the solution obtained in Proposition \ref{straight-existence}. Assume that $\tilde{\Bu}$ is also a solution of problem \eqref{NS} and \eqref{flux constraint}-\eqref{BC} satisfying
\begin{equation*}
	\|\nabla \tilde{\Bu}\|_{L^2(\Omega_{t})}^2+\|\tilde{\Bu}\|_{L^2(\partial\Omega_{t}\cap \partial\Omega)}^2\leq C(1+t).
\end{equation*} Then $\oBu :=\tilde{\Bu}-\Bu$ is a weak solution to the equations
\begin{equation}\label{3-15}
\left\{\begin{aligned}
&-\Delta \oBu +\oBu \cdot \nabla\Bu +\Bu\cdot \nabla \oBu +\oBu\cdot \nabla \oBu +\nabla p=0  ~~~~&\text{ in }\Omega,\\
&{\rm div}~\oBu =0&\text{ in }\Omega,\\
&\oBu \cdot \Bn=0,~(\Bn\cdot \BD(\oBu )+\alpha \oBu )\cdot \Bt=0&\text{ on } \partial\Omega,\\
&\int_{\Sigma(x_1)}\oBu  \cdot \Bn \,ds=0 &\text{ for any }x_1\in \R.
\end{aligned}\right.
\end{equation}
Multiplying the problem \eqref{3-15} by $\zeta\oBu $ where $\zeta$ is defined in Lemma \ref{lemma3}, and integrating over $\Omega$ yield
\begin{equation}\label{3-16}
\begin{aligned}
\int_{\Omega}2\BD(\oBu ):\BD(\zeta\oBu )+[\oBu \cdot \nabla \Bu +(\Bu+\oBu)\cdot \nabla \oBu ]\cdot (\zeta\oBu) \, dx & \\
- \int_{\Omega} p{\rm div}(\zeta\oBu)\,dx+2\alpha \int_{ \partial\Omega}\zeta \oBu ^2\,ds&=0.
\end{aligned}
\end{equation}
Similar to \eqref{3-3}, \eqref{3-4-1} and \eqref{3-6}, one has
\begin{equation}\label{3-17}
	\mfc\int_{\Omega}\zeta|\nabla\oBu |^2\,dx\leq 2\int_{\Omega}\BD(\oBu ):\BD(\zeta\oBu )\,dx+\alpha \int_{ \partial\Omega}\zeta|\oBu |^2\,ds
+ \left|\int_{ E}\partial_{x_1}\zeta \partial_{x_1}\oBu \cdot \oBu\,dx\right|,
\end{equation}
\begin{equation}\label{3-18}
\begin{aligned}
	-\int_{\Omega}\oBu \cdot\nabla \Bu\cdot (\zeta \oBu) \,dx
	=&\int_{\Omega}\zeta\oBu \cdot\nabla \oBu \cdot \Bu\,dx+\int_{ E}(\oBu \cdot \Bu)\overline{u}_1\partial_{x_1}\zeta\,dx\\
	=&\int_{\Omega_{t-1}}\oBu \cdot\nabla \oBu \cdot \Bu\,dx+\int_{ E}\zeta \oBu \cdot\nabla \oBu \cdot \Bu+(\oBu \cdot \Bu)\overline{u}_1\partial_{x_1}\zeta\,dx
\end{aligned}
\end{equation}
and
\begin{equation}\label{3-19}
-\int_{\Omega}(\Bu\cdot \nabla \oBu +\oBu\cdot \nabla \oBu )\cdot (\zeta\oBu) \,dx
=\int_{ E}\frac12|\oBu |^2(u_1+\overline{u}_1)\partial_{x_1}\zeta\,dx.
\end{equation}
Substituting \eqref{3-17}-\eqref{3-19} into \eqref{3-16} gives
\begin{equation}\label{3-20}
\begin{aligned}
&\mfc\int_{\Omega}\zeta|\nabla\Bu|^2\,dx+\alpha \int_{\partial\Omega}\zeta|\oBu |^2\,ds \\
\leq &\int_{\Omega_{t-1}} \oBu \cdot\nabla \oBu \cdot \Bu\,dx+\int_{E}\zeta \oBu \cdot\nabla \oBu \cdot \Bu\,dx\\
&+ \left|\int_{E}\partial_{x_1}\zeta \partial_{x_1}\oBu \cdot \oBu\,dx\right|+\int_{ E}\left[(\oBu \cdot \Bu)\overline{u}_1+\frac12|\oBu |^2(u_1+\overline{u}_1)+p\overline{u}_1   \right] \partial_{x_1}\zeta\,dx.
\end{aligned}
\end{equation}

{\em Step 2. Estimates for the Dirichlet norm.} Decompose $\Omega_{t-1}$ into several parts $\Omega_t^i=\{\Bx\in\Omega:~x_1\in(A_{i-1},A_i)\}$, where $-t+1 =A_0\leq A_1\leq \cdots\leq A_{N(t)}=t-1 $ and $\frac12\leq A_i-A_{i-1}\leq 1$ for every $i$. By Proposition \ref{uniform estimate} and Lemmas \ref{lemmaA1}-\ref{lemmaA2}, one has
\begin{equation}\label{3-22}
\begin{aligned}
 \int_{\Omega_{t-1}} \oBu \cdot\nabla \oBu \cdot \Bu \,dx \leq & \sum_{i=1}^{N(t)}\int_{\Omega_t^i }|\oBu \cdot\nabla \oBu \cdot \Bu|\,dx \\
\leq &  \sum_{i=1}^{N(t)}\|\nabla\oBu \|_{L^2(\Omega_t^i )} \|\oBu \|_{L^4(\Omega_t^i )}(\|\Bv\|_{L^4(\Omega_t^i )}+\|\Bg\|_{L^4(\Omega_t^i )}) \\
\leq&C \sum_{i=1}^{N(t)}\|\nabla\oBu \|_{L^2(\Omega_t^i )}^2(\|\nabla\Bv\|_{L^2(\Omega_t^i)}+\|\Bg\|_{L^4(\Omega_t^i )}) \\
\leq&C_8\sum_{i=1}^{N(t)}\|\nabla\oBu \|_{L^2(\Omega_t^i )}^2 \\
=&C_8\int_{\Omega_{t-1}} |\nabla\oBu |^2 \, dx.
\end{aligned}
\end{equation}
Since the constant $C_8$ is of the same order as $\Phi^2$ when $\Phi\to 0$, there exists some $\Phi_0>0$, such that for any $\Phi\in [0,\Phi_0)$, one has
\begin{equation}\label{3-23}
\int_{\Omega_{t-1}}\oBu \cdot\nabla \oBu \cdot \Bu\,dx \leq \frac{\mfc}{2}\int_{\Omega_t}\zeta|\nabla\oBu |^2\,dx.
\end{equation}


Then one uses Lemmas \ref{lemmaA1}-\ref{lemmaA2}, \ref{lemma1}, and Proposition \ref{uniform estimate} to obtain
\begin{equation}\label{3-24-0}
\begin{aligned}
\int_{ E^\pm} \zeta \oBu \cdot \nabla \oBu\cdot \Bu \, dx \leq&\|\oBu\|_{L^4( E^\pm)}\|\nabla\oBu\|_{L^2( E^\pm)}\|\Bu\|_{L^4( E^\pm)}\\
\leq&\|\oBu\|_{L^4( E^\pm)}\|\nabla\oBu\|_{L^2( E^\pm)}(\|\Bv\|_{L^4( E^\pm)}+\|\Bg\|_{L^4( E^\pm)})\\
\leq&C\|\oBu\|_{L^4( E^\pm)}\|\nabla\oBu\|_{L^2( E^\pm)}(\|\nabla \Bv\|_{L^2( E^\pm)}+\|\Bg\|_{L^4( E^\pm)})\\
\leq&C\|\nabla\oBu \|_{L^2( E^\pm)}^2
\end{aligned}
\end{equation}
and
\begin{equation}\label{3-24}
\begin{aligned}
&\left|\int_{ E^\pm}\partial_{x_1}\zeta\partial_{x_1}\oBu \cdot \oBu\,dx\right|+\int_{E^\pm}\left[(\oBu \cdot \Bu)\overline{u}_1+\frac12|\oBu |^2(u_1+\overline{u}_1) \right] \partial_{x_1}\zeta\,dx\\
\leq&C\left(\|\nabla\oBu \|_{L^2( E^\pm)}\|\oBu \|_{L^2( E^\pm)}+\|\oBu \|_{L^4( E^\pm)}^2\|\Bu\|_{L^2( E^\pm)}+\|\oBu\|_{L^4( E^\pm)}^2\|\oBu\|_{L^2( E^\pm)}\right)\\
\leq&C\left[\|\nabla\oBu \|_{L^2( E^\pm)}^2+\|\nabla \oBu \|_{L^2( E^\pm)}^2(\|\Bv\|_{L^2( E^\pm)}+\|\Bg\|_{L^2( E^\pm)})+\|\nabla \oBu\|_{L^2( E^\pm)}^3\right]\\
\leq&C \left(\|\nabla\oBu \|_{L^2( E^\pm)}^2+\|\nabla\oBu \|_{L^2( E^\pm)}^3\right).
\end{aligned}
\end{equation}
The estimate of $\int_{ E^\pm}p\overline{u}_1\partial_{x_1}\zeta\,dx$ is similar to \eqref{3-13-1}. That is,
\begin{equation}\label{3-25}
\begin{aligned}
&\left|\int_{ E^\pm}p\overline{u}_1\partial_{x_1}\zeta\,dx\right|= \left|\int_{ E^\pm}p\overline{u}_1\,dx\right|=\left|\int_{E^\pm}p{\rm div}\,\Ba\,dx\right|\\
=&\left|\int_{ E^\pm} 2\BD(\overline {\Bu}):\BD(\Ba)+(\oBu \cdot \nabla \Bu +(\Bu+\oBu)\cdot \nabla \oBu )\cdot\Ba\,dx\right|\\
=&\left|\int_{ E^\pm} 2\BD(\overline {\Bu}):\BD(\Ba)-\oBu \cdot \nabla \Ba\cdot\Bu -(\Bu+\oBu)\cdot \nabla \Ba \cdot \oBu\,dx\right|\\
\le &C\left(\|\nabla \overline {\Bu}\|_{L^2( E^\pm)}+\|\overline {\Bu}\|_{L^4( E^\pm)}\|\Bu\|_{L^4( E^\pm)}+\|\overline {\Bu}\|_{L^4( E^\pm)}^2\right)\|\nabla\Ba\|_{L^2( E^\pm)}\\
\leq&C(\|\nabla\oBu \|_{L^2( E^\pm)}^2+\|\nabla\oBu \|_{L^2( E^\pm)}^3),
\end{aligned}
\end{equation}
where $\Ba\in H_0^1(E^\pm)$ satisfies
\[
\text{div}\Ba =\bar{u}_1\quad \text{in}\,\, E^\pm.
\]

{\em Step 3. Growth estimate. } Let
\begin{equation*}
y(t)=\int_{\Omega}\zeta|\nabla \oBu |^2\,dx+\alpha\int_{ \partial\Omega}\zeta| \oBu |^2\,ds.
 \end{equation*}
 Combining \eqref{3-20}-\eqref{3-25} gives the differential inequality
\begin{equation*}
y(t)\leq C_9 \left\{y'(t)+[y'(t)]^\frac32\right\}.
\end{equation*}
This, together with Lemma \ref{lemmaA4}, implies that either $\oBu=0$ or
\begin{equation*}
\liminf_{t\rightarrow +\infty} \frac{y(t)}{t^3}>0.
\end{equation*}
Hence the proof of the proposition is completed.
\end{proof}

In particular, if an outlet of the channel is straight, for example,
\begin{equation*}
\Sigma(x_1)=\Sigma^\sharp(x_1):=(-1,1)\ \ \ \text{ when $x_1>0$},
\end{equation*}
we shall show that the solution obtained in Proposition \ref{straight-existence} tends to $\BU$ at infinity, where $\BU$ is given in \eqref{shearflow}, and is the shear flow solution of the Navier-Stokes system with Navier-slip boundary condition in the straight channel ${\Omega}^\sharp=\{(x_1,x_2):~x_1\in \R,~x_2\in (-1,1)\}$.
\begin{pro}\label{Poiseuille}
Assume that the outlet $\Omega^+=\{\Bx\in\Omega:~x_1>0\}= (0, +\infty) \times (-1, 1)$ is straight. There exists a constant $\Phi_1>0$, such that if  $\Phi\in [0,\Phi_1)$,  and the solution $\Bu$ of the problem \eqref{NS} and \eqref{flux constraint}-\eqref{BC}  satisfies
\begin{equation}\label{3-26}
\liminf_{t\rightarrow + \infty} t^{-3}\int_{\Omega^+_t} |\nabla \Bu|^2\,dx =0,
\end{equation}
where $\Omega^+_t=\{\Bx\in\Omega:~0<x_1<t\}$, then it holds that
\begin{equation*}
\|\Bu-\BU\|_{H^1(\Omega^+)}<\infty,
\end{equation*}
where $\BU$ is given in \eqref{shearflow}.
\end{pro}
\begin{proof}
First, let us introduce a new truncating function
\begin{equation*}
\zeta_1(\Bx,t)=\left\{
\begin{aligned}
&1,~~~~~~~~~~~~~~ &&\text{ if }x_1\in (1,t-1),\\
&0,~~~~~~~~~~~~~~ &&\text{ if }x_1\in (-\infty,0)\cup(t,\infty),\\
&t-x_1,~~~~~~&&\text{ if }x_1\in [t-1,t],\\
&x_1,~~~~~~&&\text{ if }x_1\in [0,1].
\end{aligned}\right.
\end{equation*}
Denote $\overline{\Bv}=\Bu-\BU$. Taking the test function $\Bp=\zeta_1\overline{\Bv}$  in \eqref{3-14-1} gives
\begin{equation*}
\begin{aligned}
&2\int_{\Omega^+}\BD(\overline{\Bv}):\BD(\zeta_1\overline{\Bv})\,dx+2\alpha \int_{\partial\Omega^+\cap \partial\Omega}\zeta_1|\overline{\Bv}|^2\,ds\\
&+\int_{\Omega^+}\left[\overline{\Bv}\cdot\nabla \BU+(\overline{\Bv}+\BU)\cdot\nabla \overline{\Bv} \right] \cdot (\zeta_1\overline{\Bv})- p{\rm div}\,(\zeta_1\overline{\Bv})\,dx\\
=&\int_{\Omega^+}-2\BD(\BU):\BD(\zeta_1\overline{\Bv})-\BU \cdot \nabla \BU\cdot (\zeta_1\overline{\Bv})\,dx-2\alpha \int_{\partial\Omega^+\cap \partial\Omega}\BU\cdot (\zeta_1\overline{\Bv})\,ds.
\end{aligned}\end{equation*}

Noting that $\BU$ is a solution of the Navier-Stokes system  in $\Omega^+$, one has
\begin{equation*}
\int_{\Omega^+}-2\BD(\BU):\BD(\zeta_1\overline{\Bv})-\BU \cdot \nabla \BU\cdot (\zeta_1\overline{\Bv})\,dx-2\alpha \int_{\partial\Omega^+\cap \partial\Omega}\BU\cdot (\zeta_1\overline{\Bv})\,ds=\int_{\Omega^+} - P {\rm div}\,(\zeta_1\overline{\Bv})\,dx,
\end{equation*}
where $P$ is the pressure for the Navier-Stokes system associated with the velocity  $\BU$. Let $\bar{p}=p-P$. Hence,
\begin{equation*}
\begin{aligned}
&2\int_{\Omega^+}\BD(\overline{\Bv}):\BD(\zeta_1\overline{\Bv})\,dx+2\alpha \int_{\partial\Omega^+\cap \partial\Omega}\zeta_1|\overline{\Bv}|^2\,ds\\
&+\int_{\Omega^+} \left[ \overline{\Bv}\cdot\nabla \BU+ (\overline{\Bv}+\BU)\cdot\nabla \overline{\Bv} \right] \cdot (\zeta_1\overline{\Bv} )- \bar{p}{\rm div}\,(\zeta_1\overline{\Bv})\,dx=0,
\end{aligned}\end{equation*}

Let
\begin{equation*}
y^+(t)=\int_{\Omega^+}\zeta_1(\Bx,t)|\nabla \oBv |^2\,dx+\alpha\int_{\partial\Omega^+ \cap \partial\Omega}\zeta_1(\Bx,t)| \oBv |^2\,ds.
\end{equation*}
Following similar proof of Proposition \ref{uniqueness}, one obtains
\begin{equation*} \begin{aligned}
y^+ & \leq C_9 \left\{ \int_{ E_0 } |\nabla \overline{\Bv} |^2 \, dx + \left(\int_{ E_0} |\nabla \overline{\Bv} |^2 \, dx \right)^{\frac32}   +(y^+)'+[(y^+)']^\frac32 \right\} \\
& \leq C_{10} +C_9 \left\{(y^+)'+[(y^+)']^\frac32 \right\},
\end{aligned}
\end{equation*}
where $ E_0=\{\Bx\in\Omega:~0<x_1<1\}$ and   $C_{10}$ is  a constant independent of $t$.
This, together with  the assumption \eqref{3-26} and Lemma \ref{lemmaA4}, implies that $y^+(t)$ is bounded. Hence the proof of the proposition is completed.
\end{proof}
Combining Propositions \ref{straight-existence}, \ref{uniqueness}, and \ref{Poiseuille} together finishes the proof of Theorem \ref{bounded channel}.

\section{Flows in channels with unbounded outlets}\label{secexist2}
In this section, we study the flows in channels with unbounded width. Recall the definition of $\beta$ which is given in \eqref{assumpf}. In the rest of this section, $(4\beta)^{-1}$ is used to here and there.  For convenience, denote
\begin{equation*}
	\beta^*:=(4\beta)^{-1}.
\end{equation*}
Clearly, one has
\begin{equation*}
	\|f'\|_{L^\infty}= 2\beta=(2\beta^*)^{-1}
\end{equation*}
and
\begin{equation}\label{4-17}
	\frac12 f(t)\leq f(\xi) \leq \frac32f(t) \text{ for any }\xi\in[t-\beta^* f(t),\,  t+\beta^* f(t)].
\end{equation}
Define
\begin{equation*}
	k(t):=\int_0^t f^{-\frac53}(\xi)\,d\xi
\end{equation*}
and let $h(t)$ be the inverse function of $k(t)$. Then one has
\begin{equation*}
	t=\int_0^{h(t)}f^{-\frac53}(\xi)\,d\xi
\quad \text{and}\quad
	h'(t)=f^{\frac53}(h(t)).
\end{equation*}
Denote 
\begin{equation}\label{defh}
	h_L(t)=h(-t)+\beta^*f(h(-t))  \text{ and }h_R(t)=h(t)-\beta^*f(h(t)).
\end{equation}
Direct computations give 
\begin{equation}\label{4-1}
	\frac{d}{dt}h_L(t)=-h'(-t)-\beta^*f'(h(-t))h'(-t)=-[1+\beta^*f'(h(-t))]f^\frac53(h(-t))\leq -\frac{d^\frac53}{2}
\end{equation}
and 
\begin{equation}\label{4-2}
	\frac{d}{dt}h_R(t)=h'(t)-\beta^*f'(h(t))h'(t)=[1-\beta^*f'(h(t))]f^\frac53(h(-t))\ge \frac{d^\frac53}{2}.
\end{equation}


The existence of the solutions for problem \eqref{NS}, \eqref{flux constraint} and \eqref{BC} is investigated in three cases, according to  the range of $k$.

{\bf Case 1. The range of $k(t)$ is $(-\infty, \infty)$.} In this case, the function $h(t)$ is defined on $(-\infty,\infty)$. It follows from \eqref{4-1} and \eqref{4-2} that for suitably large $t$, one has 
\begin{equation*}
	h_L(t)<h_R(t).
\end{equation*}
 Then we introduce a new truncating function  $\hat{\zeta}(\Bx,t)$ on $\Omega$ as follows,
\begin{equation} \label{cut-off-hat}
\hat{\zeta}(\Bx,t)=\left\{
\begin{aligned}
&0,~~~~~~~~~~~~~~ &&\text{ if }x_1\in (-\infty,h(-t))\cup(h(t),\infty),\\
&\frac{h(t)-x_1}{f(h(t))},~~~~~~&&\text{ if }x_1\in [h_R(t),\, h(t)],\\
&\beta^*,~~~~~~~~~~~~~~ &&\text{ if }x_1\in (h_L(t),\, h_R(t)),\\
&\frac{-h(-t)+x_1}{f(h(-t))},~~~~~~&&\text{ if }x_1\in [h(-t),\, h_L(t)].
\end{aligned}\right.
\end{equation}
For the sake of convenience, one denotes
\begin{equation}\label{4-6}
\hat{\Omega}_{t}=\{\Bx\in \Omega:x_1\in (h(-t),h(t))\}\ \ \ \text{ and } \ \ \ \breve{\Omega}_{t}=\hat{\Omega}_{t}\setminus \overline{\hat{ E}},
\end{equation}
where
$\hat{ E}=\hat{ E}^+\cup \hat{ E}^-$
with
\begin{equation}\label{4-7}
\hat{ E}^-=\{\Bx\in\Omega:x_1\in (h(-t),h_L(t))\},~\hat{ E}^+=\{\Bx\in\Omega:x_1\in (h_R(t),h(t))\}.
\end{equation}
Clearly, $\nabla \hat{\zeta}$ and $\partial_t\hat{\zeta}$ vanish outside $\hat{ E}$ and satisfy
\begin{equation}\label{4-3}
|\nabla\hat{\zeta}|=|\partial_{x_1}\hat{\zeta}|= [f(h(\pm t))]^{-1}\,\,\text{in}\,\,\hat{E}^\pm,
\end{equation}
and
\begin{equation}\label{4-4}
\partial_t\hat{\zeta}=\frac{h'(\pm t)}{f(h(\pm t))}\left[1\mp\frac{\pm h(\pm t)\mp x_1}{f(h(\pm t))}f'(h(\pm t))\right]\ge \frac12\frac{h'(\pm t)}{f(h(\pm t))}=\frac12[f(h(\pm t))]^\frac23\,\,\text{in}\,\,\hat{E}^\pm.
\end{equation}

With the help of the new truncating function $\hat{\zeta}(\Bx,t)$, we have the following lemma which is used to prove the uniform local estimate for approximate solutions.

\begin{lemma}\label{lemma4}
Assume that the domain $\Omega$ satisfies \eqref{deffbar}, and
\be \nonumber
\int_{-\infty}^0 f^{-\frac53}(\tau) \, d\tau = \infty, \ \ \ \ \ \int_0^{+\infty} f^{-\frac53}(\tau) \, d\tau = \infty.
\ee
Let $\Bv^T$ be the solution of the approximate problem \eqref{aNS} on $\hat{\Omega}_T$, which is obtained in Proposition \ref{appro-existence} and satisfies the energy estimate \eqref{2-10-1}. Then there exists a positive constant $C_{15}$ independent of $t$ and $T$ such that
\begin{equation}\label{4-9}
\|\nabla \Bv^T\|_{L^2(\breve{\Omega}_{t})}^2+\alpha\|\Bv^T\|_{L^2(\partial\breve{\Omega}_{t}\cap \partial\Omega)}^2\leq C_{15}\left(1+\int_{h(-t)}^{h(t)} f^{-3}(\tau)\,d\tau\right) \,\,\text{for any }t^* \leq t\leq T,
\end{equation}
 where
\begin{equation}\label{deft*}
t^*=\sup\{t>0:h_L(t)\ge  h_R(t)\}.
\end{equation}
\end{lemma}
\begin{proof}
The superscript $T$ will be omitted throughout the proof. The proof is quite similar to that for Lemma \ref{lemma3}. Taking the test function $\Bp=\hat{\zeta} \Bv$ in \eqref{2-11} yields
\begin{equation}\label{4-10}
\begin{aligned}
	&\frac{\mfc}{2}\int_{\hat{\Omega}_T}\hat{\zeta}|\nabla\Bv|^2\,dx+\alpha \int_{\partial\hat{\Omega}_T\cap \partial\Omega}\hat{\zeta}|\Bv|^2\,ds\\
	\leq &\int_{\hat{\Omega}_T}\hat{\zeta} (-\nabla \Bg:\nabla \Bv+\Bg\cdot\nabla\Bv\cdot \Bg)\,dx\\
	&+\int_{\hat{ E}} \left[ (\Bv\cdot \Bg)v_1+\frac12|\Bv|^2v_1+pv_1 \right] \partial_{x_1}\hat{\zeta}\,dx +  \left|\int_{\hat{ E}} \partial_{x_1}\hat{\zeta} \partial_{x_1}\Bv\cdot \Bv\, dx\right| \\
	&+\int_{\hat{ E}}\left[ -\partial_{x_1}\Bg\cdot \Bv+(\Bg\cdot\Bv)g_1 \right] \partial_{x_1}\hat{\zeta}\,dx.
\end{aligned}
\end{equation}

The estimates for the terms on the right hand side of \eqref{4-10} are quite similar to \eqref{3-9}-\eqref{3-14}. By virtue of \eqref{4-3} and Lemmas \ref{lemmaA1} and \ref{lemmaA2}, one obtains
\begin{equation}\label{4-11}
\left|\int_{\hat{ E}^\pm} \partial_{x_1}\hat{\zeta} \partial_{x_1}\Bv\cdot \Bv \, dx \right| \leq [f(h(\pm t))]^{-1}M_1(\hat{E}^\pm)\|\nabla\Bv\|_{L^2(\hat{ E}^\pm)}^2
\end{equation}
and
\begin{equation}\label{4-12}
\begin{aligned}
	\left|\int_{\hat{ E}^\pm}\frac12 v_1|\Bv|^2\partial_{x_1}\hat{\zeta}\,dx\right|\leq&\frac12 [f(h(\pm t))]^{-1} \|\Bv\|_{L^2(\hat{ E}^\pm)}\|\Bv\|_{L^4(\hat{ E}^\pm)}^2\\
	\leq&\frac12[f(h(\pm t))]^{-1}M_1(\hat{E}^\pm)M_4^2(\hat{E}^\pm)\|\nabla \Bv\|_{L^2(\hat{ E}^\pm)}^3.
\end{aligned}
\end{equation}
Using Lemmas \ref{lemma1} and \ref{lemmaA1}, Young's inequality, and \eqref{4-17} gives
\begin{equation}\label{4-13}
\begin{aligned}
	\left|\int_{\hat{ E}^\pm}(\Bv\cdot \Bg)v_1\partial_{x_1} \hat{\zeta}\,dx\right|\leq& C[f(h(\pm t))]^{-1}\|\Bg\|_{L^\infty(\hat{ E}^\pm)}\|\Bv\|_{L^2(\hat{ E}^\pm)}^2\\
	\leq& C[f(h(\pm t))]^{-2} M_1^2(\hat{E}^\pm) \|\nabla\Bv\|_{L^2(\hat{ E}^\pm)}^2
\end{aligned}
\end{equation}
and
\begin{equation}\label{4-14}
\begin{aligned}
	& \left|\int_{\hat{ E}^\pm}\left[ -\partial_{x_1}\Bg\cdot \Bv+(\Bg\cdot\Bv)g_1 \right] \partial_{x_1}\hat{\zeta}\,dx\right| \\
	\leq& [f(h(\pm t))]^{-2} M_1^2(\hat{E}^\pm)\|\nabla\Bv\|_{L^2(\hat{ E}^\pm)}^2 + C \int_{\hat{ E}^\pm} | \nabla \Bg|^2 + |\Bg|^4 \, dx.
\end{aligned}
\end{equation}
Furthermore, one has
\begin{equation}\label{4-15}
\left|\int_{\hat{\Omega}_T}\hat{\zeta} (-\nabla \Bg:\nabla \Bv+\Bg\cdot\nabla\Bv\cdot \Bg)\,dx\right|\leq \frac{\mfc}{4}\int_{\hat{\Omega}_T}\hat{\zeta}|\nabla\Bv|^2\,dx+C \int_{\hat{\Omega}_t} | \nabla \Bg|^2 + |\Bg|^4 \, dx.
\end{equation}

One applies Lemmas \ref{lemmaA1}, \ref{lemmaA2}, and \ref{lemmaA5} and integration by parts to conclude 
\begin{equation*}
	\begin{aligned}
		&\left|\int_{\hat{ E}^\pm}pv_1\partial_{x_1}\hat{\zeta}\,dx\right|=[f(h(\pm t))]^{-1}\left|\int_{\hat{ E}^\pm}p{\rm div}\,\Ba\,dx\right|\\
		=&[f(h(\pm t))]^{-1}\left|\int_{\hat{ E}^\pm}2\BD(\Bv):\BD(\Ba)+(\Bv\cdot \nabla \Bg +(\Bg+\Bv)\cdot\nabla \Bv-\Delta \Bg+\Bg\cdot \nabla \Bg)\cdot \Ba\,dx\right|\\
		=&[f(h(\pm t))]^{-1}\left|\int_{\hat{ E}^\pm}2\BD(\Bv):\BD(\Ba)-\Bv\cdot \nabla \Ba\cdot \Bg -(\Bg+\Bv)\cdot\nabla \Ba\cdot \Bv+\nabla \Bg:\nabla \Ba-\Bg\cdot \nabla \Ba\cdot\Bg\,dx\right|\\
		\le &C[f(h(\pm t))]^{-1}\|\nabla\Ba\|_{L^2(\hat{ E}^\pm)}\left(\|\nabla \Bv\|_{L^2(\hat{ E}^\pm)}+\|\Bv\|_{L^4(\hat{ E}^\pm)}^2+\|\nabla\Bg\|_{L^2(\hat{ E}^\pm)}+\|\Bg\|_{L^4(\hat{ E}^\pm)}^2\right)\\
		\leq&C[f(h(\pm t))]^{-1}M_5(\hat{E}^\pm)M_1(\hat{E}^\pm)\|\nabla \Bv\|_{L^2(\hat{ E}^\pm)}\left(\|\nabla \Bv\|_{L^2(\hat{ E}^\pm)}+M_4^2(\hat{E}^\pm)\|\nabla\Bv\|_{L^2(\hat{ E}^\pm)}^2\right.\\
		&\left.+\|\nabla\Bg\|_{L^2(\hat{ E}^\pm)}+\|\Bg\|_{L^4(\hat{ E}^\pm)}^2\right),
	\end{aligned}
\end{equation*}
where $\Ba\in H_0^1(\hat{E}^\pm)$ satisfies
\[
\text{div}~\Ba =v_1 \quad \text{in}\,\, \hat{E}^{\pm}
\]  
and
\begin{equation}\label{4-5}
\|\nabla \Ba\|_{L^2(\hat{E}^\pm)}\leq M_5(\hat{E}^\pm)\|v_1\|_{L^2(\hat{E}^\pm)}.
\end{equation}
Similar to the proof of Lemma \ref{lemma3}, the constant $M_5(\hat{E}^\pm)$ in \eqref{4-5} should also be  independent of $t$. Hence, using Young's inequality gives
\begin{equation}\label{4-16}
	\begin{aligned}
		\left|\int_{\hat{ E}^\pm}pv_1\partial_{x_1}\hat{\zeta}\,dx\right|\leq&C[f(h(\pm t))]^{-1}M_1(\hat{E}^\pm) \left(\|\nabla \Bv\|_{L^2(\hat{ E}^\pm)}^2+M_4^2\|\nabla\Bv\|_{L^2(\hat{ E}^\pm)}^3\right)\\
		&+C[f(h(\pm t))]^{-2} M_1^2(\hat{E}^\pm) \|\nabla\Bv\|_{L^2(\hat{ E}^\pm)}^2 + C \int_{\hat{ E}^\pm} |\nabla \Bg|^2 + |\Bg|^4 \, dx .
	\end{aligned}
\end{equation}

Define
\begin{equation*}
\hat{y}(t)=\int_{\hat{\Omega}_T}\hat{\zeta} |\nabla \Bv|^2\,dx+\alpha\int_{\partial\Omega\cap \partial \hat{\Omega}_T}\hat{\zeta} |\Bv|^2\,ds.
\end{equation*}
Combining the previous estimates yields
\begin{equation*}\begin{aligned}
	\hat{y}'(t)=&\int_{\hat{\Omega}_T}\partial_t\hat{\zeta} |\nabla \Bv|^2\,dx+\alpha\int_{\partial\Omega\cap \partial\hat{\Omega}_T}\partial_t\hat{\zeta} |\Bv|^2\,ds\\
	\ge&\frac12 [f(h(-t))]^\frac23\left(\int_{\hat{ E}^-} |\nabla \Bv|^2\,dx+\alpha\int_{\partial\Omega\cap \partial\hat{ E}^-}|\Bv|^2\,ds\right)\\
	&+\frac12 [f(h(t))]^\frac23\left(\int_{\hat{ E}^+} |\nabla \Bv|^2\,dx+\alpha\int_{\partial\Omega\cap \partial\hat{ E}^+}|\Bv|^2\,ds\right).
\end{aligned}
\end{equation*}
It follows from Lemmas \ref{lemmaA1} and \ref{lemmaA2} that there exists a uniform constant $C>0$ such that the constants $M_1(\hat{E}^\pm)$ and $M_4(\hat{E}^\pm)$ appeared in \eqref{4-11}-\eqref{4-16} satisfy
\begin{equation*}
C^{-1} f(h(\pm t)) \leq M_1(\hat{E}^\pm)\leq C f(h(\pm t))\ \ \ \mbox{and}\ \ \ C^{-1} [f(h(\pm t))]^\frac12 \leq M_4(\hat{E}^\pm)\leq C [f(h(\pm t))]^\frac12.
\end{equation*}
Using Lemma \ref{lemma1}, one can combine \eqref{4-10}-\eqref{4-16} to conclude
\begin{equation*}
\begin{aligned}
	\hat{y}(t)
	\leq& C\|\nabla \Bv\|_{L^2(\hat{E})}^2+Cf(h(-t))\|\nabla \Bv\|_{L^2(\hat{E}^-)}^3+Cf(h(t))\|\nabla \Bv\|_{L^2(\hat{E}^+)}^3+C\int_{\hat{\Omega}_t } |\nabla \Bg|^2 + |\Bg|^4 \, dx\\
	\leq& C\left\{[f(h(-t))^{-\frac23}+f(h(t))^{-\frac23}]\hat{y}'(t)+  [\hat{y}'(t)]^\frac32\right\}+C\int_{h(-t)}^{h(t)} f^{-3}(\tau)\,d\tau\\
	\leq &C_{11}\left\{\hat{y}'(t)+  [\hat{y}'(t)]^\frac32 \right\}+C_{12}\int_{h(-t)}^{h(t)} f^{-3}(\tau)\,d\tau.
\end{aligned}
\end{equation*}

Define
\begin{equation*}
\hat\Psi(\tau)=C_{11}\left(\tau+ \tau^\frac{3}{2}  \right)\quad
\text{and}
\quad
\hat\varphi(t)=C_{13}+C_{14}\int_{h(-t)}^{h(t)} f^{-3}(\tau)\,d\tau,
\end{equation*}
where $C_{13}$ and $C_{14}$ are large enough such that
\begin{equation*}
C_{12}\int_{h(-t)}^{h(t)} f^{-3}(\tau)\,d\tau\leq \frac12\varphi(t)\ \  \text{ and }\ \
\hat\varphi(t)\ge 2 \hat\Psi(\hat\varphi'(t))\quad \text{for any }t\geq t^*.
\end{equation*}
 This holds since
\begin{equation*}
\begin{aligned}
	|\hat\varphi'(t)|=&C_{14}\left|\frac{d}{dt}\int_{h(-t)}^{h(t)}f^{-3}(\tau)\,d\tau\right|
	=C_{14} \left|\frac{h'( t)}{[f(h( t))]^3} + \frac{h'(-t)}{[f(h(-t))]^3}\right| \\
	\leq& C_{14}[f(h(t))]^{-\frac43} +C_{14}[f(h(-t))]^{-\frac43} \\
	\leq& 2 C_{14} d^{-\frac43},
\end{aligned}
\end{equation*}
where $d$ is defined in \eqref{assumpf}. The estimate \eqref{2-10-1} shows
\begin{equation*}
\hat{y}(T)=\|\hat{\zeta}(\cdot, T)^\frac12\nabla \Bv\|_{L^2(\Omega)}^2+\alpha\|\hat{\zeta}(\cdot, T)^\frac12\Bv\|_{L^2(\partial\Omega)}^2\leq C_0\int_{\hat{\Omega}_T} | \nabla \Bg|^2 + |\Bg|^4 \, dx \leq \hat\varphi(T),
\end{equation*}
provided  $C_{13}$ and $C_{14}$ are large enough. Hence it follows from Lemma \ref{lemmaA4}  that for any $ t^*\leq t\leq T$, one has
\begin{equation*}
\hat{y}(t)=\|\hat{\zeta}(\cdot, t)^\frac12\nabla \Bv \|_{L^2(\Omega )}^2+\alpha\|\hat{\zeta}(\cdot, t)^\frac12\Bv\|_{L^2(\partial \Omega)}^2\leq C_{13}+C_{14}\int_{h(-t)}^{h(t)} f^{-3}(\tau)\,d\tau.
\end{equation*}
 In particular, one has
\begin{equation*}
\|\nabla \Bv \|_{L^2(\breve{\Omega}_t )}^2+\alpha\|\Bv\|_{L^2(\partial \breve{\Omega}_t\cap \partial \Omega)}^2\leq C_{13}(\beta^*)^{-1}+C_{14}(\beta^*)^{-1}\int_{h(-t)}^{h(t)} f^{-3}(\tau)\,d\tau.
\end{equation*}
This finishes the proof of the lemma.
\end{proof}

With the help of Lemma \ref{lemma4}, one could find at least one  solution of \eqref{NS1} in a way similar to Proposition \ref{straight-existence}.
\begin{pro}\label{unbounded exits-1}  Assume that the domain $\Omega$ satisfies the assumptions in Lemma \ref{lemma4},  the problem \eqref{NS} and \eqref{flux constraint}-\eqref{BC} has a solution $\Bu=\Bv+\Bg\in H_\sigma(\Omega)$ satisfying
\begin{equation}
\|\nabla \Bu\|_{L^2(\breve{\Omega}_{t})}^2+\alpha\|\Bu\|_{L^2(\partial\breve{\Omega}_{t}\cap \partial\Omega)}^2\leq C_{16}\left(1+\int_{h(-t)}^{h(t)} f^{-3}(\tau)\,d\tau\right),
\end{equation}
where the constant $C_{16}$ depends only on $\alpha$, $\Phi$, and $\Omega$.
\end{pro}

Next, we prove that the solution $\Bu$ satisfies the estimate \eqref{1-9}.
\begin{pro}\label{unbounded exits-2}
Let $\Bu=\Bv+\Bg$ be the solution obtained in Proposition \ref{unbounded exits-1}. There exists a constant $C_{21}$ depending only on $\alpha$, $\Phi$, and $\Omega$ such that for any $t\ge 0$, one has
\begin{equation}\label{est111.5}
\|\nabla \Bu\|_{L^2(\Omega_{0,t})}^2+\alpha\|\Bu\|_{L^2(\partial\Omega_{0,t}\cap \partial\Omega)}^2\leq C_{21}\left(1+\int_0^t f^{-3}(\tau )\,d\tau\right)
\end{equation}
and
\begin{equation}\label{est11.6}
\|\nabla \Bu\|_{L^2(\Omega_{-t,0})}^2+\alpha\|\Bu\|_{L^2(\partial\Omega_{-t,0}\cap \partial\Omega)}^2\leq C_{21}\left(1+\int_{-t}^0 f^{-3}(\tau )\,d\tau\right).
\end{equation}
\end{pro}
\begin{proof}
It's sufficient to prove \eqref{est111.5} since the proof for \eqref{est11.6} is similar.  First, for $t$ suitably large,  we introduce the following truncating function
\begin{equation*}
\hat{\zeta}^+(\Bx,t)=\left\{
\begin{aligned}
	&0, ~~~~~~&&\text{ if }x_1\in (-\infty,\,0),\\
	&\beta^* x_1, ~~~~~~&&\text{ if }x_1\in [0,\,1],\\
	&\beta^*,~~~~~~~~~~~~~~ &&\text{ if }x_1\in (1,\,h_R(t)),\\
	&\frac{h(t)-x_1}{f(h(t))},~~~~~~&&\text{ if }x_1\in [h_R(t),\, h(t)],\\
	&0,~~~~~~~~~~~~~~ &&\text{ if }x_1\in (h(t),\, \infty),
\end{aligned}\right.
\end{equation*}
where $h_R(t)$ is defined in \eqref{defh}. Taking the test function $\Bp=\hat{\zeta}^+ \Bv$ in \eqref{2-11} and following  proof of  Lemma \ref{lemma4} similarly, one has
\begin{equation*}
\begin{aligned}
	\hat{y}^+\leq& C\left\{\int_{ E_0}|\nabla \Bv|^2\,dx+\left(\int_{ E_0}|\nabla \Bv|^2\,dx\right)^\frac32+(\hat{y}^+)'+[(\hat{y}^+)']^\frac32\right\}+C\int_{0}^{h(t)}f^{-3}(\tau)\,d\tau\\
	\leq& C\left\{1+[(\hat{y}^+)']^\frac32\right\}+C\int_{0}^{h(t)}f^{-3}(\tau)\,d\tau\\
	\leq& C_{17}[ (\hat{y}^+)']^\frac32 +C_{18}\left(1+\int_{0}^{h(t)}f^{-3}(\tau)\,d\tau\right),
\end{aligned}
\end{equation*}
where $E_0 = \{\Bx\in \Omega:\ 0<x_1 < 1 \}$ and
\begin{equation*}
\hat{y}^+(t)=\int_{\Omega}\hat{\zeta}^+|\nabla \Bv |^2\,dx+\alpha\int_{\partial\Omega}\hat{\zeta}^+| \Bv |^2\,ds.
\end{equation*}
Set
\begin{equation*}
 \delta_1=\frac12,\quad  \tilde\Psi(\tau)=C_{17}\tau^\frac32,\quad \text{and}\quad \tilde\varphi(t)=C_{19}+C_{20}\int_{0}^{h(t)}f^{-3}(\tau)\,d\tau.
\end{equation*}
Similar to the proof of Lemma \ref{lemma4}, we choose the constants $C_{19}$ and $C_{20}$ to be sufficiently large such that
\begin{equation*}
C_{18}\left(1+\int_{0}^{h(t)}f^{-3}(\tau)\,d\tau\right)\leq \frac{1}{2}\tilde\varphi(t)\ \ \text{ and }\ \ \tilde\varphi(t)\ge 2 \tilde\Psi(\tilde\varphi'(t)).
\end{equation*}
It also follows from the proof of Lemma \ref{lemma4} that one has
\begin{equation*}
\hat{y}^+(t) \leq C_{13} + C_{14}\int_{h(-t)}^{h(t)} f^{-3}(\tau)\,d\tau
\end{equation*}
and
\begin{equation*}
\left|\frac{d}{dt}\int_{h(-t)}^{h(t)} f^{-3}(\tau)\,d\tau\right|\leq  C  d^{-\frac43}.
\end{equation*}
Hence, it holds that 
\begin{equation*}
\liminf_{t\rightarrow + \infty} \frac{\hat{y}^+(t) }{\tilde{z}(t)} = 0,
\end{equation*}
where $\tilde{z}(t)=\frac{1}{108C_{17}^2}t^3$ is a nonnegative solution to the ordinary differential equation
\begin{equation*}
\tilde{z}(t)=\delta_1^{-1}\Psi(\tilde{z}'(t))=2C_{17}[\tilde{z}'(t)]^\frac32.
\end{equation*}
It follows from Lemma \ref{lemmaA4} that one has
\begin{equation}\label{4-24}
\hat{y}^+(t)\leq C_{19}+C_{20}\int_{0}^{h(t)} f^{-3}(\tau)\,d\tau.
\end{equation}
With the help of \eqref{4-17}, one has further
\begin{equation}\label{4-25}
\begin{aligned}
	\int_{0}^{h(t)} f^{-3}(\tau)\,d\tau=&\int_{0}^{h_R(t)} f^{-3}(\tau)\,d\tau+\int_{h_R(t)}^{h(t)} f^{-3}(\tau)\,d\tau\\
	\leq&\int_{0}^{h_R(t)} f^{-3}(\tau)\,d\tau+  \max_{\xi\in [h_R(t),\, h(t)]}f^{-3}(\xi )\cdot \beta^* f(h(t))\\
	\leq&\int_{0}^{h_R(t)} f^{-3}(\tau)\,d\tau+ 2^3 \beta^* f^{-2}(h(t))\\
	\leq&\int_{0}^{h_R(t)} f^{-3}(\tau)\,d\tau+  2^3 \beta^*  d ^{-2},
\end{aligned}
\end{equation}
where $h_R(t)$ is defined in \eqref{defh}. Combining  \eqref{4-24} and \eqref{4-25} yields
\begin{equation*}
\|\nabla \Bv\|_{L^2(\Omega_{0,h_R(t)})}^2+\alpha\|\Bv\|_{L^2(\partial\Omega_{0,h_R(t)}\cap \partial\Omega)}^2\leq C_{21}\left(1+\int_{0}^{h_R(t)} f^{-3}(\tau)\,d\tau\right).
\end{equation*}
This, together with Lemma \ref{lemma1}, finishes the proof of the proposition.
\end{proof}

Hence we finish the proof for Part (i) of Theorem \ref{unbounded channel} in the case that the range of $k(t)$ is $(-\infty, \infty)$.


{\bf Case 2. The range of $k(t)$ is $(-L, \, R)$, $0< L, R< \infty.$}  In this case, it holds that
\be \nonumber
\int_{-\infty}^{+\infty} f^{-\frac53}(\tau) \, d\tau= R+ L < \infty.
\ee
Let $v^T$ be the solution of the approximate problem \eqref{aNS} on $\Omega_T$, which is obtained in Proposition \ref{appro-existence} and satisfies \eqref{2-10-1}. Hence, one has 
\be \nonumber
\int_{\Omega_T} |\nabla v^T |^2 \, dx + \alpha \int_{\partial \Omega_T \cap \partial \Omega} |v^T|^2 \, ds
\leq C \int_{-T}^T f^{-3}(\tau)\, d\tau  \leq C \int_{-\infty}^{+\infty} f^{-\frac53} (\tau) \, d\tau.
\ee
With the help of this uniform estimate and Lemma \ref{lemma1}, there exists at least one solution of \eqref{NS1}, which satisfies the estimate
\be \label{estimate-case2}
\| \nabla \Bu\|_{L^2(\Omega)}^2 + \alpha \|\Bu\|_{L^2(\partial \Omega)}^2 \leq C.
\ee
Hence we finish the proof for Part (i) of Theorem \ref{unbounded channel} in the case that the range of $k(t)$ is $(-L, R)$.

\vspace{3mm}
{\bf Case 3. The range of $k(t)$ is $(-L, \, \infty)$ or $(-\infty, \, R)$, $0<L, R< \infty$.}  Without loss of generality, we assume that
the range of $k(t)$ is $(-L, \infty)$. In this case, $h(t)$ is defined on $(-L,\infty)$. By \eqref{4-2}, one has $h_R(t)=h(t)-\beta^*f(h(t))>0$ for suitably large $t$. Then we introduce the new truncating function as follows,
\begin{equation} \label{cut-off-hat-new}
\hat{\zeta}^L_T (\Bx,t)=\left\{
\begin{aligned}
		&\beta^*,~~~~~~~~~~~~~~ &&\text{ if }x_1\in (-T,\, h_R(t)),\\
	&\frac{h(t)-x_1}{f(h(t))},~~~~~~&& \text{ if } x_1\in [h_R(t),\, h(t)],\\
	&0,~~~~~~~~~~~~~~ &&\text{ if } x_1\in (h(t),\,\infty),
\end{aligned}\right.
\end{equation}
where $h_R(t)$ is defined in \eqref{defh}.  Denote 
\begin{equation}\label{defhatt}
	\hat{t}=\sup\{t>0:\ h_R(t) \leq 0 \}.
\end{equation}

\begin{lemma}\label{lemma-case3}  Assume that the domain $\Omega$ satisfies \eqref{assumpf''},  and
\be \nonumber
\int_{-\infty}^0 f^{-\frac53}(\tau) \, d\tau = L< \infty, \ \ \ \ \ \int_0^{+\infty} f^{-\frac53}(\tau) \, d\tau = \infty.
\ee
Let $\Bv^T$ be the solution of the approximate problem \eqref{aNS} in $\Omega_{-T, h(T)}$, which is obtained in Proposition \ref{appro-existence} and satisfies the energy estimate \eqref{2-10-1}. Then there exists a positive constant $C_{22}$ independent of $t$ and $T$ such that for any $\hat{t} \leq t\leq T$, one has
\begin{equation}\label{case3-1}
	\|\nabla \Bv^T\|_{L^2(\Omega_{-T,\,  h_R(t)} )}^2+ \alpha\|\Bv^T\|_{L^2(\partial\Omega_{-T, \,h_R(t)}\cap \partial\Omega)}^2\leq C_{22}\left(1+\int_0^{h(t)} f^{-3}(\tau)\,d\tau\right),
\end{equation}
 where $h_R(t)$ is defined in \eqref{defh} and $C_{22}$ is independent of $T$. 
\end{lemma}

\begin{proof}
The superscript $T$ is omitted throughout the proof.
We follow the proof of Lemma \ref{lemma4} by taking the test function $\Bp=\hat{\zeta}^L_T \Bv$ in \eqref{2-11}. Similarly, one has
\be \label{case3-4}
\begin{aligned}
	\hat{y}^L(t) & \leq C \left\{ (\hat{y}^L)^{\prime} + \left[ (\hat{y}^L)^\prime    \right]^{\frac32}   \right\} + C \int_{-T}^{h(t)}
	f^{-3}(\tau) \, d\tau \\
	& \leq C \left\{ (\hat{y}^L)^{\prime} + \left[ (\hat{y}^L)^\prime    \right]^{\frac32}   \right\} + C \left(1+ \int_{0}^{h(t)}
	f^{-3}(\tau) \, d\tau \right) ,
\end{aligned}
\ee
where
\be \nonumber
\hat{y}^L (t) = \int_{\Omega_{-T, h(T)}} \hat{\zeta}^L_T |\nabla \Bv|^2 \, dx + \alpha \int_{\partial \Omega_{-T, h(T)} \cap \partial \Omega} \hat{\zeta}^L_T |\Bv|^2 \, ds.
\ee
Hence, the same argument as in proof of Lemma \ref{lemma4} yields
\be \label{case3-6}
\hat{y}^L(t) \leq C \left( 1  + \int_0^{h(t)} f^{-3} (\tau) \, d\tau \right).
\ee
This completes the proof of the lemma.
\end{proof}

\begin{pro}\label{prop-case3}
Assume that the domain $\Omega$ satisfies the assumptions of Lemma \ref{lemma-case3},  the problem \eqref{NS} and \eqref{flux constraint}-\eqref{BC} has a solution $\Bu=\Bv+\Bg\in H_\sigma(\Omega)$ satisfying
\begin{equation}
	\|\nabla \Bu\|_{L^2(\Omega_{0, t} )}^2 + \alpha\|\Bu\|_{L^2(\partial \Omega_{0, t} \cap \partial\Omega)}^2\leq C_{23}\left(1+\int_0^t f^{-3}(\tau)\,d\tau\right)
\end{equation}
and
\be
\|\nabla \Bu\|_{L^2(\Omega_{-t , 0} )}^2 + \alpha\|\Bu\|_{L^2(\partial \Omega_{-t, 0} \cap \partial\Omega)}^2\leq C_{23},
\ee
where the constant $C_{23}$ depends only on $\alpha$, $\Phi$, and $\Omega$.
\end{pro}
\begin{proof}
With the help of Lemma \ref{lemma-case3}, one can find at least one solution $\Bu = \Bv + \Bg$ of \eqref{NS1} in a way similar to Proposition \ref{straight-existence}. Following the same argument in the estimate \eqref{4-25} yields
\be \label{case3-10}
\begin{aligned}
	 &\|\nabla \Bv\|_{L^2 (\Omega_{0, h_R(t)}) }^2 +  \alpha \|\Bv\|_{L^2(\partial \Omega_{0, h_R(t)} \cap \partial \Omega)}^2 \\
	\leq &
	C \left(1 +  \int_0^{h(t)} f^{-3}(\tau) \, d\tau \right)\\
	\leq & C \left(1 +  \int_0^{h_R(t)} f^{-3}(\tau) \, d\tau \right) + C \int_{h_R(t)}^{ h(t)} f^{-3}(\tau) \, d\tau \\
	\leq & C_{23} \left(1+  \int_0^{h_R(t)} f^{-3}(\tau) \, d\tau \right),
\end{aligned}
\ee
where $h_R(t)$ is defined in \eqref{defh}. Hence one has
\be \label{case3-11}
\|\nabla \Bv\|_{L^2(\Omega_{0,h_R(t)} )}^2  + \alpha \|\Bv\|_{L^2(\partial \Omega_{0, h_R(t)} \cap \partial \Omega)}^2  \leq C_{23} \left(1 +
\int_0^{h_R(t)} f^{-3}(\tau) \, d\tau \right).
\ee
On the other hand, according to Lemma \ref{lemma-case3}, it holds that
\be \label{case3-12}
\|\nabla \Bv\|_{L^2(\Omega_{-T, 0} )}^2 + \alpha \|\Bv\|_{L^2(\partial \Omega_{-T, 0} \cap \partial \Omega)}^2 \leq C\left(1 + \int_0^{h({\hat{t}}) } f^{-3}(\tau) \, d\tau\right) \leq C_{23}.
\ee
Combining the estimates \eqref{case3-11}-\eqref{case3-12} and Lemma \ref{lemma1} finishes the proof of the proposition.
\end{proof}

Hence we finish the proof for Part (i) of Theorem \ref{unbounded channel} in the case that the range of $k(t)$ is $(-L, +\infty)$. The same proof applies to the case that the range of $k(t)$ is $(-\infty, R)$. The proof of existence for flows in channels with unbounded outlets is completed.

We are ready to prove the uniqueness of solutions when the flux $\Phi$ is small.
In fact, one can derive some refined estimate for the local Dirichlet norm of $\Bu$, which plays an important role in proving the uniqueness when $\Phi$ is small.
\begin{pro}\label{decay rate-right}
Let $\Bu=\Bv+\Bg$ be the solution obtained in Part (i) of Theorem \ref{unbounded channel}. Assume further that either
\begin{equation}\label{4-27-2}
	\left|\int_0^{\infty} f^{-3}(\tau)\,d\tau\right|=\infty,\ \ \ \lim_{t\to \infty}f'(t)= 0,
\end{equation}
or
\begin{equation}\label{4-27-3}
	\left|\int_0^{\infty}f^{-3}(\tau)\,d\tau\right|<\infty,\ \ \ \lim_{t\to \infty}\frac{\sup_{\tau \ge t} f'(\tau)}{\left|\int_t^{ \infty} f^{-3}(\tau)\,d\tau\right|^\frac12}= 0.
\end{equation}
 Then there exists a constant $C_{31}$ depending only on $\alpha, \Phi$, and $\Omega$ such that for any $t\geq 0$, one has
\begin{equation*}
\|\nabla \Bu\|_{L^2(\Omega_{t-\beta^* f(t),t})}^2+\alpha\|\Bu\|_{L^2(\partial\Omega_{t-\beta^* f(t),t}\cap \partial\Omega)}^2\leq \frac{C_{31}}{  f^2(t)}.
\end{equation*}

\end{pro}
\begin{proof} We divide the proof into three steps.

{\em Step 1. Truncating function.}
Clearly,
\begin{equation*}
	\frac{d}{dt}(t\pm \beta^* f(t))=1\pm \beta^* f'(t) \ge \frac12.
\end{equation*}
Hence the function $t\pm \beta^* f(t)$ are strictly monotone increasing functions on $\mathbb{R}$.
 For any fixed $T>0$, one can uniquely define the numbers $\hat{T},T_1$, and $T_2$ by
\begin{equation*}
\hat{T}= T - \beta^* f(T), \ \ \hat{T} = T_1 + \beta^* f(T_1), \ \ \mbox{and}\ \ T= T_2 - \beta^* f(T_2).
\end{equation*}

 Let $T_0\ge 1$ be a positive constant to be determined.
We introduce two monotone increasing functions $m_i(t)(i=1,2)$ such that for any  $t\in [0,t_1]$,
\begin{equation}\label{4-26}
\left\{\begin{aligned}
&	\frac{d}{dt}m_1(t)=f^{\frac53}(T_1-m_1(t)),\\
&	\frac{d}{dt}m_2(t)=f^{\frac53}(T_2+m_2(t)),\\
&	m_i(0)=0, i=1, 2,
\end{aligned}\right.
\end{equation}
where $t_1$ is the number satisfying
\begin{equation}\label{4-21}
m_1(t_1)=T_1-T_0.
\end{equation}
Noting that  $\frac{d}{dt}m_1(t)\ge d^\frac{5}{3}>0$, the number $t_1$ is well-defined. Then  we define the new truncating function $\tilde{\zeta}^+$ as follows,
\begin{equation*}
\tilde{\zeta}^+(\Bx,t)=\left\{
\begin{aligned}
	&\frac{x_1-T_1+m_1(t)}{f(T_1-m_1(t))},~~&&\text{ if }x_1\in [T_1-m_1(t),T_1^*(t)],\\
	&\beta^*,~~~~~~~~~ &&\text{ if }\begin{array}{r}x_1\in (T_1^*(t),
		T_2^*(t))\end{array},\\
	&\frac{T_2+m_2(t)-x_1}{f(T_2+m_2(t))},~~&&\text{ if }x_1\in [T_2^*(t),T_2+m_2(t)],\\
	&0,~~~~~~~~ &&\text{ if }x_1\in (-\infty,T_1-m_1(t) )\cup (T_2+m_2(t),\infty),
\end{aligned}\right.
\end{equation*}
where
\[
T_1^*(t)=T_1-m_1(t)+\beta^* f(T_1-m_1(t))\quad \text{and}\quad T_2^*(t)=T_2+m_2(t)-\beta^* f(T_2+m_2(t)).
\]
With the help of \eqref{4-26},  similar to \eqref{4-3}-\eqref{4-4}, one has
\begin{equation*}
|\nabla \tilde{\zeta}^+|=|\partial_{x_1}\tilde{\zeta}^+|=\frac{1}{f(T_i\pm m_i(t))}
\quad\text{and}\quad
|\partial_t\tilde{\zeta}^+|\geq \frac12 [f(T_i\pm m_i(t))]^\frac23  \,\,\text{in}\,\, \operatorname{supp} \nabla \tilde{\zeta}^+=\operatorname{supp} \partial_t \tilde{\zeta}^+.
\end{equation*}

{\em Step 2. Energy estimate. } Taking the test function $\Bp=\tilde{\zeta}^+ \Bv$ in \eqref{2-11} and following the proof of Lemma \ref{lemma4} yield that for any $t\in [0,t_1]$,
\begin{equation}\label{eq130.5}
	\begin{aligned}
\tilde{y}^+\leq & C_{24}\left\{[f^{-\frac23}(T_2+m_2(t))+f^{-\frac23}(T_1-m_1(t))](\tilde{y}^+)'+[(\tilde{y}^+)']^\frac32\right\}\\
 &+C_{25}\int_{T_1-m_1(t)}^{T_2+m_2(t)} f^{-3}(\tau) \,d\tau,
\end{aligned}
\end{equation}
 where
\begin{equation*}
\tilde{y}^+(t)=\int_{\Omega}\tilde{\zeta}^+|\nabla \Bv |^2\,dx+\alpha\int_{\partial\Omega}\tilde{\zeta}^+| \Bv |^2\,ds.
\end{equation*}
By virtue of Propositions \ref{unbounded exits-2} and \ref{prop-case3}, one has
\begin{equation}\label{4-22-1}
	\begin{aligned}
		\tilde{y}^+(t_1)\leq &C \left(1+\int_{0}^{T_2+m_2(t_1)}f^{-3}(\tau) \,d\tau\right)\\
		\leq& C\int_{T_0}^{T_2+m_2(t_1)}f^{-3}(\tau) \,d\tau+C \left(1+\int_{0}^{T_0}f^{-3}(\tau) \,d\tau\right).
	\end{aligned}
\end{equation}
{\em Step 3. Analysis for flows in channels satisfying \eqref{4-27-2}.} 
Firstly, under the assumption \eqref{4-27-2}, choose $T_0$ and $T$ to be sufficiently large such that
\begin{equation}\label{4-23-2}
	1+\int_{0}^{T_0}f^{-3}(\tau)\,d\tau\leq 2\int_{0}^{T_0}f^{-3}(\tau)\,d\tau\leq 2\int_{T_0}^{T_1}f^{-3}(\tau)\,d\tau.
\end{equation}
Recalling that $T_1-m_1(t_1)=T_0$, one uses \eqref{4-22-1} and \eqref{4-23-2} to obtain
\begin{equation*}
	\tilde{y}^+(t_1)\leq C_{26}\int_{T_1-m_1(t_1)}^{T_2+m_2(t_1)}f^{-3}(\tau) \,d\tau.
\end{equation*}
Now, we set $\delta_1=\frac12$,
\begin{equation*}
	\Psi(t,\tau)=C_{24}\left\{ [f^{-\frac23}(T_2+m_2(t))+f^{-\frac23}(T_1-m_1(t))]\tau +\tau ^\frac32 \right\},
\end{equation*}
and
\begin{equation*}
	\varphi(t)=(2C_{25}+C_{26})\int_{T_1-m_1(t)}^{T_2+m_2(t)}f^{-3}(\tau)\,d\tau+C_{27}f^{-2}(T),
\end{equation*}
where $C_{27}$ is to be chosen. Thus, one has
\begin{equation}\label{eqyPsiphi}
	\tilde{y}^+\leq \Psi(t,(\tilde{y}^+)')+\frac12 \varphi(t)
\quad
\text{and}
\quad
	\tilde{y}^+(t_1)\leq \varphi(t_1).
\end{equation}
Moreover, according to \eqref{4-26} and the definition of $\varphi(t)$ and $\Psi(t,\tau)$, it holds that
\begin{equation*}
	\begin{aligned}
		\varphi'(t)
		=&(2C_{25}+C_{26})\left(\frac{m_2'(t)}{f^3(T_2+m_2(t))}+\frac{m_1'(t)}{f^3(T_1-m_1(t))}\right)\\
		=&(2C_{25}+C_{26})\left(f^{-\frac43}(T_2+m_2(t))+f^{-\frac43}(T_1-m_1(t))\right).
	\end{aligned}
\end{equation*}
Therefore, it holds that
\begin{equation*}
	\begin{aligned}
		\Psi(t, \varphi'(t))=&C_{24}\left\{ [f^{-\frac23}(T_2+m_2(t))+f^{-\frac23}(T_1-m_1(t))]\varphi'(t)+[\varphi'(t) ]^\frac32 \right\}\\	
		\leq &C\left[f^{-2}(T_2+m_2(t))+f^{-2}(T_1-m_1(t))\right]\\
		=&C\left[ 2f^{-2}(T)+2\int_{T_1-m_1(t)}^T(f'f^{-3})(\tau)\,d\tau-2\int_T^{T_2+m_2(t)}(f'f^{-3})(\tau)\,d\tau\right]\\
		\leq &2C\left( f^{-2}(T)+\int_{T_1-m_1(t)}^{T_2+m_2(t)}(f'f^{-3})(\tau)\,d\tau\right)\\
		\leq&2C\left( f^{-2}(T)+ \gamma_0(T_0) \int_{T_1-m_1(t)}^{T_2+m_2(t)}f^{-3}(\tau)\,d\tau\right),
	\end{aligned}
\end{equation*}
where
\begin{equation*}
	\gamma_0(T_0): =\sup_{t\ge T_0}|f'(t)|.
\end{equation*}
According to the assumption \eqref{4-27-2}, one could choose sufficiently large $T_0$ and $C_{27}$ such that
\begin{equation}\label{eq133.5}
	\varphi(t)\ge 2\Psi(t, \varphi'(t)).
\end{equation}
Now, it follows from Lemma \ref{lemmaA4} that one has
\begin{equation}\label{4-29-2}
	\tilde{y}^+(t)\leq \varphi(t) \text{ for any }t\in[0,t_1].
\end{equation}

{\em Step 4. Analysis for flows in channels satisfying \eqref{4-27-3}. }
If instead of \eqref{4-27-2}, the assumption \eqref{4-27-3} holds, we choose $T_0$ and $T$ to be sufficiently large such that
\begin{equation*}
	\int_{T_0}^\infty f^{-3}(\tau)\,d\tau\leq 1\quad \text{and}\quad
	\int_{T_0}^{T_1} f^{-3}(\tau)\,d\tau\ge \frac12 \int_{T_0}^\infty f^{-3}(\tau)\,d\tau.
\end{equation*}
Hence, it holds that 
\begin{equation}\label{4-23-3}
	1+\int_{0}^{T_0}f^{-3}(\tau)\,d\tau\leq 1+\int_{0}^{\infty}f^{-3}(\tau)\,d\tau\leq \frac{2C}{\int_{T_0}^\infty f^{-3}(\tau)\,d\tau} \int_{T_0}^{T_1}f^{-3}(\tau)\,d\tau.
\end{equation}
Recalling that  $T_1-m_1(t_1)=T_0$, one combines \eqref{4-22-1} and \eqref{4-23-3} to obtain
\begin{equation*}
	\begin{aligned}
		\tilde{y}^+(t_1)\leq &C \int_{T_0}^{T_2+m_2(t_1)}f^{-3}(\tau) \,d\tau+C \left(1+\int_{0}^{T_0}f^{-3}(\tau) \,d\tau\right)\\
		\leq& \frac{C_{26}}{\int_{T_0}^\infty f^{-3}(\tau)\,d\tau}\int_{T_1-m_1(t_1)}^{T_2+m_2(t_1)}f^{-3}(\tau) \,d\tau.
	\end{aligned}
\end{equation*}
Now,  set $\delta_1=\frac12$,
\begin{equation*}
	\Psi(t,\tau)=C_{24}\left([f^{-\frac23}(T_2+m_2(t))+f^{-\frac23}(T_1-m_1(t))]\tau +\tau ^\frac32 \right),
\end{equation*}
and
\begin{equation*}
	\varphi(t)=\left(2C_{25}+\frac{C_{28}}{\int_{T_0}^\infty f^{-3}(\tau)\,d\tau}\right)\int_{T_1-m_1(t)}^{T_2+m_2(t)}f^{-3}(\tau)\,d\tau+C_{29}f^{-2}(T),
\end{equation*}
where $C_{29}$ is to be determined. Then the inequalities in \eqref{eqyPsiphi} still hold.
Moreover, according to \eqref{4-26} and the definition of $\varphi(t)$ and $\Psi(t,\tau)$, one has
\begin{equation*}
	\begin{aligned}
		\varphi'(t)
		=&\left(2C_{25}+\frac{C_{28}}{\int_{T_0}^\infty f^{-3}(\tau)\,d\tau}\right)\left(\frac{m_2'(t)}{f^3(T_2+m_2(t))}+\frac{m_1'(t)}{f^3(T_1-m_1(t))}\right)\\
		=&\left(2C_{25}+\frac{C_{28}}{\int_{T_0}^\infty f^{-3}(\tau)\,d\tau}\right)\left(f^{-\frac43}(T_2+m_2(t))+f^{-\frac43}(T_1-m_1(t))\right).
	\end{aligned}
\end{equation*}
Hence,
\begin{equation*}
	\begin{aligned}
	\Psi(t, \varphi'(t))=&C_{24}\left([f^{-\frac23}(T_2+m_2(t))+f^{-\frac23}(T_1-m_1(t))]\varphi'(t)+[\varphi'(t) ]^\frac32 \right)\\	
		\leq &\frac{C}{\left(\int_{T_0}^\infty f^{-3}(\tau)\,d\tau\right)^\frac32}\left[f^{-2}(T_2+m_2(t))+f^{-2}(T_1-m_1(t))\right]\\	
		=&\frac{2C}{\left(\int_{T_0}^\infty f^{-3}(\tau)\,d\tau\right)^\frac32}\left( f^{-2}(T)+\int_{T_1-m_1(t)}^T(f'f^{-3})(\tau)\,d\tau-\int_T^{T_2+m_2(t)}(f'f^{-3})(\tau)\,d\tau\right)\\
		\leq&\frac{2C}{\left(\int_{T_0}^\infty f^{-3}(\tau)\,d\tau\right)^\frac32}\left( f^{-2}(T)+ \int_{T_1-m_1(t)}^{T_2+m_2(t)}| f'f^{-3}|(\tau)\,d\tau\right)\\
		\leq&2C\left( \frac{f^{-2}(T)}{\left(\int_{T_0}^\infty f^{-3}(\tau)\,d\tau\right)^\frac32} +\frac{\gamma_1(T_0)}{\int_{T_0}^\infty f^{-3}(\tau)\,d\tau}\int_{T_1-m_1(t)}^{T_2+m_2(t)}f^{-3}(\tau)\,d\tau\right),
	\end{aligned}
\end{equation*}
where
\begin{equation*}
	\gamma_1(T_0) =\frac{\sup_{t\ge T_0}|f'(t)|}{\left(\int_{T_0}^\infty f^{-3}(\tau)\,d\tau\right)^\frac12}.
\end{equation*}
According to the assumption \eqref{4-27-3}, one could choose sufficiently large $T_0$ and $C_{29}$ such that \eqref{eq133.5} holds. One can also get \eqref{4-29-2} with the aid of Lemma \ref{lemmaA4}.

{\em Step 5. Growth estimate.}
In particular, taking $t=0$ in \eqref{eq130.5} gives
\begin{equation}\label{4-28}
\|\nabla \Bv\|_{L^2(\Omega_{\hat{T},T})}^2+\alpha\|\Bv\|_{L^2(\partial\Omega_{\hat{T},T}\cap \partial\Omega)}^2\leq C\int_{T_1}^{T_2}f^{-3}(\tau)\,d\tau+Cf^{-2}(T).
\end{equation}
Finally, using the inequality \eqref{4-17}, one has
\begin{equation}\label{4-29}
\begin{aligned}
	\int_{T_1}^{T_2}f^{-3}(\tau)\,d\tau= &\int_{T_1}^{\hat{T}}f^{-3}(\tau)\,d\tau+\int_{\hat{T}}^Tf^{-3}(\tau)\,d\tau+  \int_{T}^{T_2}f^{-3}(\tau)\,d\tau\\
	\leq& 27\beta^* f^{-3}(\hat{T})f(T_1)+8\beta^* f^{-3}(T)f(T)+ 27 \beta^*f^{-3}(T)f(T_2)\\
	\leq& 54 \beta^* f^{-2}(\hat{T})+8\beta^* f^{-2}(T)+ 27 \beta^* f^{-3}(T) f(T_2)\\
	\leq &Cf^{-2}(T).
\end{aligned}
\end{equation}
Combining \eqref{4-28} and \eqref{4-29} gives
\begin{equation}\label{4-30}
\|\nabla \Bv\|_{L^2(\Omega_{\hat{T},T})}^2+\alpha\|\Bv\|_{L^2(\partial\Omega_{\hat{T},T}\cap \partial\Omega)}^2\leq \frac{C_{30}}{f^2(T)}.
\end{equation}
This, together with Lemma \ref{lemma1}, finishes the proof of the proposition.
\end{proof}

Similarly, one can also prove the estimate for $t<0$.
\begin{pro}\label{decay rate-left}
Let $\Bu=\Bv+\Bg$ be the solution obtained in Part (i) of Theorem \ref{unbounded channel}. Assume further that either
\begin{equation}\label{4-27-4}
	\int_{-\infty}^{0} f^{-3}(\tau)\,d\tau= \infty,\ \ \ \lim_{t\to -\infty}f'(t)= 0,
\end{equation}
or
\begin{equation}\label{4-27-5}
\int_{-\infty}^{0} f^{-3}(\tau)\,d\tau <\infty,\ \ \ \lim_{t\to -\infty}\frac{\sup_{ \tau \leq t} |f'(\tau) | }{\left|\int_{- \infty}^t  f^{-3}(\tau)\,d\tau\right|^\frac12}= 0.
\end{equation}
 Then there exists a constant $C_{31}$ depending only on $\alpha, \Phi$, and $\Omega$ such that for any $t\geq 0$, one has
\begin{equation*}
\|\nabla \Bu\|_{L^2(\Omega_{-t,-t+\beta^* f(-t)})}^2+\alpha\|\Bu\|_{L^2(\partial\Omega_{-t,-t+\beta^* f(-t)}\cap \partial\Omega)}^2\leq \frac{C_{31}}{  f^2(-t)}.
\end{equation*}

\end{pro}

With the help of the decay rate of the local Dirichlet norm of solutions obtained in Propositions \ref{decay rate-right}-\ref{decay rate-left}, we are ready to prove the uniqueness of solution when the flux $\Phi$ is sufficiently small.
\begin{pro}\label{unique-unbounded} Under the assumptions of Propositions \ref{decay rate-right}-\ref{decay rate-left}, there exists a constant $\Phi_2$ such that for any  $\Phi\in [0, \Phi_2)$, the solution $\Bu$ obtained in Part (i) of Theorem \ref{unbounded channel} is unique.
\end{pro}
\begin{proof}We divide the proof into three steps.

{\em Step 1. Set up.}
Assume that $\tilde{\Bu}$ is also a solution of problem \eqref{NS} and \eqref{flux constraint}-\eqref{BC} satisfying
\begin{equation*}
\|\nabla \tilde\Bu \|_{L^2(\Omega_{t})}^2+\|\tilde\Bu \|_{L^2(\partial\Omega_{t}\cap \partial\Omega)}^2\leq C\left(1+\int_{-t}^t f^{-3}(\tau)\,d\tau\right)\quad \text{for any}\,\, t>0.
\end{equation*}
 Then $\oBu :=
\tilde\Bu-\Bu$ is a weak solution to the problem \eqref{3-15}. Let $\hat{\zeta}(\Bx,t)$ be the truncating function defined in \eqref{cut-off-hat}.  Similar to \eqref{3-16}-\eqref{3-20},  multiplying the first equation in \eqref{3-15} by $\hat{\zeta}\oBu $ and integrating over $\Omega$ yield
\begin{equation}\label{4-31}
\begin{aligned}
	&\mfc \int_{\Omega}\hat{\zeta}| \nabla \oBu |^2\,dx+\alpha \int_{\partial\Omega}\hat{\zeta}|\oBu |^2\,ds \\
	\leq &\int_{\breve{\Omega}_t}\hat{\zeta}\oBu \cdot\nabla \oBu \cdot \Bu \,dx + \int_{\hat{ E}}  \hat{\zeta} \oBu \cdot\nabla \oBu \cdot \Bu\,dx+\left|\int_{\hat{ E}}\partial_{x_1}\hat{\zeta} \partial_{x_1}\oBu\cdot \oBu\,dx\right|\\
	&+\int_{\hat{ E}}\left[ (\oBu \cdot \Bu)\overline{u}_1+\frac12|\oBu |^2(u_1+\overline{u}_1)+p\overline{u}_1   \right] \partial_{x_1}\hat{\zeta}\,dx,
\end{aligned}
\end{equation}
where $\breve{\Omega}_t$ and $\hat{E}^\pm$ are defined in \eqref{4-6} and \eqref{4-7}.

{\em Step 2. Estimate for the Dirichlet norm.} Let $\breve{\Omega}_t^i=\{\Bx\in\Omega:~x_1\in(A_{i-1},A_i),~i=1,2,\cdots, N(t)\}$. Here the sequence $\{A_i\}$ satisfies $h_L(t)=A_0<\cdots<A_j=0 <A_{j+1}<\cdots < A_{N(t)}=h_R(t)$,
\begin{equation*}
\frac{\beta^*}{2}f(A_i) \leq  A_{i+1}-A_{i}\leq \beta^* f(A_{i}) \text{ for any }0\leq i\leq j-1
\end{equation*}
and
\begin{equation*}
\frac{\beta^*}{2} f(A_{i+1}) \leq  A_{i+1}-A_{i}\leq \beta^* f(A_{i+1})\text{ for any }j\leq i\leq  N(t)-1.
\end{equation*}
By Lemmas \ref{lemmaA1}, \ref{lemmaA2}, and \ref{lemma1}, and Propositions \ref{decay rate-right}-\ref{decay rate-left}, one has
\begin{equation}\label{4-32}
\begin{aligned}
	\int_{\breve{\Omega}_{t}}\hat{\zeta}\oBu \cdot\nabla \oBu \cdot \Bu\,dx
	\leq & \beta^* \sum_{i=1}^{N(t)}\int_{\breve{\Omega}_t^i }|\oBu \cdot\nabla \oBu \cdot \Bu|\,dx \\
	\leq &  \beta^* \sum_{i=1}^{N(t)}\|\nabla\oBu \|_{L^2(\breve{\Omega}_t^i )} \|\oBu \|_{L^4(\breve{\Omega}_t^i )}(\|\Bv\|_{L^4(\breve{\Omega}_t^i )}+\|\Bg\|_{L^4(\breve{\Omega}_t^i )}) \\
	\leq&\beta^* \sum_{i=1}^{N(t)}\|\nabla\oBu \|_{L^2(\breve{\Omega}_t^i )}^2(M_4^2\|\nabla\Bv\|_{L^2(\breve{\Omega}_t^i)}+M_4\|\Bg\|_{L^4(\breve{\Omega}_t^i )}) \\
	\leq&C\sum_{i=1}^{j}\|\nabla\oBu \|_{L^2(\breve{\Omega}_t^i )}^2(f(A_{i-1})\cdot f^{-1}(A_{i-1})+f^{\frac12}(A_{i-1})f^{-\frac12}(A_{i-1})) \\
	&+C\sum_{i=j+1}^{N(t)}\|\nabla\oBu \|_{L^2(\breve{\Omega}_t^i )}^2(f(A_i)\cdot f^{-1}(A_i)+f^{\frac12}(A_i)f^{-\frac12}(A_i)) \\
	\leq&C_{32}\sum_{i=1}^{N(t)}\|\nabla\oBu \|_{L^2(\breve{\Omega}_t^i )}^2 \\
	=&C_{32}\int_{\breve{\Omega}_{t}} |\nabla\oBu |^2 \, dx,
\end{aligned}
\end{equation}
where the constant $C_{32}$ goes to zero as $\Phi\to 0$. Hence there exists a $\Phi_2>0$, such that for any $\Phi\in [0,\Phi_2)$, one has
\begin{equation}\label{4-33}
\int_{\breve{\Omega}_{t}}\hat{\zeta}\oBu \cdot\nabla \oBu \cdot \Bu\,dx \leq \frac{\mfc}{2}\int_{\hat{\Omega}_t}\hat{\zeta}|\nabla\oBu |^2\,dx.
\end{equation}

On the other hand, using Lemmas \ref{lemmaA1} and \ref{lemmaA2} yields
\begin{equation}\label{4-34}
\begin{aligned}
	&\left|\int_{\hat{ E}}\partial_{x_1}\hat{\zeta}\partial_{x_1}\oBu\cdot \oBu\,dx\right|+\int_{\hat{ E}^\pm}\left[(\oBu \cdot \Bu)\overline{u}_1+\frac12|\oBu |^2(u_1+\overline{u}_1) \right] \partial_{x_1}\hat{\zeta}\,dx\\
	&\quad 
	+ \int_{\hat{ E}^\pm} \hat{\zeta}  \oBu\cdot\nabla \oBu\cdot\Bu \, dx \\
	\leq&C[f(h(\pm t))]^{-1}\left[ \|\nabla\oBu \|_{L^2(\hat{ E}^\pm)}\|\oBu \|_{L^2(\hat{ E}^\pm)}+\|\oBu \|_{L^4(\hat{ E}^\pm)}^2 \left(\|\Bu\|_{L^2(\hat{ E}^\pm)}+\|\oBu\|_{L^2(\hat{ E}^\pm)} \right)\right]\\
	&+\beta^* \|\oBu\|_{L^4(\hat{ E}^\pm)}\|\nabla \oBu\|_{L^2(\hat{ E}^\pm)}\|\Bu\|_{L^4(\hat{ E}^\pm)}\\
	\leq&C[f(h(\pm t))]^{-1}\left(M_1(\hat{ E}^\pm)\|\nabla\oBu \|_{L^2(\hat{ E}^\pm)}^2+M_4^2(\hat{ E}^\pm)\|\nabla \oBu \|_{L^2(\hat{ E}^\pm)}^2\|\Bu\|_{L^2(\hat{ E}^\pm)}\right)\\
	&+C[f(h(\pm t))]^{-1}M_1(\hat{ E}^\pm)M_4^2(\hat{ E}^\pm)\|\nabla\oBu\|_{L^2(\hat{ E}^\pm)}^3+\beta^* M_4(\hat{ E}^\pm)\|\nabla \oBu\|_{L^2(\hat{ E}^\pm)}^2\|\Bu\|_{L^4(\hat{ E}^\pm)}.
\end{aligned}
\end{equation}

Similar to \eqref{4-16}, one can estimate the term
 $\int_{\hat{ E}^\pm}p\overline{u}_1\partial_{x_1}\zeta\,dx$. More precisely,
\begin{equation}\label{4-35}
\begin{aligned}
	&\left|\int_{\hat{ E}^\pm}p\overline{u}_1\partial_{x_1}\hat{\zeta}\,dx\right|=\left|\int_{\hat{ E}^\pm}\partial_{x_1}\hat{\zeta} p{\rm div}\,\Ba\,dx\right|=[f(h(\pm t))]^{-1}\left|\int_{\hat{ E}^\pm} p{\rm div}\,\Ba\,dx\right|\\
	=&[f(h(\pm t))]^{-1}\left|\int_{\hat{ E}^\pm} 2\BD(\overline {\Bu}):\BD(\Ba)+(\oBu \cdot \nabla \Bu +(\Bu+\oBu)\cdot \nabla \oBu )\cdot\Ba\,dx\right|\\
	=&[f(h(\pm t))]^{-1}\left|\int_{\hat{ E}^\pm} 2\BD(\overline {\Bu}):\BD(\Ba)-\oBu \cdot \nabla \Ba\cdot\Bu -(\Bu+\oBu)\cdot \nabla \Ba \cdot \oBu\,dx\right|\\
	\le &C[f(h(\pm t))]^{-1}\|\nabla\Ba\|_{L^2(\hat{ E}^\pm)}\left(\|\nabla \overline {\Bu}\|_{L^2(\hat{ E}^\pm)}+\|\overline {\Bu}\|_{L^4(\hat{ E}^\pm)}\|\Bu\|_{L^4(\hat{ E}^\pm)}+\|\overline {\Bu}\|_{L^4(\hat{ E}^\pm)}^2\right)\\
	\le &C[f(h(\pm t))]^{-1}M_1(\hat{ E}^\pm)M_5(\hat{ E}^\pm)
\left(\|\nabla \overline {\Bu}\|_{L^2(\hat{ E}^\pm)}^2+M_4(\hat{ E}^\pm)\|\nabla\overline {\Bu}\|_{L^2(\hat{ E}^\pm)}^2\|\Bu\|_{L^4(\hat{ E}^\pm)}\right.\\
&\left.+M_4^2(\hat{ E}^\pm)\|\nabla \oBu\|_{L^2(\hat{ E}^\pm)}^3\right),
\end{aligned}
\end{equation}
where $\Ba\in H_0^1(\hat{E}^\pm)$  satisfies
\[
\text{div}~\Ba= \overline{u}_1\quad \text{in}\,\,\hat{E}^{\pm}
\]
and 
\begin{equation*}
\|\nabla\Ba\|_{L^2(\hat{E}^\pm)} \leq M_5(\hat{ E}^\pm)\|\overline{u}_1\|_{L^2(\hat{E}^\pm)}.
\end{equation*}

Note that for the subdomain $\hat{ E}^\pm$,  $M_5(E^\pm)$ is a uniform constant and  the constants $M_1(\hat{E}^\pm),M_4(\hat{E}^\pm)$ appeared in \eqref{4-34}-\eqref{4-35} satisfy the following estimates,
\begin{equation}\label{4-36}
 C^{-1} f(h(\pm t)) \leq M_1(\hat{ E}^\pm)\leq C f(h(\pm t)),  \ \  C^{-1} [f(h(\pm t))]^\frac12\leq M_4(\hat{ E}^\pm)\leq C [f(h(\pm t))]^\frac12.
\end{equation}
Moreover, according to Lemmas \ref{lemmaA1}-\ref{lemmaA2}, and \ref{lemma1}, and Propositions \ref{decay rate-right}-\ref{decay rate-left}, one has
\begin{equation}\label{4-37}
\begin{aligned}
	\|\Bu\|_{L^4(\hat{ E}^\pm)}\leq &\|\Bv\|_{L^4(\hat{ E}^\pm)}+\|\Bg\|_{L^4(\hat{ E}^\pm)}
	\leq M_4(\hat{ E}^\pm)\|\nabla\Bv\|_{L^2(\hat{ E}^\pm)}+\|\Bg\|_{L^4(\hat{ E}^\pm)}\\
	\leq& C[f(h(\pm t))]^{-\frac12}
\end{aligned}
\end{equation}
and
\begin{equation}\label{4-38}
\|\Bu\|_{L^2(\hat{ E}^\pm)}\leq \|\Bv\|_{L^2(\hat{ E}^\pm)}+\|\Bg\|_{L^2(\hat{ E}^\pm)}\leq M_1(\hat{ E}^\pm)\|\nabla\Bv\|_{L^2(\hat{ E}^\pm)}+\|\Bg\|_{L^2(\hat{ E}^\pm)}\leq C.
\end{equation}

{\em Step 3. Growth estimate.} Let
\begin{equation*}
\hat{y}(t)=\int_{\Omega}\hat{\zeta }|\nabla \oBu |^2\,dx+\alpha\int_{ \partial\Omega}\hat{\zeta} | \oBu |^2\,ds.
\end{equation*}
Combining \eqref{4-31}-\eqref{4-38} gives the differential inequality
\begin{equation*}
\hat{y}\leq C\left[\hat{y}'+(\hat{y}')^\frac32\right].
\end{equation*}
It follows from Lemma \ref{lemmaA4} that one has either $\oBu=0$ or
\begin{equation*}
\liminf_{t\rightarrow + \infty} t^{-3}\hat{y}(t)>0.
\end{equation*}
Hence the proof of the proposition is completed.
\end{proof}
Combining Propositions \ref{unbounded exits-1}, \ref{unbounded exits-2}, and \ref{unique-unbounded} together finishes the proof of Theorem \ref{unbounded channel}.


\medskip

{\bf Acknowledgement.}
This work is financially supported by the National Key R\&D Program of China, Project Number 2020YFA0712000.
The research of Wang was partially supported by NSFC grant 12171349. The research of  Xie was partially supported by  NSFC grant 11971307, and Natural Science Foundation of Shanghai 21ZR1433300, Program of Shanghai Academic Research Leader 22XD1421400. The authors thank Professors Changfeng Gui and Congming Li for helpful discussions.

\end{document}